\documentclass[11pt]{amsart}

\usepackage{geometry}
\usepackage{amsmath, amssymb, amsthm, latexsym, mathrsfs, mathtools}
\usepackage{caption, subcaption, paralist}
\usepackage[normalem]{ulem}
\usepackage{url}
\usepackage{xcolor}

\usepackage[backend=biber, style=numeric, giveninits]{biblatex}
\usepackage{hyperref}
\hypersetup{
    colorlinks=true, 
    linktoc=all,     
    linkcolor=blue,  
}

\usepackage{tikz, tikz-cd}
\usepackage[all]{xy}
\usetikzlibrary{calc, decorations.pathmorphing}

\tikzset{curve/.style={settings={#1},to path={(\tikztostart)
    .. controls ($(\tikztostart)!\pv{pos}!(\tikztotarget)!\pv{height}!270:(\tikztotarget)$)
    and ($(\tikztostart)!1-\pv{pos}!(\tikztotarget)!\pv{height}!270:(\tikztotarget)$)
    .. (\tikztotarget)\tikztonodes}},
    settings/.code={\tikzset{quiver/.cd,#1}
        \def\pv##1{\pgfkeysvalueof{/tikz/quiver/##1}}},
    quiver/.cd,pos/.initial=0.35,height/.initial=0}

\tikzset{tail reversed/.code={\pgfsetarrowsstart{tikzcd to}}}
\tikzset{2tail/.code={\pgfsetarrowsstart{Implies[reversed]}}}
\tikzset{2tail reversed/.code={\pgfsetarrowsstart{Implies}}}
\tikzset{no body/.style={/tikz/dash pattern=on 0 off 1mm}}

\numberwithin{equation}{section}
\makeatletter
\renewcommand\subsubsection{\@secnumfont}{\bfseries}%
\renewcommand\subsubsection{\@startsection{subsubsection}{3}
  \z@{.5\linespacing\@plus.7\linespacing}{-.5em}%
  {\normalfont\bfseries}}
\makeatother

\theoremstyle{plain}
\newtheorem{thm}{Theorem}[section]
\newtheorem{lem}[thm]{Lemma}
\newtheorem{prop}[thm]{Proposition}
\newtheorem{cor}[thm]{Corollary}

\theoremstyle{definition}
\newtheorem{defi}[thm]{Definition}
\newtheorem{rmk}[thm]{Remark}
\newtheorem{conj}[thm]{Conjecture}

\newcommand{\thmref}[1]{Theorem~\ref{#1}}
\newcommand{\lemref}[1]{Lemma~\ref{#1}}
\newcommand{\propref}[1]{Proposition~\ref{#1}}

\newcommand{\rmkref}[1]{Remark~\ref{#1}}

\newcommand{\conjref}[1]{Conjecture~\ref{#1}}

\newcommand*{\complex}{\mathbf{C}}
\newcommand*{\Q}{\mathbf{Q}}
\newcommand{\z}{\mf Z}
\newcommand{\C}{\mf C}
\newcommand{\h}{\mf H}
\newcommand{\htwo}{\mf H_2}

\newcommand{\mbb}{\mathbb}

\newcommand{\mf}{\mathbf}
\newcommand{\mc}{\mathcal}

\newcommand{\mrm}{\mathrm}

\newcommand{\im}{\mrm{Im}}
\newcommand{\tr}{\mathrm{tr}}

\newcommand*{\abs}[1]{\left\lvert#1\right\rvert}
\newcommand*{\norm}[1]{\left\lVert#1\right\rVert}
\newcommand{\lan}{\langle }
\newcommand{\ran}{\rangle}

\newcommand{\tp}[1]{#1^t}
\newcommand{\lbr}{\left(}
\newcommand{\rbr}{\right)}
\newcommand{\q}{\quad}

\newcommand{\n}{\nonumber}

\newcommand{\sumn}{\sum \nolimits}

\newcommand{\psmb}{\left( \begin{smallmatrix}}
\newcommand{\psme}{ \end{smallmatrix} \right) }
\newcommand{\smat}[4]{\left( \begin{smallmatrix}#1&#2\\#3&#4\end{smallmatrix}\right)}

\newcommand{\pmat}[4]{\begin{pmatrix} #1 & #2 \\ #3 & #4 \\ \end{pmatrix} }

\newcommand{\sltwo}{\mrm{SL}_2(\mf Z)}
\newcommand{\sltwor}{\mrm{SL}_2(\mf R)}

\newcommand{\GL}{\mrm{GL}}
\newcommand{\spn}{\mrm{Sp}_n(\mf Z)}

\newcommand{\sptwo}{\mrm{Sp}_2( \mf Z)}

\newcommand{\Gn}{\Gamma_0^{(2)}(N)}
\newcommand{\Gp}{\Gamma_0^{(2)}(p)}
\newcommand{\gp}{\Gamma^{(2)}_0(p)}

\newcommand{\fp}{\mc F_2(p)}
\newcommand{\pfp}{\mf F_p}

\newcommand{\sktwon}{S_k^{(2)}(N)}
\newcommand{\stwon}{S^{(2)}_k(N)}
\newcommand{\sktwonnew}{S_k^{(2),new}(N)}
\newcommand{\sktwonold}{S_k^{(2),old}(N)}

\newcommand{\jkmc}{J_{k,m}^{cusp}}

\newcommand{\jkp}{J_{k,1}^{cusp}(p)}
\newcommand{\jkn}{J^{cusp}_{k,1}(N)}
\newcommand{\skk}{\mrm{SK}_k}
\newcommand{\skkn}{\mrm{SK}_k(N)}
\newcommand{\skkp}{\mrm{SK}_k(p)}
\newcommand{\skkm}{\mrm{SK}_k(M)}

\newcommand{\bkn}{\mathbb B_{N}(Z)}

\newcommand{\finv}{\det(Y)^{k/2} |F(Z)|}

\newcommand{\tz}{{\mrm t}_Z}
\newcommand{\tw}{{\mrm t}_W}
\newcommand{\my}{\mf y}

\author{Pramath Anamby }
\address{School of Arts and Sciences\\ 
Ahmedabad University\\ 
Ahmedabad -- 380052, India.}
\email{pramath.anamby@gmail.com, pramath.anamby@ahduni.edu.in}

\author{Soumya Das}
\address{Department of Mathematics\\ 
Indian Institute of Science\\ 
Bengaluru -- 560012, India.}
\email{soumya@iisc.ac.in}

\date{}

\begin{document}
\title[Classical SK lifts and their sizes]{New and old Saito-Kurokawa lifts classically via $L^2$ norms and bounds on their sup-norms: level aspect}

\begin{abstract}
In the first half of the paper, we lay down a classical approach to the study of Saito-Kurokawa (SK) lifts of (Hecke congruence) square-free level, including the allied new-oldform theory. Our treatment of this relies on a novel idea of computing ranks of certain matrices whose entries are $L^2$-norms of eigenforms. For computing the $L^2$ norms we work with the Hecke algebra of $\mrm{GSp}(2)$.

In the second half, we formulate precise conjectures on the $L^\infty$ size of the space of SK lifts of square-free level, measured by the supremum of its Bergman kernel, and prove bounds towards them using the results from the first half. Here we rely on counting points on lattices, and on the geometric side of the Bergman kernels of spaces of Jacobi forms underlying the SK lifts. Along the way, we prove a non-trivial bound for the sup-norm of a Jacobi newform of square-free level and also discuss about their size on average.
\end{abstract}
\maketitle
\section{Introduction}
One of the primary objectives of this paper is to initiate 
a study of the classical sup-norm problem in the context of holomorphic Siegel cusp forms of degree $2$ (denoted as $S_k^2(N)$) in the level aspect.
 The sup-norm problem is in a nascent stage for higher rank groups -- especially for the symplectic groups, which are the objects of interest here -- and only recently some results have been obtained: viz. Das-Krishna \cite{sd-hk}, Kramer-Mandal \cite{kr-sup}, Saha et al \cite{saha-deg2} in the scalar weight aspect, or the restriction-to-compact setting in the Laplace eigenvalue aspect by Blomer-Pohl \cite{blo-pohl}. It is well known that automorphic forms on higher rank groups, which do not have global Whittaker models, like the case of holomorphic Siegel modular forms, pose significant obstacles to analytical questions. There are two distinct aspects of the problem: 
 
-- one is to obtain the correct $L^\infty$-size of the space (referred to as `size' henceforth) measured by the supremum of its Bergman kernel (BK for short); and 
 
-- second, to obtain a good bound on the size of a single eigenform. We briefly discuss these topics below in the current context.

As an initial investigation, in this paper we focus our attention on the subspace of Saito-Kurokawa lifts (SK lifts in short) of (Hecke congruence) level $N$ which is a subspace of $S^{(2)}_k(N)$ obtained as functorial lifting from the space of holomorphic elliptic cusp forms, denoted by $S_k(N)$. We denote the space of the SK lifts by $\skkn$.
See the works of R. Schmidt \cite{schmidt-SKlift} for a representation-theoretic treatment where they arise via CAP representations, and that of Eichler-Zagier \cite{EZ} and T. Ibukiyama \cite{Ibu-SK} for a classical construction. Both approaches will play a role in this paper, but the classical setting will be the predominant one. Even though it is helpful that SK lifts are related to $\mrm{GL}(2)$
 objects; but in the context of sup-norm questions, some of the standard tools that are needed (viz. the geometric side of the Bergman kernel), are not available at our disposal in a useful way. In this way, the problem is nicely balanced. 

Also as far as we are aware, nothing is known about the correct size of $\skkn$. Our first task in this direction is then to formulate a conjecture about this and test it over certain subspaces. We easily get a global lower bound (viz. $N^{-2}$), and can check that at least the same appears as an upper bound for the oldspace. Even this requires some work, see sections~\ref{old-cont-02} and \ref{old-cont-1}. 

With the above remarks in view, the primary objectives of this paper are twofold:
\begin{enumerate}
    \item[(i)] To describe the new-oldform theory in the space of SK lifts in the classical language for square-free levels, which is interesting in its own right, and
    \item[(ii)] Use this description to obtain bounds towards the sizes of various objects related to SK lifts.
\end{enumerate}

\textit{Throughout this paper, we call $F$ to be a Hecke eigenform if it is an eigenfunction of all Hecke operators $T_S(n)$ with $(n,N)=1$. The same convention will apply to all relevant spaces of modular forms.}

\subsection{Old and new SK lifts vs. old and new EZI lifts}
Let  $k$ be even. We define the space of SK lifts to be the span of those Hecke eigenforms $F \in \sktwon$ for whose away-from-level spinor $L$-function factorizes as $\displaystyle L^N(F,s)= \zeta^N(s+1/2)\zeta^N(s-1/2)L^N(f,s)$. We refer the reader to Section~\ref{sec:SKintro} for more details. New and oldforms therein are just the new and oldforms in $\stwon$ which are SK lifts. The newforms in $\sktwon$ are as described in \cite{schmidt-AL}.
The classical Hecke equivariant liftings of Eichler-Zagier-Ibukiyama (EZI in short, see e.g., Figure~1 and subsection~\ref{EZI Lifts} for more details) provide a convenient description of most of the SK lifts, but unfortunately do not account for all of the oldforms (cf. Figure~2; subsection \ref{MaassDist}). We make some efforts to clarify the picture below in the classical setting.
As far as we are aware, the classical theory of SK lifts for higher levels is not written down in the literature.

The first part of the paper is devoted to the treatment of (i). To summarize: we will show that the EZI lifts of newforms (from either elliptic or Jacobi forms) account for all the newforms in the SK space; however not so for all the oldforms. Here corresponding to each Jacobi form $\phi$ of level $N$ there are three oldforms in level $Np$ for $p \nmid N$, one of which is not an EZI lift.\footnote{It is an open problem to construct this `exceptional' oldform (e.g., to construct $F|W_p$, where $F$ SK lift of level $1$), as a lifting from a suitable space of Jacobi forms -- in the example, perhaps those come from Jacobi forms on the lattice $p\z \times\z$.} These may be visualized from the diagrams given below, where $\mc S$, $\mrm{EZ}$, $\mrm{EZI}$ denote the Shintani, Eichler-Zagier, Eichler-Zagier-Ibukiyama maps respectively.

\begin{figure}[!htbp]
\centering
\begin{tikzcd}
f && h & \phi && {F_f}
	\arrow["{\mathrm{EZ}}"', maps to, from=1-3, to=1-4]
	\arrow["{\mathrm{EZI}}"', maps to, from=1-4, to=1-6]
	\arrow["\mc S"', maps to, from=1-1, to=1-3]
	\arrow["{\mathrm{SK}}", curve={height=-24pt}, dashed, maps to, from=1-1, to=1-6]
\end{tikzcd}
\caption{ {\small The SK and EZI maps on the newspace with $f\in S_{2k-2}^{new}(N)$.}}
\end{figure}
\begin{figure}[!htbp]
  \begin{tikzcd}
	&  f & h & \phi & {F_f} & {F_f|W_p} \\
	f \\
	& {f|U(p)} & {h|U(p^2)} & {\phi|U_J(p)} && {F_f|U_S(p)}
	\arrow[maps to, from=2-1, to=3-2]
	\arrow["{\mathcal{S}}"', maps to, from=3-2, to=3-3]
	\arrow["{\mathrm{EZ}}", maps to, from=1-3, to=1-4]
	\arrow["{\mathrm{EZ}}"', maps to, from=3-3, to=3-4]
	\arrow[maps to, from=2-1, to=1-2]
	\arrow["{\mathcal{S}}", maps to, from=1-2, to=1-3]
	\arrow["{\mathrm{EZI}}", maps to,from=1-4, to=1-5]
	\arrow["{W_p}", maps to,from=1-5, to=1-6]
	\arrow["{\mathrm{EZI}}"', maps to, from=3-4, to=3-6]
	\arrow["{\mathrm{SK}}"{description}, curve={height=24pt}, dashed, maps to, from=1-2, to=1-5]
	\arrow["{\mathrm{SK~ but ~not~ EZI}}"{description}, curve={height=16pt}, dashed, maps to, from=2-1, to=1-6]
	\arrow["{\mathrm{SK}}"{description}, curve={height=-28pt}, dashed, maps to, from=3-2, to=3-6]
\end{tikzcd}
\caption{{\small The SK and EZI maps on the oldspace with $f \in S_{2k-2}(N/p)$.}}
\end{figure}
As discussed in  \cite{schmidt-SKlift}, there are four possible sources of oldforms in level $Np$ in the old-class of any cusp form $F$ of level $N$ where $(p,N)=1$ for the space $\sktwon$, which reduces to three for the space of SK lifts. This feature is nicely captured in \thmref{skchar-intro}~(3) given below. In fact, if $k \ge 3$, there are three or four oldforms in the old class of $F$ according as $F$ is an SK lift or not! This gives a new characterization of SK lifts (cf. \cite{saha-pitale, farmer2013survey}).
We will also produce an explicit orthonormal basis for the old and new spaces for the SK lifts in question. Even though the above information is probably in principle available from the representation theoretic context; we believe that the classical treatment is not, even though desirable\footnote{Ibukiyama, personal communication}; and more importantly, such a description is crucial to us. 

\subsection{Maa{\ss}-Ibukiyama relations characterize EZI lifts}
From Figure~2, it is evident that not all SK lifts are EZI lifts. This leads us to wonder if one can characterize the EZI lifts inside the space of SK lifts uniquely. It is quite pleasant that the answer turns out to be what we call the Maa{\ss}-Ibukiyama relations. In order to characterize SK lifts of level one, Maa{\ss}--relations were first considered by H. Maa{\ss} in his proof of the Maa{\ss} conjectures (see \cite{maass1978lineare}, \cite{maass1}). In our setting, we use a variant of it in higher levels, which are defined as follows. Let $G \in \stwon$. Then $G$ satisfies the Maa{\ss}-Ibukiyama relations if for all $T= \psmb n & r/2 \\ r/2 & m \psme$ (cf. \cite[\S 3.4]{Ibu-SK})
\begin{equation}\label{Maassrel}
       A_G(T):= A_G(n,r,m)= \sum_{d|(n,r,m), (d,N)=1} d^{k-1} A_G(nm/d^2, r/d,1).
    \end{equation}
Then in \thmref{ezi-charac}, we show that $G$ is an EZI lift if and only if it satisfies \eqref{Maassrel}. The proof depends on the description of oldforms discussed above. Let us mention here the related work of B. Heim \cite{heim2017maass}, which is useful for us. 

We also point out here that the Maa{\ss}-Ibukiyama relations considered above are not the `Maa{\ss}-relations' that one would expect from representation theory in higher levels, see the work of Marzec  \cite{marzec2021maass} in this regard. We would discuss this in a bit more detail in Section~\ref{intr-maass}.

\subsection{\texorpdfstring{The matrix $M_p$ that controls SK lifts}{The matrix Mp that controls SK lifts}}\label{Mpintro}
We now say a few words about our treatment of (i). Usually, the classical treatment of new-oldform theory, say for elliptic modular forms, goes via their Fourier expansion \'a la Atkin-Lehner \cite{AL}. However in the present situation, especially for the oldspace, this entails knowledge of Fourier expansions at various cusps. For instance, one has to look at the Fourier expansion of $F| U_S(p)W_p$ with $F$ being of level $N$ and $p \nmid N$ -- which may not be easy to obtain. See \rmkref{mp-not-fe} for more discussion on this point. To get around this, we consider instead an approach based solely on $L^2$ norms -- this forms the backbone of the part~(i) discussed above. Thus this method works only for cusp forms, but for the Eisenstein series, we believe a direct Fourier expansion-based approach should work. To describe this idea, we need a bit of notation.

\textbf{ \textsl{Throughout this paper $k \ge 2$ is even and all of our Petersson norms are not normalized by volume (cf. \eqref{pet-def}).}}

As will be discussed in detail in Section~\ref{sec:SKintro}, there are four types of operators, denoted as $R_i=R_i(p)$, $i=1,\ldots, 4$ with $R_1=\mrm{Id.}, R_2=U_S(p), R_3=W_p, R_4=U_S(p)W_p$, via which oldforms arise in level $Np$ from level $N$, where $p\nmid N$ -- see subsection~\ref{skold}. Put $F_{j,p} := F|R_j(p)$.
With this data, we associate a matrix $M_p$, (which we call the `inner-product' matrix) as follows.
\begin{defi}
     \begin{equation}\label{Mpdef}
       M_p=M_p(F):=\left(\lan F_{i,p}, F_{j,p} \ran_{Np} \right)_{1\le i, j\le 4}.
    \end{equation}
\end{defi}
In fact, the new-oldform theory would follow from the following characterization of the space of SK lifts in terms of the $4 \times 4$ matrix $M_p$, which therefore plays a central role. We now mention a curtailed version of one of our main theorems in this connection.

\begin{thm}\label{skchar-intro}
    Let $k \ge 3$, $N \ge 1$ be square-free and $F \in \stwon$ be a eigenform with eigenvalues $\lambda_F(n)$ for $(n,N)=1$. Then the following are equivalent.
    \begin{enumerate}
        \item $F$ is an SK lift of a elliptic eigenform of level $N$ and weight $2k-2$.
        \item The inner product matrix $M_p$ has rank $3$ for some prime $p \ge 17$ with $p\nmid N$.
        \item There exists a prime $p\nmid N$, such that $\displaystyle
            \lambda_F(p^2)=\lambda_F(p)^2-(p^{k-1}+p^{k-2})\lambda_F(p)+p^{2k-2}$.
    \end{enumerate}
\end{thm}
Let us also mention here that \thmref{skchar-intro}~(3) was known (for all $p$) in level one from \cite{farmer2013survey}, \cite{saha-pitale},  whereas we mainly work with (2). From \thmref{skchar-intro} we can obtain an explicit orthonormal basis for the oldspace since we can show that the leading $3 \times 3$ minor of $M_p$ is non-zero. Thus $\{F_1,F_2,F_3 \}$ is a basis for the old-class of $F$. 
This then enables us to study its $L^\infty$ size. 
Further, one needs to suitably choose a $W_p$ invariant orthonormal basis from this set.
The $W_p$ invariance helps in localizing the study of the sup-norm problem in suitable regions. Also, as a byproduct of \thmref{skchar-intro}, we have the following corollary for non-SK lifts, and this is in conformity with the investigations in \cite{schmidt-AL}.

\begin{cor}
An eigenform $F\in \stwon$ is a non-SK lift if and only if the matrix $M_p(F)$ has rank $4$, for $p \ge 17$ with $p \nmid N$. In other words, the oldclass of an eigenform $F$  is $4$ dimensional -- and is spanned by the $F_j$, $j=1,\ldots,4$ -- if and only if it is not an SK lift. 
\end{cor}
The condition $p \ge 17$ most likely can be removed, but at the moment we cannot see it immediately. It arises in course of working out some inequalities involving the Ramanujan-Petersson conjecture, cf. Remark~\ref{p17}.

\subsection{Computations in the Hecke algebra}
In order to obtain \thmref{skchar-intro} for the treatment of new-oldform theory, we need extensive information about various inner product relations. This is the content of Section~\ref{norm-comput} where we express the entries $\displaystyle \lan F_{i,p}, F_{j,p} \ran_{Np}$ of $M_p$ as 
\begin{align} \label{norm-reln-intro}
 \lan F_{i,p}, F_{j,p} \ran_{Np}= P(\lambda_F(p),\lambda_F(p^2)) \cdot  \lan F, F \ran_{Np},
\end{align} 
 where $P$ is a polynomial in $\Q[x,y]$. This leads us to the discussion about the several possible methods that could be used to obtain such relations. A by-now standard method (which was our initial approach) is to use the Kohnen-Skoruppa Dirichlet series 
 \[ \displaystyle D(F,G,s) = \sumn_{m \ge 1} \lan \phi_{F,m}, \phi_{G,m} \ran m^{-s},\]
 where $\phi_{F,m}$ denotes the $m$-th Fourier-Jacobi coefficient of $F$ and $\lan \phi_{F,m}, \phi_{G,m} \ran$ denotes the Petersson inner product of the Jacobi cusp forms $\phi_{F,m}, \psi_{F,m}$. The idea is to compare the residue of this series at $s=k$ in two ways.
 Indeed, for a newform $F$, this is the approach taken in \cite{Ag-Br}. However, there seem to be some inaccuracies in this paper -- this will be discussed in some detail in Section~\ref{skN-conj}. Most importantly, we just mention here that the method loc. cit. works \textit{only if} at least one of the components is a newform, and \textit{not} otherwise. Since we also have to work with oldforms, this distinction plays an important role, and we have to take a different route to avoid this.

Namely, we work with the operators $R_j$ introduced above and related objects viewed inside the Hecke algebra over $\Q$ for $\Gamma^{(2)}_0(N)$ and thus most of our calculations here are algebraic. Even though we sometimes have used a computer for some of the computations in this part (codes available); we feel it is quite remarkable that the Hecke relations in \thmref{skchar-intro}~(3) precisely show up in the factorization of the intricate expression of $\det (M_p)$ etc. In a sense, this also serves as a cross-check for all of our inner product calculations.
Some of the calculations in the Appendix to \cite{blo-pohl} are quite useful to us here. 

As final remarks, let us mention that our method can be adapted to study various problems about $\sktwon$; for instance, it seems likely that one can reprove B\"ocherer's result that $U_S(p)$ is injective on $\sktwon$ for all $p|N $ ($N$ square-free, cf. \cite{boech-Up}) by this method. Moreover, this approach is expected to work for other spaces, e.g., Hermitian modular forms, etc.

\subsection{Some background on the sup-norm problem in the level aspect}
The second part of the paper depends on (i) and addresses the sup-norm questions for various spaces related to $\skkn$ mentioned in (ii) above. We begin with a brief discussion of the problem in degree $1$ on average.

Let $S_k(N)$ denote the space of holomorphic cusp forms of weight $k$ and level $\Gamma_0(N)$. Denote by $\mc B_{k}(N)$ an orthonormal basis for $S_k(N)$. Then the Bergman kernel for the space $S_k(N)$ is given by
\begin{equation}
    \mbb B_{k,N}(\tau)    := \sumn_{f\in \mc B_{k}(N)} v^k |f(\tau)|^2.
\end{equation}
The size of the space $S_k(N)$ is then measured by
\begin{align}
    \sup(S_k(N)) := \sup \nolimits_{\tau \in \h}  \mbb B_{k,N}(\tau).
\end{align}
From the work of Michel-Ullmo \cite{michel1998points} and later by Kramer-Jorgensson \cite{jorgenson2004bounding} it is known that for $k=2$ and square-free $N$, the bounds $ \displaystyle \sup(S_2(N)) \ll \tau_5(N) \log N$ and $\ll 1$ respectively, hold. Here $\tau_5(N) = \sum_{abcde=n}1$.
Whereas Michel-Ullmo use the spectral Large sieve, Kramer-Jorgensson use the heat kernel method to arrive at the above bounds. Steiner has shown for $k>2$ the hybrid result $ \displaystyle \sup(S_k(N)) \ll_A 1+\nu_N \, k^{-A}$ for any $A>0$. Here $\nu_N \ll N \log N$ denotes the index of $\Gamma_0(N)$ in $\sltwo$. From this, it is not clear whether one gets a bound $O_k(1)$ (which is expected), as the implied constant depends on $A$ and $k$ is fixed. However, we note that Steiner's result is quite general—valid for all real weights and congruence subgroups.

In an appendix to this paper -- see the Introduction, subsection~\ref{intro-app} and Section~\ref{appendix} -- we will prove in particular the bound $\displaystyle \sup(S_\kappa(N)) \asymp_k 1$ e.g., for all $N$ square-free and for all $\kappa \ge 5/2$, which could be half-integral.

\subsection{Conjectures and results on sup-norms}
In the remaining part of the introduction, we will discuss our results about the sizes of various spaces that are intertwined with SK lifts. First, we discuss about Jacobi forms, whose sizes have a direct bearing on that of SK lifts, apart from their own interest.

\subsection{Size of Jacobi forms of index 1} \label{jac-size-intro}
Let $\jkn$ be the space of Jacobi cusp forms of index $1$ and level $N$. For any $\phi\in\jkn$, let $\tilde{\phi}(\tau, z):=v^{k/2}e^{-2\pi y^2/v}|\phi(\tau, z)|$ be the invariant function related to $\phi$, where we write $\tau=u+iv$ and $z=x+iy$. Let $B^J_k(N)$ denote an orthonormal basis for $\jkn$. We define the size of $\jkn$ as
\begin{align}
    \sup(J_{k,1}^{cusp}(N)):=\sup_{(\tau, z)\in \mf H\times \mf C} \sum_{\phi\in B_k^J(N)}  \widetilde \phi (\tau,z)^2.
\end{align}
It is well-known that such Jacobi forms can be uniquely expressed as
\begin{align} \label{th-dec-intro}
    \phi = h_0(\tau) \theta_0(\tau,z) + h_1(\tau) \theta_1(\tau,z),
\end{align}
called the theta decomposition, and both $h_0,h_1$ are determined by a single $h \in S_{k-1/2}(4N)$ -- in fact the assignment $\phi \leftrightarrow h$ is a Hecke equivariant isomorphism. Here $\theta_j$ are certain Jacobi theta series, see subsection~\ref{theta-decom}.
This provides the link to pass to half-integral weights, but there is one catch. Namely from \eqref{th-dec-intro} one can expect that, say when $v\approx N^{-1}$,
\begin{align} \label{jac-size-eq}
   \sum_{\phi\in B_k^J(N)}  \widetilde \phi (\tau,z)^2\approx  \mbb B_{k-1/2,N}(\tau) \cdot v^{1/2} (|\theta_0|^2 + |\theta_1|^2) \ll 1 \cdot v^{-1/2} \ll   N^{1/2},
\end{align}
since the theta series behave like $v^{-1/2}$ for small $v$. This extra contribution from the theta series is the ``catch" mentioned above, without which one can just get the bound $O_k(1)$ confirming the \conjref{jacobi-conj} below.
Technicalities aside, we indeed prove that the bound in \eqref{jac-size-eq} is true. 
\begin{thm} \label{jk1-sup}
Let $k \ge 4$ be even and $N$ square-free. Then,
\begin{align}
   1 \ll_k  \sup( J^{cusp}_{k,1}(N)) \ll_k N^{1/2}.
   \end{align}
\end{thm}

Given the lower bound $1$ it seems reasonable to expect that the size of $J^{cusp}_{k,1}(N)$ is $1$, but we can not corroborate this at the moment.
\begin{conj} \label{jacobi-conj}
    Let $k \ge 4$ be even. Then for large $N$, one has $\displaystyle \sup( J^{cusp}_{k,1}(N)) \asymp_k 1$.
\end{conj}

\subsection{Non-trivial bound for a Jacobi newform}
However, quite interestingly, we are able to prove a non-trivial bound for the sup-norm of an individual Jacobi form, relying upon the corresponding result for half-integral weights by \cite{Kir}. The `trivial bound' is $O_k(1)$, as follows from \conjref{jacobi-conj}. Alternatively, we can refer to \propref{phi-crude} for the bound $O(N^\epsilon)$, which we also call the trivial bound.
In this case as well, we encounter the ``catch" mentioned around \eqref{jac-size-eq}, and we have to work to overcome it. We prove the following.

\begin{thm}\label{thm:indJacobi}
    Let $N$ be an odd, square-free integer and $\phi\in\jkn$ be an $L^2$-normalized eigenform of all Hecke operators away from the level $N$. Then 
    \begin{equation}
        \sup \nolimits_{(\tau,z) \in \h\times\C}  \widetilde \phi (\tau,z) \ll _{k,\epsilon} N^{-1/36+\epsilon}.
    \end{equation}
\end{thm}
The main point behind the success (see subsection~\ref{key-point}) here is that even though one obtains an `extra' $v^{1/2}$ from $\displaystyle v^k |\phi|^2 \approx v^{k-1/2}|h|^2 \cdot v^{1/2} |\theta|^2$ after reduction to a convenient region (where $v \gg 1/N$), one still has some room to `absorb' the factor $v^{1/2}$ into the bounds for $h$ (via its Fourier expansion) and thus can proceed further. Note that the weight of $h$ is $k-1/2$.
Without this feature, we wouldn't have been able to succeed, since the theta series doesn't decay with $N$ when $v \approx 1/N$. The reader might be curious why the same idea can't be used in the previous subsection: the answer is that the room available for a single $\phi$ is lost when we consider the question on average, i.e., consider the sum of the squares, since this produces a factor $\approx N$ from the number of summands.

The bulk of our calculations are focused on a single $\phi$ in Section~\ref{jacobi-sup-norms}. The average result follows along the same lines. We end this discussion with a conjecture.
\begin{conj}
    Let $k \ge 4$ be even and $\phi\in\jkn$ be an $L^2$-normalized newform. Then for large $N$ and any $\epsilon>0$, one has $\norm{\phi}_\infty \ll_{k,\epsilon} N^{-1/2+\epsilon}$.
\end{conj}

\subsection{Size of spaces of SK lifts}
We now move to the next level of complexity and consider the space of Saito-Kurokawa lifts of elliptic modular forms of level $N$ -- the main objects of interest in this paper. In the weight aspect in level $1$, such a problem was studied by V. Blomer \cite{Bl} under generalized Lindel\"{o}f hypothesis, by relying on the Fourier expansion directly. In \cite{DS} an approach using the Fourier-Jacobi expansion was laid down. The size of $\skk$ was precisely computed in \cite{PASD}: viz., $\displaystyle \sup(\skk) \asymp k^{5/2}$.

Let $B_{k}^{SK}(N)$ denote an orthonormal basis for the space $\skkn$, and put
\begin{align}
    \bkn := \sumn_{F \in B_{k}^{SK}(N)} \det(Y)^k |F(Z)|^2; \q  \sup(\skkn) : = \sup\nolimits_{Z \in \mf H_2} \bkn .
\end{align}
For SK lifts of level $p$, we prove the following result (restriction to prime levels is mainly for convenience and keeping in mind the length of the paper).
\begin{thm}\label{thmsk}
   Let $k > 5$ be even and $p$ be a prime. Then for any $\epsilon>0$,
   \begin{align}
     p^{-2} \ll_k   \sup(\skkp) \ll_{k,\epsilon} p^{-1+\epsilon}.
   \end{align}
\end{thm}
Going by the above lower bound, and the size of the oldspace at level $p$, viz. $\displaystyle \sup(\skkp^{old}) \ll p^{-2}$ (see Sections~\ref{old-cont-02} and \ref{old-cont-1}) -- which surprisingly seems not at all trivial -- leads us to propose the following conjecture about the size of $\skkn$.

\begin{conj} \label{sk-conj}
For all even $k \ge 4$ and $N \ge 1$, one has 
     $ \displaystyle \sup(\skkn) \asymp k^{5/2} N^{-2}$.
\end{conj}
For a newform, we propose the following.
\begin{conj} \label{lev-conj}
 Let $k \ge 4$ be even and $N \ge 1$. For an $L^2$-normalized newform $F \in \skkn$ one has $\norm{F}_\infty \ll_k N^{-3/2}$.
\end{conj}
These conjectures are consistent with the expected size of a Hecke eigenform in $\sktwon$, viz., $\displaystyle \norm{F}_\infty \ll_k N^{-3/2}$, where the latter is inspired from the expected abound $\displaystyle \sup(\sktwon)\asymp_k 1$. To see the consistency mentioned above, note that $\dim \skkn \asymp_k N$.
In a way, testing the global sup-norm over smaller and smaller subspaces, say e.g., the SK lifts as in this paper, lends support towards the plausibility of \conjref{lev-conj}. Probably an asymptotic formula holds in the Conjecture~\ref{sk-conj} above, but it may not be easy to formulate it.

Perhaps it is worth remarking here that in order to arrive at the statement of \conjref{sk-conj}, one invariably has to consider norm relations between, say, a Jacobi newform $\phi$ and its EZI lift $F$. However, various conflicting normalizations in the literature, and in some cases, ambiguous statements on this point (cf. \cite{Ag-Br}, \cite{horie}) make the path misleading. To us, arriving at the statement of \conjref{sk-conj} was as subtle as compared to working on it. We refer the reader to Section~\ref{skN-conj} for a more detailed discussion.

\subsection{Discussion about the proofs} The lower bound in \thmref{thmsk} 
 follows from bounding below the size of the newspace. Here, we first use the norm relation between the SK newforms and the corresponding Jacobi newforms. There are some complications here already (cf. Section~\ref{skN-conj}). In a few words, the norm relations (cf. \eqref{norm-reln-intro}) are much more subtle for the oldspace, and need to be worked out from scratch, see \propref{oldbasis}. In fact, this is one of the other reasons to write Section~\ref{norm-comput} in the first place.

Next, we use Waldspurger's formula to relate the Fourier coefficients to the twisted central $L$-values of the corresponding elliptic newforms. This reduces the problem of obtaining a lower bound for $\sup(\skkp)$ to one such for the average of these central $L$-values: cf. \eqref{avg-Lvalues}. This is already available from \cite[Theorem 1.8]{PASD} and gives us the desired global lower bound $p^{-2}$. This picture most likely generalizes to square-free levels, see \rmkref{sq-free-central}. In fact, all of our results on the sizes of SK lifts should, in principle, generalize to square-free levels at the cost of some more delicate technicalities, and most of the ingredients are already present in this paper. We choose to work with prime levels mainly because the paper is already long and technical.

\subsubsection{The three regions: Type~0,~1,~2} To obtain an upper bound for the size of $\skkp$, it is necessary to choose a suitable fundamental domain for $\Gamma_0^{(2)}(p) \backslash\mf H_2$. This is described in Section~\ref{fund-dom}. This essentially leads to three types of regions, which we call Type~0,1,2 -- and the analysis proceeds differently for each of these, see Table~\ref{tab:placeholder}. Moreover, we estimate the sizes of the newspace and oldspace separately, as we can't fruitfully use the geometric side of the Bergman kernel. Our method requires us to  pass via Jacobi forms -- for this, we need norm relations between $F$ and $\phi$. But due to the non-uniform nature of the norm relations (see Section~\ref{skN-conj}), we again need to handle the oldforms separately.

\subsubsection{Type~0 -- Fourier expansion}
This is the easiest region, where $Y:= \im(Z) \gg 1$, and even the upper bound $O(p^{-2})$ (cf. \eqref{skN-conj}) follows from the Fourier expansion itself. The method used here is the same as that in \cite{PASD}, or \cite{sd-hk}  -- viz. looking at the second moment of Fourier coefficients, which are controlled by Fourier coefficients of certain Poincar\'e series. See Section~\ref{Type0}, which also contains some general results about estimating the Fourier expansion in the level aspect. These may be useful elsewhere as well.

\subsubsection{Type~1 -- Fourier-Jacobi expansion}
This is at the next level of difficulty. Here we resort to the Fourier-Jacobi (FJ) expansion of $F$: namely, we look at 
\begin{equation} \label{fj-intro}
F(Z) = \sumn_m \phi_{F,m}(\tau,z) e(m \tau'),
\end{equation}
where $Z = \psmb \tau & z \\ z & \tau' \psme$. We then transfer the problem to the FJ coefficients $\phi_m=\phi|V_m$ (here $V_m$ is a certain Hecke operator, see subsection~\ref{Jacobiprelim}) and ultimately to one on $\phi_{F,1}$ via the Cauchy-Schwarz inequality -- i.e., to the corresponding question about the size of $\jkp$, which was discussed in the Introduction, subsection~\ref{jac-size-intro}. In this region we obtain the bound $O(p^{-3/2})$ unconditionally and $O(p^{-2})$ assuming \conjref{jacobi-conj}.
We refer the reader to Section~\ref{avgtype1} for more details. 

\subsubsection{Type~2 -- counting matrices}
This is the hardest of the three regions, where we have $Y\gg 1/p$.
It is worth mentioning here that having an old-basis that is invariant under the involution $W_p$ at our disposal is crucial to be able to reduce to this region. This is done in subsection~\ref{wp-inv}. Recall that we bound the oldspace separately.

It is well known that the Fourier expansion is hopeless in this region. Thus, we again fall back to the FJ expansion of $F$. We view the FJ coefficients (cf. \eqref{fj-intro}) akin to Fourier coefficients and start with a Rankin-Selberg type argument: roughly we write, after suitably truncating the $m$ sum in \eqref{fj-intro}, that
\begin{equation}
    |F|^2 \ll (\sumn_m |\phi_{F,m}|^2) (\sumn_m e(- 4 \pi m v') ),
\end{equation}
and then sum over $F$ on both sides and appeal to the geometric side of the Bergman kernel for the sum over $F$ i.e., $\phi$. This leads to a fairly involved counting problem for discrete matrices, which finally gives the bound $O(p^{-1})$ in this region for the newspace. The counting problem can be thought of as the usual one weighted by certain upper triangular matrices: e.g., in we have to count
\begin{equation} \label{count-with-wt-intro}
    \mc C(\tau, m, \delta):= \#\{(\gamma, \alpha_{m}, \beta_{m}): \gamma\in \Gamma_0(N), |A_{\beta_{m}^{-1}\gamma\alpha_{m}}(\tau)|< \delta\}.
\end{equation}
Here $\alpha_m,\beta_m$ are certain upper triangular matrices, $\displaystyle A_g(\tau):= (g\tau-\overline{\tau})j(g,\tau)\im(\tau)^{-1}$, $\im(\tau) \gg p^{-1}$, and $m$ runs over an interval of length $\approx p$, see subsection~\ref{count-N}. It may be worth noting that the counting arguments used here are written for any square-free $N$.

Along the way, we prove various miscellaneous results. For instance we show in \propref{phi-crude} the `trivial bound' $\norm{\phi}_\infty \ll N^{\epsilon} m^{1+\epsilon}$ for any $L^2$-normalized $\phi \in \jkmc$ in the $N$ aspect. Improving the $N$ aspect is probably possible by inserting an  amplifier.
This bound (in fact any polynomial bound) is sufficient for our purposes of truncating a certain Fourier-Jacobi expansion, see subsections~\ref{triv-bd-type2},~\ref{new-type2}. 

\subsubsection{Contribution of oldforms.} As stated earlier, we bound the oldspace separately. This is due to two reasons. First, as mentioned before, the non-uniform nature of the norm relations. Second, the basis of oldspace is not simply an EZI lift of a basis of Jacobi oldforms (see Section \ref{orthbasis}), and thus is not captured fully by the BK of the space of Jacobi forms.

As in the case of newspace, the size of the oldspace is estimated separately in Type 0, 1, and Type 2 regions. In fact, in regions 1 and 2 we have to work with a bit more complicated version of the counting argument while dealing with the oldspace $\skk|U_S(p)$, see sections~\ref{old-cont-02} and ~\ref{old-cont-1}. When $n=1$, estimating the contribution of the subspaces arising due to the $U(p)$ operator is quite straightforward because of a linear Hecke relation between the Hecke operators $T(p)$, $U(p)$ and $B_p$ (see subsection \ref{Upcontn=1}). However, no such useful relation exists for the $U_S(p)$ operator. The other methods such as using the Fourier expansion of $F|U_S(p)$ or using the double coset decomposition of $U_S(p)$ do not seem to yield the desired bounds. Thus, the use of BK of Jacobi forms of level $1$ and the counting argument becomes necessary in this case.

\subsection{\texorpdfstring{Appendix~1: the case $n=1$}{Appendix: the case n=1}} \label{intro-app}

In an Appendix (Section~\ref{appendix}) we provide optimal bounds for the size of $S_k(N)$ for all $k > 2$ and $N$, including $k$ half-integral ($N$ square-free). For this, we give a direct argument relying on Poincar\'e series and mean squares of Fourier coefficients. 
When the level is not square-free some delicate arguments involving various cusps will be used, however.
This is more complicated than the square-free case, and we rely on a lemma of Saha \cite{saha2017sup} that allows us to work in convenient regions, and then we can follow our original trajectory of proof. When $k=2$ and $N$ is square-free, the following is known from \cite{jorgenson2004bounding}, but it seems nothing is known for non-square-free levels.

\begin{thm} \label{skn-int}
    Let $N$ be any integer and $k >2$ be any positive integer. Then
    \begin{equation}
        \sup( S_{k}(N) )\asymp_k 1. 
    \end{equation}
    The same result holds for all principal congruence subgroups and for half-integral weights $\ge 5/2$ and square-free level.
\end{thm}
The implied constant can be shown to be polynomial in $k$.
The main issue for half-integral weights is that the theory of Atkin-Lehner operators is subtle. Such operators aid in localizing the mass of the eigenforms in the orthonormal basis.
Our arguments use no (arithmetic) diagonalization of oldforms (cf. \cite{michel1998points}), nor Deligne's bound on Fourier coefficients. We essentially relate the problem to majorants of suitable Eisenstein series, which are well understood.

\subsection{\texorpdfstring{Appendix~2: The Jacobi $U_J(p)$ operator}{the jacobi operator}} \label{intro-app2}

In the second Appendix, we prove that the Hecke operator $U_J(p)$ at a prime $p|N$ is not self-adjoint on $J^{cusp}_k(N)$, if $N$ is odd and square-free. This shows that the statement of \cite[Proposition~2]{Ag-Br} should be restricted to the newspace. We prove this by
showing that the self-adjointness leads to incompatible sizes of the Satake parameters of certain eigenforms coming from lower levels. 
This is done by certain inner product calculations on the Kohnen's plus space, and may be of independent interest. This statement is required to formulate the conjecture on the average size of the sup-norms (see Section \ref{skN-conj} for more details).

\subsection*{Acknowledgments}
{\small
It is a pleasure for the authors to thank Abhishek Saha and Tomoyoshi Ibukiyama for their comments on the topic of oldforms for SK lifts.
S.D. thanks IISc. Bangalore, UGC Centre for Advanced Studies, DST, India, for financial support.

P.A. held an NBHM Postdoctoral Fellowship at IISER, Pune, during the preparation of this article and thanks the National Board for Higher Mathematics (NBHM), DAE, India, for the financial support provided through this fellowship. He also thanks Ahmedabad University for providing excellent infrastructure and research support. }

\section{Notations and preliminaries}
In this paper, we will mostly use standard notation, some of which are collected below, and the rest will be introduced as and when it is necessary. We note that $A \ll B$ and $A=O(B)$ are the Vinogradov and Landau notations, respectively. By $A\asymp B$ we mean that there exists a constant $c\ge 1$ such that $B/c\le A\le cB$. Any subscripts under them (e.g., $A\ll_n B$) indicate the dependence of any implicit constants on those parameters. Throughout, $\epsilon$ will denote a small positive number, which can vary from one line to another.

Let $\z, \Q, \mf R$, and $\complex$ denote the integers,
rationals, reals, and complex numbers, respectively. For a commutative ring $R$ with unit, $\mrm M_{n,m}(R)$ and $\mrm M_n( R)$ denote the set of $n \times m$  and $n\times n$ matrices over $R$, respectively. $\mrm{GL}_n (R)$ denotes the group of invertible elements in $\mrm M_n(R)$. $\mrm{Sym}_n(R)$ denotes the set of all $n\times n$ symmetric matrices over $R$. $\mrm {Sp}_n(R)$ denotes the symplectic group of degree $n$ over $R$. $\Lambda_n$ (resp. $\Lambda_n^+$) denotes the set of all $n\times n$ real symmetric, positive semi--definite (resp. positive--definite), half--integral matrices  (that is $T=(t_{ij})$ with  $ 2t_{ij},\;t_{ii}\in \z$). We will denote the transpose of $A$ by $\tp{A}$. For matrices $A$ and $B$ of appropriate size, we write $A[B]:=BAB^t$. Moreover, $1_n$ and $0_n$ will be the $n\times n$ identity and zero matrices, respectively. We drop the subscripts when the order of these matrices is obvious. For standard facts about Siegel modular forms and Jacobi forms, we refer the reader to \cite{Fr}, \cite{klingen}, \cite{EZ}.

\subsection{Jacobi forms of higher level.}\label{Jacobiprelim}
Let $k$, $m$ and $N$ be positive integers. A holomorphic function $\phi:\mf H\times \mf C\rightarrow \mf C$ is a Jacobi form of weight $k$, index $m$ and level $N$ if:
\begin{enumerate}

\item $\phi\left(\frac{a \tau+b}{c \tau+d}, \frac{z}{c \tau+d}\right)(c \tau + d)^{-k} e^m\left(\frac{-c z^{2}}{c \tau+d}\right)=\phi(\tau,z)$\qquad for all $\gamma=\smat{a}{b}{c}{d}\in \Gamma_0(N)$.

\item $\phi(\tau, z+\lambda\tau+\mu)e^m(\lambda^2\tau+2\lambda z)=\phi(\tau,z)$ \q\;\qquad for all $(\lambda,\mu)\in \mf Z^2$.

\item $\phi$ is bounded at the cusps of $\Gamma_0(N)\backslash \mf H$.
\end{enumerate}
If the Fourier coefficients corresponding to $4mn-h_Mr^2=0$ vanish for all $M\in \mrm{SL}_2(\mf Z)$, then $\phi$ is called a Jacobi \textit{cusp} form. We denote by $J_{k,m}(N)$ the space of Jacobi forms of weight $k$, index $m$ and level $N$. We denote the space of Jacobi cusp forms by $J_{k,m}^{cusp}(N)$. 

The operators $V_m$ are defined on Jacobi forms of index $1$ and level $N$ (even though they can be defined for any index). For $\phi \in \jkn$ define (cf. \cite[\S~4 (2)]{EZ}, \cite[Section 3]{Ibu-SK})
\begin{align} \label{vmdef}
    (\phi|V_m)(\tau,z)= m^{k-1} \sumn_{\gamma } (c \tau+d)^{-k} e(-m c z^2/(c \tau+d)^2) \phi \big(\frac{a \tau +b}{c \tau+d}, \frac{m z}{c \tau+d} \big),
\end{align}
where $\gamma =\smat{a}{b}{c}{d} $ runs over a set of representatives $\Gamma_0(N) \backslash M_{2,m}(\z)$. Here $M_{2,m}(\z)$ denotes the set of size $2$ integral matrices $\smat{a}{b}{c}{d}$ with determinant $m$ and $(a,N)=1$.

\subsection{Siegel modular forms.}
Let $\mf H_n$ denote Siegel's upper half space of degree $n$. For any integer $k$ and any function $F$ on $\mf H_n$, $M=\smat{A}{B}{C}{D}\in \mrm {GSp}^+_n(R) $ acts as 
\begin{equation}
    (F|_k M) (Z) := r(M)^k \det(CZ+D)^{-k} F(MZ),
\end{equation}
where $MZ:= (AZ+B)(CZ+D)^{-1}$ and $MJM^t=r(M) J$.

Let $\Gamma_0^{(n)}(N)=\left\{\smat{A}{B}{C}{D}\in \spn : C\equiv 0 \mod N \right\}$ denote the Hecke congruence subgroup of level $N$.  A holomorphic function $F:\h_n \rightarrow \mf C$ is called a Siegel modular form of degree $n$, weight $k$ and level $N$ if $F|_k M= F$ for all $M\in \Gn$. When $n=1$, $F$ needs to satisfy an additional condition of \textit{boundedness at cusps}. In addition, if $F$ \textit{vanishes} at all the cusps, we say that $F$ is a Siegel cusp form. The space of Siegel cusp forms of degree $n$ is denoted by $S_k^{(n)}(N)$.

Throughout this paper, all Petersson inner products are \textit{non-normalised,} so that the volume of the fundamental domain is proportional to the index of the group. To be precise, when $n=2$, the inner products are given by
\begin{align} \label{pet-def}
    \lan F,G \ran_N = \int_{\Gamma_0^{(2)}(N) \backslash \mathbf H_2} F(Z) \overline{G(Z)} \det(Y)^k d \mu(Z) \q (F, G\in S_k^{(2)}(N) ).
\end{align}
We denote by $\mc F_j(N)$ to be the Siegel's fundamental domain for the action of $\Gamma_0^{(j)}(N)$ on $\h_j$ consisting of those $Z \in \h_j$ satisfying the conditions that $|\det(c_\gamma Z + d_\gamma)| \ge 1$ for all $\gamma \in \Gamma_0^{(j)}(N)$, $Y$ Minkowski reduced and $X \bmod 1$.

\subsubsection{Double coset operators and the Hecke algebras:} Let $g\in \mrm{GSp}^+_2(\mf Q)\cap M_4(\mf Z)$ be such that $gJg^t=r(g) J$. Let $\Gamma_1, \Gamma_2\subseteq \sptwo$ be any two congruence subgroups such that $[\Gamma_2:\Gamma_2\cap g^{-1} \Gamma_1 g]<\infty$ and $[\Gamma_1:\Gamma_1\cap g \Gamma_2 g^{-1}]<\infty$. Let $\Gamma_1 g \Gamma_2 = \cup_{i=1}^{d} \Gamma_1 \alpha_i $. We define the double coset operator $\Gamma_1 g \Gamma_2: S_k^{(2)}(\Gamma_1)\longrightarrow S_k^{(2)}(\Gamma_2)$ as
\begin{equation}
    F|_k (\Gamma_1 g\Gamma_2) := r(g)^{k-3} \sumn_i F|_k\alpha_i.
\end{equation}
The degree of the operator $\Gamma_1 g\Gamma_2$ is defined to be the number of cosets $\Gamma_1\alpha$ contained in $\Gamma_1 g\Gamma_2$. It is denoted by $\deg(\Gamma_1 g\Gamma_2)$ and it is equal to $d=[\Gamma_2:\Gamma_2\cap g^{-1} \Gamma_1 g]$ in this case.

Let $G:= \mrm{GSp}_2^{+}(\z)$ be the semigroup of positive integral similitudes for the symplectic group. We will mainly work with the rational Hecke algebra $\mathscr H_\Q(G, \Gamma)$ for the Hecke pair $(G,\Gamma)$, which as we may recall is the $\Q$ vector space spanned by the double cosets $\Gamma g \Gamma$ ($\Gamma$ being a congruence subgroup) and the multiplication is more generally defined as follows: \, 
for two double cosets $\Gamma_1 g_1\Gamma_2=\cup_{i=1}^{d_1} \Gamma_1 \alpha_i $ and $\Gamma_2 g_2\Gamma_3=\cup_{i=1}^{d_2} \Gamma_2 \beta_i $, the product is defined as (see \cite[Chap.~3]{shimura-aut})
\begin{equation}\label{DCMul}
    \Gamma_1 g_1\Gamma_2\cdot \Gamma_2 g_2\Gamma_3 = \sum m(g) \Gamma_1 g \Gamma_3, 
\end{equation}
where the sum is over all $g$ such that $\Gamma_1g\Gamma_3\subset \Gamma_1g_1\Gamma_2g_2\Gamma_3$ and 
\begin{equation}
    m(g)= \#\{(i,j) : \Gamma_1 \alpha_i\beta_j= \Gamma_1g \}.
\end{equation}
\subsection{Half-integral weight modular forms.}
 Let $k\in \mf Z$ and $\Gamma_0(4N):=\Gamma_0^{(1)}(4N)$. A holomorphic function $f:\mf H\rightarrow \mf C$ is called a half-integral weight modular form of weight $\kappa=k+1/2$, if $f(\gamma\tau)= \left(\frac{c}{d}\right)^{2\kappa} \epsilon_d^{-2\kappa}(c\tau+d)^\kappa f$ for all $\gamma=\smat{a}{b}{c}{d}\in\Gamma_0(4N)$ and $f$ is holomorphic at the cusps of $\Gamma_0(4N)$. Here $\left(\frac{\cdot}{d}\right)$ denotes the Legendre symbol and $\epsilon_d= 1$ if $d\equiv 1\mod 4$ and $\epsilon_d=i$ if $d\equiv -1\mod 4$.

 For any $d|N$, the Atkin-Lehner operator at $d$ of level $N$, $W(d,N)$ is defined as 
\begin{equation}
    W(d,N)=\pmat{ds}{b}{Nc}{dt}\in \mrm{M}_2(\mf Z) \text{ such that } dst-Nb/d=1.
\end{equation}

\section{Inner product relations} \label{norm-comput}
Let $N$ be any positive integer. Then for any prime $p\nmid N$, we consider the following operators acting on the space of Siegel cusp forms of level $Np$.
\begin{enumerate}
    \item The Atkin-Lehner operator $W_p$ defined as 
    \begin{equation}\label{Wpdef}
        W_p=\pmat{p\alpha 1_2}{1_2}{Np\beta 1_2}{p1_2},
    \end{equation}
    where $\alpha$, $\beta$ are integers satisfying $p\alpha- N\beta=1$.
    \item The $U_S(p)$ operator given by the double coset $\Gamma_0^{(2)}(Np)\mrm{diag}(1,1,p,p)\Gamma_0^{(2)}(Np)$. The Fourier expansion of $F|U_S(p)$ is given by 
    \begin{equation}\label{USpdef}
        F|U_S(p)(Z)= \sumn_{T} A_F(pT)e(\tr (TZ)).
    \end{equation}
\end{enumerate}
Let $p$ and $N$ be as above. For a diagonal matrix $\mrm{diag}(p^a, p^b,p^c,p^d)$ with $a+c=b+d$, we put
\begin{equation}
    T_S(p^a, p^b,p^c,p^d)= \Gamma_0^{(2)}(N)\mrm{diag}(p^a, p^b,p^c,p^d)\Gamma_0^{(2)}(N).
\end{equation}

For given Hecke operators $R,R'$, one of the standard methods (cf. \cite{ILS}) to evaluate our desired inner products -- viz. $\lan F|R, F|R' \ran$ is to consider the residues obtained by manipulating the Rankin-Selberg zeta function of $F|R$ and $F|R'$ and to use some multiplicative relations among the Fourier coefficients. Indeed, such an approach works to a large extent for elliptic modular forms,  but it is not available in the present setting.
To obtain the inner product relations, we interpret the operators as double cosets inside suitable Hecke algebras and work with them, e.g., as in \cite{schulze2018petersson}. We take care to present all the calculations as exact norm relations are crucially required for our study of SK oldforms.

For the sake of simplicity, we write $\mu(p):=[\Gamma^{(2)}_0(N):\Gamma^{(2)}_0(Np)]$ for any integer $N$ and any prime $p\nmid N$. We start with a lemma which is a degree $2$ analogue of \cite[Thm.~6]{schulze2018petersson}, if we note the slight difference in normalization of the operators.
\begin{lem} \label{p21}
    Let $N$ be any positive integer and $p\nmid N$. Let $F,F'$ be any two Siegel cusp forms of level $N$ and $F$ an eigenfunction of the operator $T_S(p)$. Then 
    \begin{equation}
        \displaystyle \lan F', F|W_p \ran_{Np} = p^{3-k}\lambda_F(p)\mu(p)^{-1}\cdot \lan F',F \ran_{Np}.
    \end{equation}
\end{lem}

\begin{proof}
Let us first note that $\gamma W_p= B_p $ for some $\gamma\in \Gamma^{(2)}_0(N)$. To see this, recall that $W_p$ normalizes $\Gamma^{(2)}_0(Np)$ and note that (as $4 \times 4$ matrices -- the entries below have to be read as multiplied by $1_2$)
\begin{equation} \label{bp=wp}
    \pmat{pa-Npb\beta }{pb\alpha-a}{N(pc-d\beta)}{d\alpha-Nc} \pmat{p\alpha}{1}{Np\beta}{p}=\pmat{p}{0}{0}{1}\pmat{a}{b}{Npc}{d}, 
\end{equation}
$\alpha, \beta$ are as in \eqref{Wpdef} and $a,b,c,d\in \mf Z$ such that $ad-Npbc=1$. Thus $F|W_p= F|B_p$.
Since the latter has level $Np$ we can compute
\begin{align} \label{ffp}
    \mu(p)  \lan F', F|W_p \ran_{Np} =\lan F', F|B_p| \mrm{Tr}_{Np,N} \ran_{Np}, 
\end{align}
where $\mrm{Tr}_{N,M} \colon S_k^{(2)}(N) \rightarrow S_k^{(2)}(M) $ is the trace map for $M|N$ defined by $\displaystyle H \mapsto \sumn_{\gamma \in \Gamma^{(2)}_0(N) \backslash \Gamma^{(2)}_0(M) } H|\gamma$.

First, we note that $F|\Gamma^{(2)}_0(N)\smat{p}{0}{0}{1} \Gamma^{(2)}_0(Np)= p^{k-3} F|B_p$.  Writing in terms of double cosets, the rightmost operator in \eqref{ffp} is 
\begin{align}
    \Gamma^{(2)}_0(N)\smat{p}{0}{0}{1} \Gamma^{(2)}_0(Np) \cdot \Gamma^{(2)}_0(Np) 1_4 \Gamma^{(2)}_0(N)& = 1 \cdot \Gamma^{(2)}_0(N)\smat{p}{0}{0}{1} \Gamma^{(2)}_0(N)\nonumber \\
    &=   \Gamma^{(2)}_0(N) \smat{1}{0}{0}{p} \Gamma^{(2)}_0(N) =  T_S(p).
\end{align}
In the above we used the formula for multiplication of double cosets in the Hecke algebra \cite{shimura-aut}, especially \cite[(3.1.1)]{shimura-aut} and \cite[Prop.~3.3]{shimura-aut}. Here, only one double coset appears after multiplication with multiplicity one. Moreover, the second equality of double cosets follows from the equality $\psmb p & 0 \\ 0 & 1 \psme=\psmb p & -t \\ N & d \psme \psmb 1 & 0 \\ 0 & p \psme \psmb  pd & t \\ -N & 1 \psme$; where $t,d$ are integers chosen such that $pd+tN=1$.
\end{proof}

\begin{lem} \label{p2p-lem}
    For a prime $p\nmid N$, a set of coset representatives in $\Gamma^{(2)}_0(Np^2) \backslash \Gamma^{(2)}_0(Np)$ is given by $\{ \smat{1}{0}{-NpB}{1}\}$ as $B=B^t$ varies $\bmod$ $p$.
\end{lem}

\begin{proof}
   First, we know that $[\Gamma_0^{(2)}(Np):\Gamma_0^{(2)}(Np^2)]=p^3$ (see for example \cite{klingen1959bemerkung}). Next, the cardinality of the set $\{ \smat{1}{0}{-NpB}{1}, B\mod p, B=B^t\}$ is $p^3$ and for any two $B$, $B' \mod p$,  $\smat{1}{0}{-NpB}{1}\cdot \smat{1}{0}{NpB'}{1}\in \Gamma_0^{(2)}(Np^2)$ $\iff$ $B\equiv B' \mod p$. This completes the proof.
\end{proof}
Let $T_S'(p)$ be the operator defined as below (see also \cite[pp. 158]{Ibu-SK})
\begin{equation} \label{t'p}
    T_S'(p):=T_S(p)^2-T_S(p^2)=pT_S(1,p,p^2,p)+p(1+p+p^2)T_S(p,p,p,p).
\end{equation}

\begin{lem}\label{L1Uswp}
    Let $N$ be any positive integer and $p\nmid N$ and $F,F'$ be any two Siegel cusp forms of level $N$. Further, let $F$ an eigenfunction of the operators $T_S(p)$ and $T_S'(p)$ with the eigenvalues $\lambda_F(p)$ and $\lambda_F'(p)$, respectively. Then   
    \begin{equation}
        \displaystyle \lan F|U_S(p), F'|W_p \ran_{Np} =p^{3-k}\mu(p)^{-1}\left(\lambda_F(p)^2- (1+\frac{1}{p})\lambda_F'(p) +p^{2k-5}(p+1) \right) \cdot \lan F,F' \ran_{Np}.
    \end{equation}
 \end{lem}   
\begin{proof}
 First, we note that $F|U_S(p)|W_p= F|B_{p^2}| \mrm{Tr}_{Np^2,Np} $. This follows from the identity
 \begin{align}
     \smat{1}{B}{0}{p} W_p = \gamma\smat{p^2}{0}{0}{1} \smat{1}{0}{-NpB}{1} \q \text{ for some } \gamma\in \Gamma^{(2)}_0(N),
 \end{align}
and \lemref{p2p-lem}. Then one can write
\begin{align}
    \lan F|U_S(p), F'|W_p \ran_{Np} &= \lan F|U_S(p)|W_p, F' \ran_{Np}   \\
    &=   p^{k-3} \lan F|B_{p^2}| \mrm{Tr}_{Np^2,Np} , F' \ran_{Np} = p^{k-3}\mu(p)^{-1}  \lan F|B_{p^2}| \mrm{Tr}_{Np^2,N} , F' \ran_{Np}  ,
\end{align}
where the factor $p^{k-3}$ comes from the definition of $U_S(p)$ given by $(F|U_S(p))(Z)= p^{k-3}\sum_{B} F|\smat{1}{B}{0}{p}$.

First, we note that $F|\Gamma^{(2)}_0(N) \smat{p^2}{0}{0}{1} \Gamma^{(2)}_0(Np^2) \cdot \Gamma^{(2)}_0(Np^2) 1_4 \Gamma^{(2)}_0(N)= p^{2k-6} F|B_{p^2}|\mrm{Tr}_{Np^2,N}$. Next, note that
\begin{align}
\Gamma^{(2)}_0(N) \smat{p^2}{0}{0}{1} \Gamma^{(2)}_0(Np^2) \cdot \Gamma^{(2)}_0(Np^2) 1_4 \Gamma^{(2)}_0(N)&= 1 \cdot \Gamma^{(2)}_0(N) \smat{p^2}{0}{0}{1} \Gamma^{(2)}_0(N) \\
= \Gamma^{(2)}_0(N)\smat{1}{0}{0}{p^2} \Gamma^{(2)}_0(N)&=  T_S(1,1,p^2,p^2),
\end{align}
by the same logic as in the analogous situation in the proof of \lemref{p21}. Thus we get that
\begin{equation}
    F|B_{p^2}| \mrm{Tr}_{p^2,1}= p^{6-2k} T_S(1,1,p^2,p^2).
\end{equation}
To complete the proof, let 
$T_S'(p)=T_S(p)^2-T_S(p^2)$.
be as in \cite{Ibu-SK}. Then a small calculation shows that
\begin{equation} \label{Ts11p2p2}
    T_S(1,1,p^2,p^2) = T_S(p)^2- (1+\frac{1}{p})T'_S(p) +p(p+1) T_S(p,p,p,p). \qedhere
\end{equation}
\end{proof}
\begin{lem}\label{L1up}
     Let $N$ be any positive integer and $p\nmid N$. Let $F,F'$ be any two Siegel cusp forms of level $N$ and $F$ an eigenfunction of the operator $T_S(p)$ with the eigenvalue $\lambda_F(p)$. Then 
     \begin{equation}
         \lan F|U_S(p), F' \ran_{Np} =p^3\lambda_{F}(p)\mu(p)^{-1}\cdot \lan F,F'\ran_{Np}.
     \end{equation}
\end{lem}
\begin{proof}
    First, we note that
    \begin{equation}
       [\Gamma^{(2)}_0(N):\Gamma^{(2)}_0(Np)]  \lan F|U_S(p), F' \ran_{Np}= \lan F|U_S(p)|\mrm{Tr}_{Np,N}, F' \ran_{Np}.
    \end{equation}
Since $F$ is of level $N$, the operator $U_S(p)|\mrm{Tr}_{Np,N}$ can be written an
\begin{align}
    \Gamma^{(2)}_0(N)\smat{1}{0}{0}{p} \Gamma^{(2)}_0(Np) \cdot \Gamma^{(2)}_0(Np) 1_4 \Gamma^{(2)}_0(N) = p^3 \Gamma^{(2)}_0(N) \smat{1}{0}{0}{p} \Gamma^{(2)}_0(N)=p^3 T_S(1,1,p,p) =p^3T_S(p).
\end{align}
The proof now follows.
\end{proof}

\begin{lem}\label{lemUpUp}
   Let $p \nmid N$ and $F,F' \in \stwon$ be such that $F'$ is an eigenfunction of the operator $T_S(p)$ with  eigenvalue $\lambda_{F'}(p)$. Then we have
    \begin{align}
        \lan F|U_S(p), F'|U_S(p)\ran_{Np} =\mu(p)^{-1}\left(p^2(p-1) \lambda_{F'}(p)^2 +p^{2k-4}\mu(p)\right)\lan F, F' \ran_{Np}.
    \end{align}
\end{lem}

\begin{rmk}
At first glance, the above relation might look asymmetrical, but it is not.  Since the LHS and hence the RHS are non-zero only if $F=F'$, as eigenvectors with distinct $T_S(p)$ eigenvalues are orthogonal.
\end{rmk}

\begin{rmk}
    In the case of elliptic modular forms, an analogous result holds (see \cite[Theorem 8]{anamby2019distinguishing}).
\end{rmk}
\begin{rmk}
    The above evaluation of the inner product $\lan F|U_S(p), F'|U_S(p)\ran_{Np}$ can also be deduced from \cite{brown2017action}, where the authors work locally, with some work. Our counting arguments involving the double cosets are similar to those in \cite{brown2017action}, and the reader may wish to compare or consult both sources.
\end{rmk}
\begin{proof}
Let $K_1,K_2 \subset \Gamma^{(2)}$ be congruence subgroups.
We begin by noting that the adjoint of $T= K_1 g K_2$ is given by $T^*= K_2 g^* K_1$ where $g^*=r(g) g^{-1}$, where $r(g)$ is the similitude of $g$. In our situation, we have $K_1=K_2=\Gamma_0^{(2)}(Np)$  and therefore
\begin{align}
     \lan F|U_S(p), F'|U_S(p)\ran_{Np} = \lan F, F'|U_S(p) U_S(p)^* \ran_{Np}.
\end{align}
Since $F$ has level $N$, 
\begin{equation}
    \lan F, F'|U_S(p) U_S(p)^* \ran_{Np} = \mu(p)^{-1} \lan F, F'|U_S(p) U_S(p)^*| \mrm{Tr}_{Np,N} \ran_{Np}.
\end{equation}
Since $F'$ has level $N$, the operator $U_S(p) U_S(p)^* \mrm{Tr}_{Np,N}$ on the RHS is given by the following double coset operator:
\begin{align}
 &\Gamma^{(2)}_0(N)\smat{1}{0}{0}{p} \Gamma_0^{(2)}(Np) \cdot \Gamma_0^{(2)}(Np) p \smat{1}{0}{0}{p^{-1} }\Gamma_0^{(2)}(Np)\cdot \Gamma_0^{(2)}(Np) 1_4 \Gamma^{(2)}_0(N)\\
 &= \Gamma^{(2)}_0(N) \smat{1}{0}{0}{p} \Gamma_0(Np) \smat{p}{0}{0}{1} \Gamma^{(2)}_0(N).
\end{align}
Since there are $3$ distinct double cosets of similitude $p^2$ for $\Gamma^{(2)}_0(N)$, from \eqref{DCMul} we can write
\begin{equation} \label{UpUp*}
    \Gamma^{(2)}_0(N)\pmat{1}{0}{0}{p} \Gamma_0^{(2)}(Np) \pmat{p}{0}{0}{1} \Gamma^{(2)}_0(N)= c_1 T_S(1, p, p^2, p) +c_2 T_S(1,1,p^2,p^2)+c_3 T_S(p,p,p,p).
\end{equation}

First, let $\mu(p^i)$ denote that index of $\Gamma_0^{(2)}(Np^i)$ in $\Gamma^{(2)}_0(N)$. Let us recall that for a double coset $x=H_1gH_2$, $\deg(x)$ denotes the number of cosets of $H_1$ in $x$, extended by linearity. We need to have information  
 about the degrees of the operators appearing in \eqref{UpUp*}. Namely,
\begin{align}
\deg(U_S(p) U_S(p)^* \cdot \mrm{Tr}_{Np,N} )=p^6\mu(p); \label{upup*}\q&\deg T_S(1,p,p^2,p) = p \mu(p); \\
   \deg T_S(1,1,p^2,p^2) = p^3 \mu(p);\q &\deg T_S(p,p,p,p)= 1.\label{degrees}
\end{align}
Of these, the first equality in \eqref{upup*} follows from the multiplicative relation 
\begin{align} \label{deg}
\deg(xy)=\deg(x) \deg(y) \q \q \q (\text{see \cite[Sec.~3.1]{shimura-aut} }),
\end{align}
that $\deg(U_S(p))= \deg(U_S(p)^*)=p^3$, and that $\deg(\mrm{Tr}_{Np,N})=\mu(p)$. It can be easily checked that the degrees of $T=\Gamma g \Gamma$ and $T^*$ are the same by applying the map $x \mapsto x^*$ on $T = \cup_m \Gamma \, m$.
For the rest, see \cite[Chap.~6, Lem.~5.1, Cor.~7.3]{krieg1990hecke}.

Following Shimura \cite{shimura-aut}, we first calculate the multiplicities $c_2$ and $c_3$ from the combinatorics of the coset decomposition of the operators $T_S(1,1,p^2,p^2)$ and $T_S(p,p,p,p)$ respectively. Keeping \eqref{UpUp*} in mind, we first note that
\begin{equation}
    \Gamma^{(2)}_0(N)\smat{1}{0}{0}{p} \Gamma_0^{(2)}(Np) = \bigcup\nolimits_{L=L^t}\Gamma^{(2)}_0(N)\smat{1}{L}{0}{p}
\end{equation}
and 
\begin{equation}\label{GpG1}
    \Gamma_0^{(2)}(Np) \smat{p}{0}{0}{1} \Gamma^{(2)}_0(N) = \bigcup\nolimits_{i,r, \mc A}\Gamma_0^{(2)}(Np)\smat{p}{0}{0}{1} \smat{1}{0}{NpL_i}{1}w_r^{N,p} n(\mc B_r)m( \mc A),
\end{equation}
where $w_r^{N,p}\in \Gamma_0^{(2)}(N)$ such that $w_r^{N,p}\equiv w_r \mod p$ and $w_r^{N,p}\equiv 1_{4} \mod q$ for all $q|N$. In \eqref{GpG1}, $\smat{1}{0}{NpL_i}{1}$ are the coset representatives of $\Gamma_0(Np^2)$ in $\Gamma_0(Np)$ follows from \lemref{p2p-lem}.
These are multiplied by the coset representatives of $\Gamma_0(Np)$ in $\sptwo$, which is justified below.

Let $\mf F_p$ be the finite field with $p$ elements.
Since $p \nmid N$, reduction of $\Gamma_0(N) \bmod p$ equals $\mrm{Sp}_2(\mf F_p)$ with kernel $\Gamma(Np)$; from which it is easy to verify \eqref{GpG1}, given that the same holds when $N=1$, as can be found from \cite{boech-naga}.

Let $\gamma_{i,r,\mc A}:=\smat{1}{0}{NpL_i}{1}w_r^{N,p} n(\mc B_r)m( \mc A)$. We write $\gamma= \gamma_{i,r,\mc A}$ when there is no ambiguity. Note also that $\{ \gamma \}$ (as in \eqref{GpG1}) also constitute a set of coset representatives of $\Gamma_0(Np^2)$ in $\Gamma^{(2)}_0(N)$.

To compute $c_3$ we have to count pairs $(L, \gamma)$ such that 
\begin{align} \label{c3-count}
\Gamma^{(2)}_0(N)\psmb 1 & L \\ 0 & p \psme \cdot \psmb p & 0 \\ 0 & 1 \psme \gamma = \Gamma^{(2)}_0(N)\psmb p & L \\ 0 & p \psme \cdot \gamma= \Gamma^{(2)}_0(N) p 1_4.
\end{align}

With the above notation, however, this amounts to checking
\begin{align}
    \pmat{1}{L/p}{0}{1} \gamma \in \Gamma^{(2)}_0(N) \text{ or } \pmat{A+L C/p}{B + LD/p}{C}{D} \in \Gamma^{(2)}_0(N),  \q (\gamma = \psmb A & B\\ C & D \psme).
\end{align}
Thus we require $\displaystyle LC, LD\in p M_2(\z)$ which forces $L \equiv 0 \bmod p$ and thus $L=0$, since $(C,D)$ is a co-prime symmetric pair. But when $L=0$, \eqref{c3-count} holds for all the $\gamma$'s. Therefore
\begin{equation}
    c_3 = \mu(p^2).
\end{equation}

Similarly, to compute $c_2$, we have to count pairs $(L, \gamma)$ such that 
\begin{align}\label{11p2p2cond}
     \pmat{p}{L}{0}{p} \cdot \pmat{A}{B/p^2}{C}{D/p^2} = \pmat{pA+LC}{B/p+LD/p^2}{pC}{D/p} \in \Gamma^{(2)}_0(N).
\end{align}
From these we obtain the conditions that $L,C \bmod p \in \GL_2(\mf F_p)$, $D \equiv 0 \bmod p$. 

When $r=0$, a matrix multiplication (using SageMath) shows that the lower right block $D$ of $\gamma_{i,r, \mc A}$ satisfies $D\equiv 0 \mod p$ only if $m(\mc A)=0$, which is not possible. Similarly,  when $r=1$, $D\equiv 0 \mod p$ only if $\mc A=\smat{0}{0}{*}{*}$, which is again not possible since $\mc A$ is invertible. Thus, we do not have any solutions to \eqref{11p2p2cond} when $r=0,1$.

For $r=2$, write $L_i=\smat{l_1}{l_2}{l_3}{l_4}$ and $\mc B_r=\smat{b_1}{b_2}{b_2}{b_4}$. Then $\gamma$ satisfies
\begin{equation}
    \gamma\equiv\left(\begin{array}{rrrr}
0 & 0 & -1 & 0 \\
0 & 0 & 0 & -1 \\
1 & 0 & -p l_{1} + b_{1} & -p l_{2} + b_{2} \\
0 & 1 & -p l_{2} + b_{2} & -p l_{4} + b_{4}
\end{array}\right) \mod p.
\end{equation}
Since $D\equiv 0 \mod p$ from \eqref{11p2p2cond}, we see that $b_1=b_2=b_4=0$. Then $\gamma=\smat{A}{B}{C}{D}$ with
$A=0$, $B=-I$, $C=I \mod p$ and $D=-pL_i$ in \eqref{11p2p2cond}. Thus we can rewrite the condition in \eqref{11p2p2cond} as
\begin{equation}
     \pmat{L}{-1/p-LL_i/p}{pC}{-L_i} \in \Gamma^{(2)}_0(N).
\end{equation}
This is possible only if $LL_i+1\equiv 0 \mod p$ and both $L$ and $L_i$ are symmetric, invertible modulo $p$. Clearly, the choice of $L$ fixes $L_i$ mod $p$ and hence globally as well, since $L,L_i$ vary mod $p$.
Now, $L$ can be chosen to be the lift of any $2\times 2$, symmetric, invertible matrix with entries in $\mf F_p$, whose number is given by $p^3-p^2$ (see \cite{carlitz1954representations} for $p>2$ and for $p=2$ it is clear). Thus we conclude that \begin{equation}
    c_2=p^3-p^2.
\end{equation}

In principle, one could compute $c_1$ by the (more complicated) above procedure, but we curtail some work by appealing to the `degree equation'. Namely, we know that
\begin{align}\label{sumdegree}
 \deg(U_S(p) U_S^*(p)) &= c_1 \deg T_S(1,p,p^2,p) + c_2 \deg T_S(1,1,p^2,p^2) +c_3 \deg T_S(p,p,p,p);
\end{align}
and from the degrees in \eqref{upup*}, \eqref{degrees} we see that
\begin{equation}
\begin{split}
    c_1&= p^{-1}\mu(p)^{-1}\left(p^6\mu(p)- p^3\mu(p) - (p^3-p^2) p^3 \mu(p)\right)= p^4-p^2.
    \end{split}
\end{equation}
Thus we conclude that 
\begin{equation}
    U_S(p)U_S(p)^*= p^2(p^2-1) T_S(1, p, p^2, p) +p^2(p-1) T_S(1,1,p^2,p^2)+p^3\mu(p) T_S(p,p,p,p).
\end{equation}
Now we have that
\begin{equation}
    T_S(1,p,p^2,p) = p^{-1}T_S'(p)-(1+p+p^2)T_S(p,p,p,p). \label{TS1pp2p}
\end{equation}
Combining this with the expression for $T_S(1,1,p^2,p^2)$ from \eqref{Ts11p2p2}, we see that
\begin{align}
    (p+1)T_S(1,p,p^2,p)+T_S(1,1,p^2,p^2)&= T_S(p)^2+\left(p(p+1)-(p+1)(1+p+p^2)\right) T_S(p,p,p,p)\nonumber\\
    &= T_S(p)^2- (p+1)(p^2+1) T_S(p,p,p,p).
\end{align}
Thus
\begin{equation}
    U_S(p) U_S(p)^* \cdot \mrm{Tr}_{Np,N}=p^2(p-1) T_S(p)^2 + \left(-p^2(p^4-1)+p^3\mu(p)\right) T_S(p,p,p,p).
\end{equation}
Thus by noting that $\mu(p)=(p+1)(p^2+1)$ and that $F|T_S(p,p,p,p)= p^{2k-6} F$, we get that
\begin{equation}
    \lan F, F'|U_S(p) U_S(p)^* \cdot \mrm{Tr}_{Np,N} \ran_{Np}= \mu(p)^{-1}\left(p^2(p-1) \lambda_{F'}(p)^2 +p^{2k-4}\mu(p)\right)\lan F, F'\ran_{Np}.\qedhere
\end{equation}
\end{proof}

\begin{lem}
    Let $N$ be any positive integer, $p\nmid N$ and $F,F'$ be any two Siegel cusp forms of level $N$. Further, let $F$ an eigenfunction of the operators $T_S(p)$ and $T_S'(p)$ with the eigenvalues $\lambda_F(p)$ and $\lambda_F'(p)$, respectively. Then
    \begin{equation}
        \lan F|U_S(p)|W_p, F'|U_S(p) \ran_{Np}= \frac{\lambda_F(p)}{p^{k-3}\mu(p)}\left(\lambda_F(p)^2- (2+\frac{1}{p})\lambda_F'(p) +p^{2k-5}(p+1)(p+2) \right)\cdot \lan F,F' \ran_{Np}.
    \end{equation}
\end{lem}
\begin{proof}
    As in Lemma \ref{L1Uswp}, we have that $F|U_S(p)|W_p= F|B_{p^2}| \mrm{Tr}_{Np^2,Np} $. Thus we get
    \begin{align}
    \lan F|U_S(p)|W(p), F'|U_S(p) \ran_{Np} &= \lan F|B_{p^2}| \mrm{Tr}_{Np^2,Np}, F' |U_S(p)\ran_{Np} \\
    &=   p^{k-3}\mu(p)^{-1} \lan F|B_{p^2}| \mrm{Tr}_{Np^2,Np}|U_S(p)^*|\mrm{Tr}_{Np,N} , F' \ran_{Np}.
\end{align}
In terms of double cosets, the operator $B_{p^2}| \mrm{Tr}_{Np^2,Np}|U_S^*(p)|\mrm{Tr}_{Np,N}$ can be written as
 \begin{align}
     &p^{6-2k} \Gamma^{(2)}_0(N) \smat{p^2}{0}{0}{1} \Gamma^{(2)}_0(Np^2) \cdot \Gamma^{(2)}_0(Np^2) 1_4 \Gamma^{(2)}_0(Np)\cdot \Gamma^{(2)}_0(Np) p \smat{1}{0}{0}{p^{-1} }\Gamma^{(2)}_0(N) \label{lem3.8-deg}\\
     &= p^{6-2k} \Gamma^{(2)}_0(N) \pmat{p^2}{0}{0}{1}\Gamma^{(2)}_0(Np)\pmat{p}{0}{0}{1} \Gamma^{(2)}_0(N). \label{lem3.8-deg1}
 \end{align}
 On the one hand, the double coset operator on the RHS above is of similitude $p^3$ and degree $p^3 \mu(p^2)$. This follows from the multiplicative-ness of the degree (cf. \eqref{deg}, and the discussion around it, applied to \eqref{lem3.8-deg}). On the other hand, from the ``symplectic divisor theorem", and \cite[Chap.~7, Cor.~7.3]{krieg1990hecke}, $T_S(1,1,p^3,p^3)$ is of degree $p^3 \mu(p^2)$. Therefore, when we multiply the double cosets in \eqref{lem3.8-deg1}, only one term can survive in the result, which has to be $T_S(1,1,p^3,p^3)$. Therefore,
 \begin{equation}
     \Gamma^{(2)}_0(N) \pmat{p^2}{0}{0}{1}\Gamma^{(2)}_0(Np)\pmat{p}{0}{0}{1} \Gamma^{(2)}_0(N) = T_S(1,1,p^3,p^3).
 \end{equation}
 Thus we get
\begin{equation}\label{USWpUS*TSp3}
    \lan F|U_S(p)|W_p, F|U_S(p) \ran_{Np} =  p^{3-k}\mu(p)^{-1} \lan F|T_S(1,1,p^3,p^3) F\ran_{Np}.
\end{equation}
Now we express the operator $T_S(1,1,p^3,p^3) $ in terms of $T_S(p)$ and $T_S'(p)$. First, we use the computations available from \cite{blo-pohl} to write
\begin{align}
    T_S(1,1,p^3,p^3) = x_0^3 \left( \mf x_3^{(0,0)} + \frac{p-1}{p}\mf x_3^{(0,1)} + \frac{(p-1)(2p-1)}{p^2}\mf x_3^{(1,1)} \right),
\end{align}
where $\mf x_3^{(a_1,a_2)}$ denotes the symmetric polynomial
\begin{align}
    \mf x_3^{(a_1,a_2)} := \sum_{(b_1,b_2) \in W_3(a_1,a_2)} x_1^{b_1} x_2^{b_2}.
\end{align}
Here $W_3(a_1,a_2)$ is the Weyl orbit of $(a_1,a_2)$ under the Weyl group -- in our case, since $(a_1,a_2) \in \{0,1\}$; they are simply
\begin{align}
    W_3(0,0) &=\{ (0,0), (0,3), (3,0), (3,3) \}, \\
    W_3(0,1) &=\{ (0,1), (1,0), (3,1), (1,3), (0,2), (2,0), (3,2), (2,3) \}, \\
    W_3(1,1) &=\{ (1,1), (2,1), (1,2), (2,2) \}.
\end{align} 
This gives us
\begin{align}
    \mf x_3^{(0,0)} &= 1+ x_1^3+x_2^3+x_1^3x_2^3, \\
     \mf x_3^{(0,1)} &= x_1+x_2+x_1x_2^3+x_1^3x_2 + x_1^2+x_2^2+ x_1^2x_2^3+x_1^3x_2^2,\\
      \mf x_3^{(1,1)} &= x_1x_2+ x_1x_2^2+x_1^2x_2+x_1^2x_2^2.
\end{align}
We also need the images for the operators $T_S(p)$ and the similitude $p^2$ Hecke operators in the polynomial algebra under the Satake map, see e.g., \cite{andrianov1, andrianov2}, \cite[Appendix]{blo-pohl}.
\begin{align}
    T_S(p) &= x_0(1+x_1)(1+x_2),\\
    T_S(1,1,p^2,p^2) &=  x_0^2 \left( (1+x_1^2)(1+x_2^2) + \frac{p-1}{p} (x_1+x_2)(1+x_1x_2) + \frac{2(p-1)}{p}  x_1x_2 \right), \\
    T_S(1,p,p^2,p) &= x_0^2 \left( \frac{p^2-1}{p^3} x_1x_2 + \frac{1}{p}(x_1+x_2)(1+x_1x_2) \right), \\
    T_S(p,p,p,p ) &= \frac{1}{p^3} x^2_0 x_1x_2.
    \end{align}
Thus to express $T_S(1,1,p^3,p^3)$ as a polynomial in the similitude $p,p^2$ operators, it is enough to do so for their images. We obtain
\begin{align}
     T_S(1,1,p^3,p^3) &= x_0^3 (1+x_1)(1+x_2) \left( (1-\frac{1}{p }x_1 + x_1^2)(1- \frac{1}{p}x_2 + x_2^2) +\frac{p-1}{p} x_1 x_2 \right) \\
     &= x_0^3 (1+x_1)(1+x_2) \left(  (1+x_1^2)(1+x_2^2) - \frac{1}{p} (x_1+x_2)(1+x_1 x_2) +\frac{1-p+p^2}{p^2} x_1x_2 \right)\\
     &= T_S(p)\left(T_S(1,1,p^2,p^2)-pT_S(1,p,p^2,p)+p^2T_S(p,p,p,p)\right)\\
     &= T_S(p)\left((p+1)T_S(1,1,p^2,p^2)-pT_S(p^2)+p(p+1)T_S(p,p,p,p)\right).
\end{align}
Since $T_S(1,1,p^2,p^2) = T_S(p)^2- (1+\frac{1}{p})T'_S(p) +p(p+1) T_S(p,p,p,p)$, the terms inside the bracket can be written as
\begin{align}
     &(p+1)T_S(p)^2-p^{-1}(p+1)^2T'_S(p)-pT_S(p^2)+p(p+1)(p+2)T_S(p,p,p,p)\\
     &=T_S(p)^2+pT_S'(p)-p^{-1}(p+1)^2T'_S(p)+p(p+1)(p+2)T_S(p,p,p,p)\\
     &= T_S(p)^2-\left(2+\frac{1}{p}\right) T_S'(p)+p(p+1)(p+2)T_S(p,p,p,p).
\end{align}
Thus we get
\begin{equation}
    T_S(1,1,p^3,p^3)= T_S(p)\left(T_S(p)^2-\left(2+\frac{1}{p}\right) T_S'(p)+p(p+1)(p+2)T_S(p,p,p,p) \right).
\end{equation}
The proof is now complete by substituting the above expression for $T_S(1,1,p^3,p^3)$ in \eqref{USWpUS*TSp3}.
\end{proof}

\section{Saito-Kurokawa lifts of higher level} \label{sec:SKintro}
Apart from  our immediate (analytic) applications, this part also has its independent interest and can serve as a classical treatment of SK lifts of higher levels.

Higher level SK newforms have been described from the viewpoint of representation theory, see e.g., \cite{schmidt-SKlift}.
Namely Schmidt shows that when $k$ is even, $N$ is square-free, and $f \in S_{2k-2}^{new}(N)$ is newform there is a functorial lifting from $\mrm{SL}(2) \times \mrm{SL}(2) \rightarrow \mrm{PGSp}(4)$ such that $1 \times \pi_f$ defines an automorphic form, which is the SK lift in question. Other local sign conditions lead to paramodular forms.

The definition via $L$-functions which we adopt is a consequence of this approach, recalled briefly below. However, for the purpose of this paper, the classical version is the backbone, and so we reformulate the situation in the classical language as well. 

\subsection{Definition of SK lifts} \label{sklift-defn}
Let $k$ be even and $N$ square-free. Let $F\in S_k^{(2)}(N)$ be a Hecke eigenform of all Hecke operators $T_S(m)$ for $(m,N)=1$ and $\lambda_F(m)$ be the corresponding Hecke eigenvalues. Define
\begin{equation}
    L^N(F,s):= \zeta^{N}(2s+1)\sumn_{(m,N)=1}\frac{\lambda_F(m)}{m^s};
\end{equation}
the prime-to-$N$ spinor zeta function of $F$.
Then $F$ is called a \textit{classical Saito-Kurokawa lift}, if there exists an integral weight newform $f$ of weight $2k-2$ and level $M|N$ such that
\begin{equation}\label{SKLfunc}
    L^N(F,s)= \zeta^N(s+1/2)\zeta^N(s-1/2)L^N(f,s).
\end{equation}
Then the subspace of Saito-Kurokawa lifts of level $N$, $\skkn$ is defined as the span of all such Hecke eigenforms $F\in S_k^{(2)}(N)$ that satisfy the property \eqref{SKLfunc} (see \cite{dickson2015siegel, schmidt-SKlift} for more details).

\subsection{Old and newforms in degree 2} \label{skold}
For $N$ square-free, we now briefly recall the theory of oldforms in the classical language from \cite{schmidt-AL} -- namely $F \in \sktwon$ is `old' at level $N$ if there is a prime $p|N$ and $F$ arises as the image of one of the following operators acting on modular forms of level $N/p$:
\begin{enumerate} \label{Schmidt-old}
    \item The identity map $T_0(p)$.
    \item The Atkin-Lehner operator $W_p$ defined as in \eqref{Wpdef}.
    \item The $U_S(p)$ operator defined as in \eqref{USpdef}.
    \item The operator $U_S(p)W_p$.
\end{enumerate}
The span of such `old' forms is the oldspace -- denoted by $\sktwonold$. The space of newforms $ \sktwonnew$ is defined as the orthogonal complement of $\sktwonold$. We put 
\begin{align} \label{new-old-def}
    \skk^{new}(N) := \sktwonnew \cap \skkn; \q \skk^{old}(N) := \sktwonold \cap \skkn.
\end{align}
To proceed further, it will be necessary for us to describe the above spaces more concretely -- namely as functorial lifts of spaces of Jacobi forms -- which was the classical approach taken by Maa{\ss}, Eichler-Zagier \cite{EZ} and finally Ibukiyama \cite{Ibu-SK} in level $N$. This is what we discuss next.

\subsection{Classical lifts from Jacobi forms (EZI lifts):}\label{EZI Lifts}
Let $\mc L_N: J_{k,1}^{cusp}(N)\longrightarrow \sktwon$ \, ($\phi \mapsto F=\mc L_N \phi$) denote the lifting map of level $N$ from \cite{Ibu-SK} defined via an explicit Fourier expansion as given below:
\begin{align}
    \mc L_N \phi (\tau, z, \tau') := \sumn_{m \ge  1} (\phi|V_m)(\tau,z) e(m\tau'),
\end{align}
where $V_m$ is the operator defined as in \eqref{vmdef}.

In \cite{Ibu-SK}, it is shown that $\mc L_N$ is injective and Hecke-equivariant in a certain sense. Let $\skk^J(N)$ denote the $\mc L_N$- image of $J_{k,1}^{cusp}(N)$ inside $S_k^{(2)}(N)$. From the work of Ibukiyama \cite{Ibu-SK} it follows that $\skk^J(N)\subset \skkn$. More precisely, we refer the reader to the REMARK after \cite[Theorem~4.1]{Ibu-SK} where he shows \eqref{SKLfunc} for the eigenforms in $\skk^J(N)$.
The newspaces in each coincide; however, this is a strict inclusion for the corresponding oldspaces.
For example, it turns out that $\skk^J(p)$ captures all the newforms in $\skkp$; and all oldforms therein except for one oldform per $\phi$ of level one (see subsection \ref{MaassDist}). This picture generalises to square-free $N$.

Let $\phi\in J_{k,1}^{cusp,new}(N)$ and $p\nmid N$. Consider the EZI lift $F=\mc L_N(\phi)$. Then we have the following relation between the eigenvalues (\cite[Theorem 4.1]{Ibu-SK}).
\begin{equation}\label{Heckrel}
    \lambda_F(p)=\lambda_\phi(p)+p^{k-1}+p^{k-2}.
\end{equation}
Now $\lambda_\phi(p) =\lambda_f(p) $. Therefore, note that by Deligne's bound, $\displaystyle \lambda_F(p) \asymp p^{k-1}$.

\subsection{SK newforms as EZI lifts}
In this subsection, we show that the EZI lifts of Jacobi newforms of level $N$ are precisely those in the sense of R. Schmidt's definition -- i.e., they are orthogonal to the oldspace. This is known from \cite[Lem.~5.2.2]{saha-paul}, but our proof is classical, and in the spirit of this paper.

\begin{lem}\label{basisnew}
    Let ~$\skk^{new}(N)$ be the space of newforms in $\skkn$ as defined in \eqref{new-old-def} and let $\skk^{J,new}(N)$ be the space spanned by the EZI lifts of newforms in $\jkn$. Then $ \skk^{new}(N) = \skk^{J,new}(N)$ and all newforms in $\skkn$ are precisely the EZI lifts of the Jacobi newforms as above.
\end{lem}

\begin{proof}
The inclusion $\skk^{J,new}(N) \subset \skk^{new}(N) $ would follow if we can show that any EZI lift $F = \mc L_N(\psi) \in \skk^{J,new}(N)$ is orthogonal to all oldforms. We prove the inclusion by induction on the number of prime factors of $N$. When $N=1$, this is true by the classical theory of Eichler-Zagier. Now assume that the inclusion holds for level $N/p$ with square-free $N$ and $(p,N/p)=1$. Then we know that $\skk^{old}(N)=\sum_{M|N, M<N}\skk^{new}(M)|\prod_{q|\frac{N}{M}} R(q)$.
Thus by induction hypothesis, we have to check that 
\begin{equation}
    \lan \mc L_{N}(\psi) , \mc L_{M}(\phi)|R(q) \ran=0 \text{ for } q|\frac{N}{M},
\end{equation}
where $\phi\in J_{k,1}^{new}(M)$ and $\psi\in J_{k,1}^{new}(N)$. Note that the Hecke operators $T_S(q')$ with $(q',N)=1$ act the same way on both $\skk^{new}(N)$ and $\skk^{new}(M)$ and commute with the operators $R(q)$. The assertion now follows, since $\psi$, $\phi$ are newforms of different levels, $\lan \psi, \phi \ran =0$, via multiplicity-one for $\GL(2)$.

For the other inclusion, we note that the spaces have the same dimension, as follows on the one hand from the injectiveness of the EZI map $\mc L_N$ (cf. \cite{Ibu-SK}), and on the other hand from Schmidt's construction and uniqueness \cite[Thm.~5.2~(ii)]{schmidt-SKlift} or from \cite[Lem.5.2.2]{saha-paul} respectively.
\end{proof}

\begin{rmk}
    The reader may directly verify in the above proof that $\displaystyle \lan \mc L_{N}(\psi), \mc L_{N}(\phi)|R(q) \ran = c \cdot \lan \psi, \phi \ran$ for some constant $c$, if $R(q)=Id, U(q)$. This follows, for e.g., from the calculations in \cite{Ag-Br}: note that the calculations work fine if at least one of the components is a newform. 
\end{rmk}

\subsection{Characterization of SK lifts of level.}

In this section, we provide a characterization of SK lifts in the level aspect (see \cite{farmer2013survey} for level $1$) in terms of the inner product matrix $M_p$. Throughout this subsection, let $N$ be a square-free positive integer and $F$ be a newform of level $N$ with eigenvalues $\lambda_F(n)$ for $(n,N)=1$. For any $p\nmid N$, let $F_{1,p}=F$, $F_{2,p}=F|U_S(p)$, $F_{3,p}=F|W_p$ and $F_{4,p}=F|U_S(p)W_p$. 

\begin{lem}\label{Mprankge3}
Let $k \ge 2$ and $p$ be a prime such that $p\nmid N$. Suppose additionally that $p>2\times 10^4$ if $k=2$ and $p\ge 11$ if $k>2$. Then for any newform $F\in S_k^{(2)}(N)$, the matrix $M_p(F)$ defined as in \eqref{Mpdef} has rank at least 3. Further, if $F$ is an SK lift, the conclusion holds for all $p\nmid N$.
\end{lem}

\begin{rmk} \label{p17}
    The omission of the first few primes arises when we want to rule out certain relations among eigenvalues, see e.g., the expressions in brackets in \eqref{Leadminor3x3} (cf. also \cite[Theorem 4.1]{farmer2013survey}). We believe that these restrictions may not be necessary. Perhaps even congruences between eigenvalues can be of help in this regard. 
\end{rmk}

\begin{proof}
We appeal to the classification \cite{arthur} of the newforms in $S_k^{(2)}(N)$ as described e.g., in \cite[Prop.~2.3.1]{saha-paul}
in terms of the Arthur-packets of the $\mrm{SO}(5) (\cong \mrm{PGSp}(2))$-automorphic spectrum, which has been classified by the work of Arthur \cite{arthur}. We first observe\footnote{We thank the referee for suggesting this proof.} that in the classification table \cite[Table~3]{schmidt-para}, newforms associated to the Howe–Piatetski–Shapiro type or the Soudry type (which in our case can occur only when $k=2$) do not exist in square-free levels. This can also be read off from \cite[Prop.~2.1, Section~2.3]{roy-schmidt-yi}. The argument is that there
exists a place $p$ where the parametrizing data $(\chi_1,\chi_2)$ in the Borel case (since $\chi_1 \neq \chi_2$ are quadratic Dirichlet characters), $\xi$ in the Klingen case, is
ramified. Then the local representations in the Arthur packet of type ({\bf B}) and ({\bf Q}) in table~1 and table~3, respectively, of \cite{schmidt-para}) do not admit $\Gamma_0(p)$-invariant vectors. The reason for the latter is that the local representations are not Iwahori-spherical, see \cite[Prop.~2.1]{roy-schmidt-yi}.

We therefore assume that $F$ is either of general type, or Yoshida type, or of SK type, see \cite{saha-paul}. 
When $k \ge 3$, the Ramanujan conjecture on the Hecke eigenvalues is known from the work of Weissauer \cite{weissauer}, and of course it fails for the SK lifts. When $k=2$, the Ramanujan conjecture is open in the general type case, but here we have the bound $|\lambda_F(p)|\le 4 p^{k-3/2+9/22}$ which follows from the transfer to $\mrm{GL}(4)$ automorphic representations (see \cite{kim2003functoriality}).

Thus, we can write, in our situation,
\begin{align}\label{lambdap-bound}
 |\lambda_F(p)|\le \begin{cases}
   4 p^{k-3/2}  & F \text{ not an SK lift and }  k >2;\\
   4 p^{k-3/2+9/22} & F \text{ not an SK lift and }  k =2.
\end{cases}
\end{align}
Thus we also get (see for example, \cite[Eq. 2.11]{anamby2021large} for the constant in the front)
\begin{align}\label{lambda'p-bound}
 |\lambda_F'(p)|\le \begin{cases}
   (6+\frac{1}{p}) p^{k-3/2}  & F \text{ not an SK lift and }  k >2;\\
   (6+\frac{1}{p}) p^{k-3/2+9/11} & F \text{ not an SK lift and }  k =2.
\end{cases}
\end{align}
We now get back to the proof.

It is enough to show that the leading $3 \times 3$ minor is non-zero. Using the inner product relations from Section \ref{norm-comput}, the leading minor is given by \footnote{This can be verified by a SAGE code available \href{https://drive.google.com/file/d/1tRieixIu2tMX_geLfEpXOBl5sX7ETpdJ/view?usp=drive_link}{here}.} (see \eqref{t'p} and \lemref{L1Uswp} for $\lambda_F'(p)$)
    \begin{equation}
      \begin{split}\label{Leadminor3x3}
          -&p^{-2 \, k}\left(p^{2} + 1\right)^{-1}[\Gamma_0^{(2)}(N):\Gamma_0^{(2)}(Np)]^{-2}\left(p \lambda_F(p)^{2} - (p^2+1) \lambda_F'(p)- p^{2 \, k-2} - p^{2 \, k }\right)\\
          &\times \left(p^{5} \lambda_F(p)^{2} - p^4(p+1)^2\lambda_F'(p) + \left(p^{2} + 2\right) {\left(p + 1\right)}^{2} p^{2 \, k}\right) \lan F, F\ran_{Np}.
      \end{split}
    \end{equation}
When $F$ is not an SK lift, we consider two cases.

\textbf{Case 1: $k=2$ --} In this case, using the bounds for $\lambda_F(p),\lambda_F'(p)$ from above and taking $p>2\times 10^4$, both terms involving the eigenvalues in \eqref{Leadminor3x3} are non-zero. 

\textbf{Case 2: $k>2$ --} In this case, taking $p\ge 11$ and using the bounds  $|\lambda_F(p)|\le 4 p^{k-3/2}$ and $|\lambda_F'(p)|\le (6+\frac{1}{p}) p^{2k-3}$, we see that both terms involving the eigenvalues in \eqref{Leadminor3x3} are non-zero. 

When $F$ is an SK lift, then we have that $\lambda_F'(p)= (p^{k-1}+p^{k-2})\lambda_F(p)-p^{2k-2}$. Thus \eqref{Leadminor3x3} reduces to
\begin{equation} \label{sk-adv}
    \begin{split}
       -&p^{-2k+6}\left(p^{2} + 1\right)^{-1}[\Gamma_0^{(2)}(N):\Gamma_0^{(2)}(Np)]^{-2} \lambda_F(p)  \left( \lambda_F(p) - 2 (p^{k -1} +p^{k-2})\right)\\
       &\left( \lambda_F(p) - p^{k-3 }(p^2+1)(p+1)\right)^{2}\lan F, F\ran_{Np}.
    \end{split}
\end{equation}
Now since $F$ is an SK lift (say of $f\in S_{2k-2}(N)$), one has $\lambda_F(p)= p^{k-1}+p^{k-2}+\lambda_f(p)$ from \eqref{SKLfunc} and by Deligne, $|\lambda_f(p)|\le 2p^{k-3/2}$. Thus, we get the required non-vanishing for all $p \nmid N$.

Thus, in both cases, the leading minor is non-zero. Thus $M_p$ has rank at least 3.
\end{proof}

With the preparations so far, we can now obtain the desired characterization of SK lifts -- stated as \thmref{skchar-intro} in the introduction.
\begin{thm}\label{th:SKChar}
     Let $k \ge 2$ and $N$ be a square-free positive integer. Then the following are equivalent.
    \begin{enumerate}
        \item 
        $F$ is an SK lift (in the sense of subsection~\ref{sklift-defn}) of an elliptic eigenform of level $N$ and weight $2k-2$ or equivalently a Jacobi eigenform of square-free level $N$.
        
         \item 
         For all $p \nmid N$,
        \begin{equation*}
            \lambda_F(p^2)=\lambda_F(p)^2-(p^{k-1}+p^{k-2})\lambda_F(p)+p^{2k-2}.
        \end{equation*}
        
        \item 
        There exists a prime $p\nmid N$  such that
        \begin{equation*}
            \lambda_F(p^2)=\lambda_F(p)^2-(p^{k-1}+p^{k-2})\lambda_F(p)+p^{2k-2}.
        \end{equation*}
        
        \item 
        The inner product matrix $M_p:=\left(\lan F_{i,p}, F_{j,p} \ran_{Np} \right)_{1\le i, j\le 4}$ has rank $3$ for some prime $p\nmid N$ such that $p\ge 17$ (resp. $p> 2\times 10^4)$ if $k>2$ (resp. $k=2$).
        
        \item For all $p \nmid N$ and $p\ge 17$ (resp. $p> 2\times 10^4)$ if $k>2$ (resp. $k=2$), $M_p$ has rank $3$.
    \end{enumerate}
\end{thm}
\begin{proof}
We recall, from the beginning of the proof of Lemma~\ref{Mprankge3}, that in our present case, we can assume that $\lambda_F(p)$ satisfies the bound in \eqref{lambdap-bound} if $F$ is not an SK lift.

The statements (1)-(3) are the level analogues of the corresponding statements for level $1$ in \cite[Theorem 4.1]{farmer2013survey}. When $k>2$, their equivalence follows from the same arguments as in \cite{farmer2013survey}. Namely, the main point for level $1$ was that $F$ satisfies the Ramanujan conjecture if and only if it is not an SK lift. This assertion is also true for level $N$ newforms from the work of Weissauer \cite{weissauer}. Thus our $F$, 
which shares the same eigenvalues as that of a newform of some level $M|N$ has the same property away from the level. 
We do not reproduce the argument again.

When $k=2$, $(1)\implies(2)\implies (3)$ follows from the same argument as in \cite{farmer2013survey}. For $(3)\implies (1)$, an argument as in \cite{farmer2013survey} shows that one of the \textit{normalized} Satake parameters ($\alpha_p, \beta_p$ as in loc.cit.) must be of size $p^{\pm 1/2}$. However, if $F$ were not an SK lift, from \cite{kim2003functoriality} it is known that both $|\alpha_p|$ and $|\beta_p|$ are $\le p^{9/22}$, which is a contradiction. This gives us the required equivalence.

We next show that (5)$\implies$(4)$\implies$(1)$\implies$(2)$\implies$ (5). The first implication (5)$\implies$(4) is trivial. To show (4)$\implies$ (1), 
we argue by proving the contrapositive. So suppose that $F$ is not an SK lift.
Next, note that for any $p\nmid N$ we have, 
\footnote{One can use the same SAGE code linked above to verify this.}
\begin{equation}\label{detMp}
        \begin{split}
            \frac{\det(M_p)}{\lan F, F\ran_{Np}}=&p^{-4k}{\left(p^{2} + 1\right)}^{-4} {\left(p + 1\right)}^{-4}\left(p^{5} \lambda_F(p)^{2} - p^4(p+1)^2\lambda_F'(p) + \left(p^{2} + 2\right) {\left(p + 1\right)}^{2} p^{2 \, k}\right)^{2}\\
    &\times\left( \lambda_F'(p) + (p^{k -1}+ p^{k-2}) \lambda_F(p) + p^{2 \, k-2}\right){\left(\lambda_F'(p) - (p^{k - 1} + p^{k-2}) \lambda_F(p) + p^{2 \, k-2}\right)}.
        \end{split}
    \end{equation}
 Next, we recall the bounds on the Hecke eigenvalues $\lambda_F(p)$ and $\lambda_F'(p)$ from \eqref{lambdap-bound} and \eqref{lambda'p-bound}.
    As we have observed in the proof of Lemma~\ref{Mprankge3}, if $p$ is as in (5), then the terms involving eigenvalues $\lambda_F(p)$ and $\lambda_F'(p)$ are non-zero in \eqref{detMp}. Thus $\det(M_p)\neq 0$. We note here that for $k>2$, the bound $p\ge 17$ is used to show that the second and third terms in \eqref{detMp} are non-zero (this is not surprising, see \cite[Theorem 4.1]{farmer2013survey} in this connection), whereas the bound for the first bracket required only $p \ge 11$, as shown in the proof of Lemma~\ref{Mprankge3}.
    
    We have already shown (1)$\implies$(2). To show (2)$\implies$(5), recall from Lemma \ref{Mprankge3} that the rank of $M_p$ is at least $3$ for any $p\nmid N$ as in the statement of (5). Thus, to show (2) $\implies$ (5), it is enough to show that the rank is equal to $3$ when the condition (2) is satisfied. 
We will now show that the rank of $M_p \le 3$.
Indeed, (2) says that $\lambda_F'(p)=\lambda_F(p)^2-\lambda_F(p^2)$. Thus from \eqref{detMp}, $\det(M_p)=0$ as soon as (2) is satisfied. Thus $M_p$ has rank $3$ for all $p\nmid N$ and $p\ge 17$ (resp. $p> 2\times 10^4)$ if $k>2$ (resp. $k=2$). This completes the proof.
\end{proof}

\subsection{SK oldforms in degree 2}
For a prime $p$ with $p|N$, define, for each $\phi \in J_{k,1}^{new}(N/p)$,
\begin{align} \label{vphi-def}
  V(\phi;N/p):= \mrm{Span}< \mc L_{N/p} \phi |R, \q R=Id, W_p, U_S(p), U_S(p)W_p>.
\end{align}
We comment here that it is possible to deduce that $\dim V(\phi;N/p)=3$ if we combine Theorem \ref{th:SKChar} with \lemref{basisnew} and replace $N$ with $N/p$. 

But here we give an alternate proof of the same fact by expressing $(L_{N/p} \phi )|U_S(p)W_p$ as an explicit linear combination of other $3$ lifts of $\phi$. We believe that an explicit linear relation is useful in various contexts, including the sup-norm problem. Moreover, the following calculations also serve as a sanity check for the calculations about the discriminant of $M_p$ in the proof of \thmref{th:SKChar}.

\begin{lem} \label{dim=3}
    For each $\phi \in J_{k,1}^{new}(N/p)$, there is an explicit linear relation (given in \eqref{expl-reln}) among the $4$ oldforms listed in \eqref{vphi-def}; in particular,
    one has $\dim V(\phi;N/p)=3$.
\end{lem}

\begin{rmk}
The reader will notice that there is no mention of the slightly restrictive conditions on $p$ as in Lemma~\ref{Mprankge3} and Theorem~\ref{th:SKChar}. This is because if we already assume that $F$ is a SK lift (of weight $k \ge 2$), then the implication (1)$\implies$(2) holds for all $p \nmid N$, and then for all these primes (2)$\implies$(4). This follows immediately from the expression \eqref{sk-adv}. Thus no such conditions are needed in the above lemma, and henceforth in the paper. 
\end{rmk}

\begin{proof}
Let $\phi \in J_{k,1}^{cusp,new}(N/p)$, and put $F= \mc L_{N/p} \phi$. For some constant $c$, let $v=v_c := F+c F|U_S(p)$ and $w=w_c:= -v|W_p$. Now, it is enough to show that $\lan v+w, v+w\ran =0$ for some unique $c$. To achieve this, we make extensive use of the inner product relations obtained in the previous section.

We first note that $\lan v+w, v+w\ran =0$ iff $c$ satisfies the quadratic equation $Ax^2+Bx+C$, where
\begin{equation}
\begin{split}
    A&=2\left(\lan F|U_S(p), F|U_S(p) \ran_{Np} -  \lan F|U_S(p), F|U_S(p)W_p \ran_{Np} \right);\\
    B&=4\left(\lan F, F|U_S(p)\ran_{Np} - \lan F|U_S(p), F|W_p \ran_{Np} \right);\\
    C&=2\left(\lan F, F \ran_{Np} - \lan F, F|W_p \ran_{Np} \right).
\end{split}
\end{equation}
From the inner product calculation it is clear that $A, B\neq 0$, thus showing that the discriminant $D:=B^2-4AC=0$ would give us the unique $c$.

Using the inner product relations, we can see that $D$ factorises into \footnote{One can use the same SAGE code linked above to verify this.}
\begin{equation}
\begin{split}
    16 &p^{-2k}[\Gamma_0^{(2)}(N):\Gamma_0^{(2)}(Np)]^{-2} \lan F, F\ran_{Np} \left( -\lambda_F'(p) + (p^{k -1}+ p^{k-2}) \lambda_F(p) - p^{2k-2}\right)\\
   & \times \left(p^{5} \lambda_F(p)^{2} - p^4(p+1)^2\lambda_F'(p) + \left(p^{2} + 2\right) {\left(p + 1\right)}^{2} p^{2 \, k}\right).
\end{split}
\end{equation}
Now on the space of SK lifts of level $N/p$, we have that 
\begin{align}\label{SKrelHecke}
    T'_S(p)= (p^{k-1}+p^{k-2})T_S(p) - p^{2k-2} Id.\q \text{ since } (p, N/p)=1.
\end{align}
Thus we have that
\begin{equation}
    \left( -\lambda_F'(p) + (p^{k -1}+ p^{k-2}) \lambda_F(p) - p^{2k-2}\right)=0.
\end{equation}
This in turn implies that $D=0$ and gives us the unique $c$. To conclude, we have thus $v=-w=v|W_p$,  which means
\begin{align} 
 \label{expl-reln} F_1+cF_2=F_3+cF_4, \q c = -B/2A= \frac{1}{p^{k-2}-\lambda_F(p)}.
\end{align}
The denominator in the expression for $c$ is clearly non-zero from \eqref{Heckrel}. Therefore, we have also shown that $\dim V(\phi;N/p)\le 3$. \lemref{dim=3} now follows from the second assertion of \lemref{Mprankge3}.
\end{proof}

\begin{rmk} \label{mp-not-fe}
    The reader may notice that another natural way to prove \lemref{dim=3} would be simply to compare the respective
    Fourier expansions of the oldforms. However, this would entail understanding the Fourier expansions of cusp forms at various cusps. For instance when $N=p$, one would need a Fourier expansion of $F|U(p)W_p$, where $F$ is a SK lift of level $1$. In principle this may be handled e.g., by considering the Rankin-Selberg convolution $R(F|\gamma, \Theta,s)$ for a suitable theta/Eisenstein series $\Theta$ (e.g.,) and then `transferring' the $\gamma$ via the integral representation onto a Siegel Eisenstein series, which we understand at the various cusps reasonably well. In any case, this will involve non-trivial arguments if it works. Surprisingly, even for $F \in \skk(1)$, it seems to be quite delicate to prove that $F, F|B_p, F|U_S(p)$ are linearly independent by only using their Fourier expansions.
\end{rmk}

\begin{lem} \label{same-old}
 With the above notation, one has 
    \begin{align}\label{SKoldasJacobi}
    \skk^{old}(N) = \sumn_{p|N} \left( \mc L_{\frac{N}{p}}(J_{k,1}^{cusp}(\frac{N}{p})) + \mc L_{\frac{N}{p}}(J_{k,1}^{cusp}(\frac{N}{p})) | W_p + + \mc L_{\frac{N}{p}}(J_{k,1}^{cusp}(\frac{N}{p})) | U_S(p) \right) .
    \end{align}
\end{lem}

\begin{proof}
Let us call the RHS of \eqref{SKoldasJacobi} as $V(\frac{N}{p})$.
    We notice from the description of oldforms from subsection~\ref{skold} and the REMARK after \cite[Theorem~4.1]{Ibu-SK}, that $V(\frac{N}{p}) \subset \skk^{old}(N)$.

For the converse, let $G \in \skk^{old}(N)$, which we can assume to be a Hecke eigenform. From definition \eqref{new-old-def}, and distinguishing via eigenvalues, we can write 
\begin{align} \label{old-new-exp}
    G= \sumn_{M|N, M<N} F_M |\prod \nolimits_{p|\frac{N}{M}} R(p),
\end{align}
where $F_M$  runs over newforms of level $M$ and for each $p|NM^{-1}$, $R(p)$ are the $4$ old-operators. 
Since $F_M$ is a newform, it follows from Lemma \ref{basisnew}  that $F_M$ is an EZI lift of a Jacobi newform of level $M$. Thus by multiplicity-one in the newspace and distinguishing via eigenvalues, we see that only one $M$ can survive in the above sum.  Writing out the terms in \eqref{old-new-exp} and using Lemma \ref{dim=3}
then gives the lemma.
\end{proof}

\begin{cor}
Let $N$ be square-free. From the above description, we can inductively write
\begin{align}
   \skkn = \sumn_{M|N} \sumn_{d,e\colon de|\frac{N}{M}} \skk^{new}(M)|W_d U_e.
\end{align}
Therefore by (generalized) M\"obius inversion, one obtains the relation 
\[ \dim \skk^{new}(N) = \sumn_{M|N} \beta(N/M) \dim \skkm,\]
where $\beta(n)$ are defined by the generating series $\displaystyle \sumn_{n \ge 1} \beta(n)n^{-s} = \zeta(s)^{-3}$, and thus also a formula for $\dim \skk^{old}(N)$.
\end{cor}

\subsection{Distinguishing EZI lifts via Maa{\ss}-Ibukiyama relations}\label{MaassDist}
We know that the space of EZI lifts sits inside the space of SK lifts and satisfies the so-called Maa{\ss}-Ibukiyama relations (cf. \cite[Prop.~3.8]{Ibu-SK}). In this subsection, we show that this condition actually characterizes them among all SK lifts of level $N$.

\begin{lem}\label{lem:Wpnonlift}
    Let $p|N$ and $F\in \mrm{SK}_k^J(N/p)$ be non-zero. Then $F|W_p\not\in \mrm{SK}_k^J(N)$.
\end{lem}
\begin{proof}
    From \cite{heim2017maass}, $G\in \mrm{SK}_k^J(N)$ iff the Maa{\ss} relations \eqref{Maassrel} are satisfied. 
Thus, it is enough to show that the Fourier coefficients of $F|W_p$ do not satisfy \eqref{Maassrel}. From \eqref{bp=wp} applied to $N/p$, let us recall that $\gamma W_p= B_p $ for some $\gamma\in \Gamma_0(N/p)$ and 
thus $F|W_p= F|B_p= p^k F(pZ)$. Therefore, its Fourier expansion can be written as 
 \begin{equation}
     F|W_p = \sumn_{T}A^*_{F}(T)e(\tr(TZ)),
 \end{equation}
where $A^*_{F}(T)= p^k A_{F}(T/p)$ when $p| \mf c(T)$ and is zero otherwise. If $F|W_p$ satisfied the Maa{\ss}-Ibukiyama relations, then we would have
\begin{align}
    A^*_{F}(n,r,m)= \sum_{d|(n,r,m), (d,N)=1} d^{k-1}  A^*_{F}(\frac{nm}{d^2},\frac{r}{d},1).
\end{align}
By the aforementioned property of $A^*_{F}(T)$, each term on the RHS above is zero. Since this is true for any $A^*_{F}(T)$, we see that $F|W_p=0$, which is impossible. This gives us the desired conclusion.
\end{proof}

\begin{lem} \label{ezi-charac}
    An element $F \in \skkn$ is an EZI lift (i.e., lies in $\mrm{SK}_k^J(N)$) if and only if it satisfies the Maa{\ss}-Ibukiyama relations \eqref{Maassrel}.
\end{lem}
\begin{proof}
    Let $F \in \skkn$ be an EZI lift. Then from \cite{heim2017maass}, $F$ satisfies the Maa{\ss}-Ibukiyama relations.

    For the converse, from Lemma \ref{basisnew} $ \skk^{new}(N) = \skk^{J,new}(N)$. Thus it is enough to consider the oldspace $\skk^{old}(N)$. From Lemma \ref{dim=3} and \ref{same-old}, we have 
    \begin{equation}
        \skk^{old}(N) =\sumn_{p|N} \left( \mc L_{\frac{N}{p}}(J_{k,1}^{cusp}(\frac{N}{p})) + \mc L_{\frac{N}{p}}(J_{k,1}^{cusp}(\frac{N}{p})) | W_p + \mc L_{\frac{N}{p}}(J_{k,1}^{cusp}(\frac{N}{p})) | U_S(p)\right).
    \end{equation}
From \cite{Ibu-SK}, we know the relation $\mc L_N (\phi)= \mc L_{N/p}(\phi)-p^{-1}\mc L_{N/p}(\phi)|W_p $. Thus combining with the fact that $\mc L_N$ and $U_S(p)$ operators commute, we can write
\begin{equation}
    \skk^{old}(N) =\sumn_{p|N}\left( \mc L_N(J_{k,1}^{cusp}(\frac{N}{p})) +  \mc L_N(U_J(p)(J_{k,1}^{cusp}(\frac{N}{p}))) \right)+\sumn_{p|N}\mc L_{\frac{N}{p}}(J_{k,1}^{cusp}(\frac{N}{p})) | W_p.
\end{equation}
By definition, the first two components satisfy the Maa{\ss} relation \eqref{Maassrel} and all the EZI lifts are contained inside them. Clearly the third component $\sum_{p|N}\mc L_{\frac{N}{p}}(J_{k,1}^{cusp}(\frac{N}{p})) | W_p$ is non-zero. Thus, it is enough to show that the third component doesn't satisfy the relations, and this is given by Lemma \ref{lem:Wpnonlift}. 
\end{proof}

\subsection{Intrinsic Maa{\ss} relations} \label{intr-maass}
In \cite{marzec2021maass}, the Maa{\ss}-relations for higher levels were studied via representation theory and
it was shown that any $G\in \skkn$ satisfies a more generalized  Maa{\ss} relation, which we now describe in our setting -- for a better perspective. For any positive integer $L$ having the same prime divisors as that of $N$,  and $T=\smat{n}{r/2}{r/2}{m}$ with $(n,r,m, N)=1$, it was proved in \cite{marzec2021maass} that if $G\in \skkn$, then
\begin{equation}\label{genMaassrel}
       A_G(LT):= A_G\left(L\smat{n}{r/2}{r/2}{m}\right)= \sum_{d|(n,r,m)} d^{k-1} A_G\left(L\smat{nm/d^2}{r/2d}{r/2d}{1}\right).
\end{equation}
However, it is not clear if these relations are also sufficient for some $F \in \sktwon$ to be a SK lift.

In contrast to Lemma~\ref{ezi-charac}, we show here that  $F|W_p$, with $F$ as in Lemma \ref{lem:Wpnonlift}, satisfies the generalized Maa{\ss}-relations of \cite{marzec2021maass} (viz., \eqref{genMaassrel}), as predicted by it. In this case, it is enough to show for $L=p^\ell$, $p|N$. As above, let $A_F^*(T)$ denote the Fourier coefficients of $F|W_p$. Then for $T=\smat{n}{r/2}{r/2}{m}$ with $(n,r,m, N)=1$, we have
\begin{align}
 A_F^*\left(p^\ell\smat{n}{r/2}{r/2}{m} \right)&=p^{k}A_F\left(p^{\ell-1}\smat{n}{r/2}{r/2}{m} \right)=p^k \sum_{d|p^{\ell-1}(n,r,m)}d^{k-1}A_F\left(\smat{p^{2\ell-2}nm/d^2}{p^{\ell-1}r/2d}{p^{\ell-1}r/2d}{1} \right)\\ 
 &= p^k\sum_{d|(n,r,m)}d^{k-1}\sum_{t=0}^{\ell-1}p^{t(k-1)}A_F\left(\smat{p^{2\ell-2-2t}nm/d^2}{p^{\ell-1-t}r/2d}{p^{\ell-1-t}r/2d}{1} \right)\\
 &=  p^k\sum_{d|(n,r,m)}d^{k-1} A_F\left(p^{\ell-1}\smat{nm/d^2}{r/2d}{r/2d}{1} \right)\\
 &=\sum_{d|(n,r,m)}d^{k-1} A_F^*\left(p^{\ell}\smat{nm/d^2}{r/2d}{r/2d}{1} \right).
\end{align}

\section{An orthonormal basis for SK lifts of level \texorpdfstring{$p$}{p}.}\label{orthbasis}

For any $\phi\in J_{k,1}^{cusp}$, from Lemma \ref{dim=3},  we know that the oldspace corresponding to $\phi$ is $3$ dimensional. Let us denote this space by $\skk^{old}(p; \phi)$. In view of the sup-norm problem, our first goal is to obtain an explicit orthonormal basis for SK lifts of level $p$ which is $W_p$ invariant. In this direction, for $\phi \in J_{k,1}^{cusp}$ and operators $R_i(p)$ introduced in subsection \ref{Mpintro} we put 
\begin{equation}
F_{i,\phi}= \mc L_1(\phi)|R_i(p) \text{ and } G_{i,\phi}= \mc L_p(\phi)|R_i(p).
\end{equation}
Whenever $\phi$ is fixed and there is no ambiguity, we drop $\phi$ from the notations and write as $F_i$ and $G_i$.
The following lemma allows us to freely switch between the forms $F_{j,\phi}$ and $G_{j,\phi}$ in subsequent calculations.
\begin{lem} \label{FG}
    With the above notation, the $\C$-spans of $\{ F_{1,\phi}, F_{2,\phi}, F_{3,\phi}, F_{4,\phi} \}$ and $\{ G_{1,\phi}, G_{2,\phi}, G_{3,\phi}, G_{4,\phi} \}$ are equal.
\end{lem}
\begin{proof}
We first note that $F(pZ) = p^{-k} F|B_p$. We claim that $G_1= F_1-p^{-1} F_3, G_2= F_2 - p^{-1} F_1, G_3= F_3 - p^{-1} F_1, G_4=F_4 - p^{k-1} F_3$.
 
We start from the observation $F|B_p = F|W_p$ since $F$ has level $1$ and from the relation given in \cite{Ibu-SK}:
\begin{align}\label{G1}
   G_1 = F_1 - p^{k-1} F_1(p Z) = F_1 - p^{-1} F_1 | W_p = F_1-p^{-1} F_3.
\end{align}
Next, one computes
    \begin{align}\label{G2}
        G_2= (F_1-p^{-1} F_3)|U_S(p) = F_2 - p^{-1} F_1|W_p|U_S(p) = F_2 - p^{-1} F_1|B_p|U_S(p) = F_2 - p^{k-1} F_1.
    \end{align}
For $G_3, G_4$ one similarly gets
    \begin{equation} \label{G3}
        G_3= (F_1-p^{-1} F_3)|W_p= F_3 - p^{-1} F_1; \q G_4 = (F_2 - p^{k-1} F_1)| W_p = F_4 - p^{k-1} F_3 . \qedhere
    \end{equation}
\end{proof}

\begin{lem}\label{lemskkpphi}
    The set $\{ G_{1,\phi}, G_{2,\phi}, G_{3,\phi} \}$ is a basis for the $\phi$-oldspace $\skk^{old}(p; \phi)$ in $\skkp$.
\end{lem}
\begin{proof}
In principle, this follows from \lemref{Mprankge3}. But the following proof is more illuminating -- and also helpful while constructing a $W_p$-invariant old-basis.

    From Lemma \ref{same-old},  we know that $\skk^{old}(p; \phi)$ is spanned by $\{ F_1, F_2, F_3, F_4 \}$, which equals the span of $\{ G_1, G_2, G_3, G_4 \}$ by \lemref{FG}. Moreover from \lemref{dim=3}, we know that $\dim \skk^{old}(p)(\phi)=3$ and from \cite{Ibu-SK}, that $G_1,G_2$ are linearly independent. 

    If we put $V=\C G_1+\C G_2$ and $W= \C G_1|W_p+\C G_2 |W_p$, then from the above we conclude that $\dim V=\dim W=2$, and that $\dim(V+W)=3$ -- which implies that $\dim(V \cap W)=1$. Now let $V \cap W= \C v$. We can then write
    \begin{align} \label{wp-inv-G}
        v= c_1 G_1+c_2 G_2 = d_1 G_1 |W_p +d_2 G_2 |W_p 
    \end{align}
for some scalars $c_j,d_j$. By applying $W_p$ on both sides, we obtain, by the uniqueness of $v$ up to scalars, that $d_j = \alpha c_j$ for some $\alpha \neq 0$ and $j=1,2$. Then \eqref{wp-inv-G} reads
\begin{align}
    v= G_{12} = \alpha G_{12}|W_p, \q \q (G_{12} :=c_1 G_1+c_2 G_2  ).
\end{align}
Clearly $\alpha^2=1$. From this one can also write
\begin{align} \label{G-reln1}
    c_1(G_1 \pm G_3) = -c_2(G_2 \pm G_4).
\end{align}
Clearly, $c_2 \neq 0$ as otherwise we must have $F_1 = \pm F_3$, which is impossible. Then note that $G_3 \not \in \C G_1 + \C G_2$ as otherwise from \eqref{G-reln1}, we would have $G_4 \in \C G_1+ \C G_2$ which would imply $\dim (V+W)=2$.
Therefore, $\{ G_1, G_2, G_3 \}$ are linearly independent. Of course, we can also replace $G_3$ with $G_4$. This proves the lemma.
\end{proof}
\begin{prop}\label{prop:oldbasis}
    Let $\mc B^J_{k,1}=\{\phi\}$ be an orthogonal  basis of Hecke eigenforms for $J_{k,1}^{cusp}$. Then the set $ B^{SK, old}:= \{G_{1,\phi}, G_{2,\phi}, G_{3,\phi}: \phi\in \mc B^J_{k,1}\}$ is a basis for the space $\skk^{old}(p)$. The spaces $\skk^{old}(p; \phi)$ with $\phi$ as above, are mutually orthogonal.
\end{prop}
\begin{proof}
From Lemma \ref{lemskkpphi}, we know that $\{G_{1,\phi}, G_{2,\phi}, G_{3,\phi}\}$ is a basis for $\skk^{old}(p; \phi)$ and we have that $\skk^{old}(p)=\sum_{\phi\in  \mc B^J_{k,1}}\skk^{old}(p; \phi)$. That this is a direct sum follows from the statement about orthogonality. But this in turn, follows immediately from multiplicity-one in degree $1$, the fact that the operators $U_S(p), W_p$ commute with Hecke operators $T_S(q)$ with $q \ne p$ and the equivariance relations from \cite[Theorem 4.1]{Ibu-SK}.
\end{proof}
\begin{rmk}
    From \propref{prop:oldbasis}, we can also obtain (the known fact from subsection \ref{sec:SKintro}) that 
    $G_{3,\phi}$ is not an EZI lift of a Jacobi cusp form of level $p$.
\end{rmk}

\subsection{An orthonormal basis for oldspace:} To obtain an orthonormal basis for the oldspace of level $p$, we make use of the basis obtained in Proposition \ref{prop:oldbasis}. In this direction, we first calculate the inner products $\lan G_i, G_j\ran$.
For this, we use Lemmas~\ref{p21},~\ref{L1Uswp},~\ref{L1up} and \ref{lemUpUp} and the relations \eqref{G1}, \eqref{G2} and \eqref{G3}, which say that $G_1= F_1-p^{-1} F_3$, $G_2=F_2 - p^{k-1} F_1 $ $G_3= F_3 - p^{-1} F_1$. 
We summarize these inner product relations in the following Proposition.
\begin{prop}\label{lem:LpIP}
    Let $\phi\in J_{k,1}^{cusp}$ and let $F_1=\mc L_1(\phi)$, $G_1=\mc L_p(\phi)$. Put $G_{2}:= G_{1}|U_S(p)$ and $G_{3}:= G_{1}|W_p$. Further, let $F_1$ be an eigenfunction of the Hecke operators $T_S(p)$ and $T_S'(p)$  with the eigenvalue $\lambda_{F_1}(p)$ and $\lambda_{F_1}'(p)$, respectively. Then we have
    \begin{align}
    \lan G_1, G_1\ran_p &= \left(1+\frac{1}{p^2}-\frac{2\lambda_{F_1}(p)}{\mu(p)p^{k-2}}\right)\lan F_1, F_1\ran_p; \label{G1G1norm}\\
    \lan G_1, G_2\ran_p&=\left(\frac{\lambda_{F_1}'(p)}{p^{k-1}(p^2+1)}-\frac{\lambda_{F_1}(p)^2}{\mu(p)p^{k-2}}+\frac{p\lambda_{F_1}(p)}{p+1}-\frac{p^{k-3}(p^4+p^2+1)}{p^2+1}\right)\lan F_1, F_1\ran_p; \label{G1G2norm} \\
    \lan G_1, G_3\ran_p &= \left(\frac{p^{1-k}\lambda_{F_1}(p)}{p+1}-\frac{2}{p}\right) \lan F_1, F_1\ran_p;\\
    \lan G_2, G_2 \ran_p &= \left(\frac{p^2(p-1)}{\mu(p)}\lambda_{F_1}(p)^2-\frac{2p^{k+2}}{\mu(p)}\lambda_{F_1}(p)+p^{2k-4}+p^{2k-2}\right) \lan F_1, F_1\ran_p;\label{G2G2norm}\\
    \lan G_2, G_3\ran_p&= \left(\frac{\lambda_F(p)^2}{\mu(p)p^{k-3}}- \frac{2p^2\lambda_{F_1}(p)}{\mu(p)}-\frac{\lambda_F'(p)}{p^{k-2}(p^2+1)}+\frac{p^{k-2}(p^2+2)}{p^2+1}\right)\lan F_1, F_1 \ran_p.
    \end{align}
\end{prop}

\begin{rmk}
   The inner products $\lan G_i, G_i\ran_p$ in Proposition \ref{lem:LpIP} are indeed positive. To see this, recall from \eqref{Heckrel} that $\lambda_{F_1}(p)= p^{k-1}+p^{k-2}+\lambda_\phi(p)$. Then one can verify that $\lan G_1, G_1\ran_p\asymp \lan F_1, F_1\ran_p$ which is clearly positive. For $G_2$, we have $\lan G_2, G_2\ran_p\asymp \mu(p)^{-1}(p^2(p-1)\lambda_\phi(p)^2- 2p^{k}\lambda_\phi(p)+p^{2k-1} )\lan F_1, F_1\ran_p $. Since $|\lambda_\phi(p)|\le 2 p^{k-3/2}$, we see that $\lan G_2, G_2\ran_p>0$.
\end{rmk}

\subsection{\texorpdfstring{$W_p$}{Wp} invariant old-basis} \label{wp-inv}
It is desirable to obtain an orthogonal basis consisting of eigenvectors of $W_p$. This is, in particular, quite useful in the context of the sup-norm problem, as it helps us to narrow the realm of search for large values in convenient regions of $\htwo$. 

As in Proposition \ref{prop:oldbasis}, let $\mc B_{k,1}^J$ be an orthonormal basis of Hecke eigenforms for $J_{k,1}^{cusp}$. For $\phi\in \mc B_{k,1}^J$, let $F_{1,\phi}:=\mc L_1(\phi)$ and $G_{1,\phi}:=\mc L_p(\phi)$. By construction, clearly $W_p$ acts on $\skk^{old}(p)$. Moreover, 
\begin{equation} \label{g-phi}
    G_{\pm, \phi}:=G_{1,\phi}  \pm G_{3,\phi}
\end{equation}
are orthogonal and are eigenvectors for $W_p$ with eigenvalues $\pm 1$ respectively. To obtain the third eigenvector, let 
\begin{align} \label{h-phi}
H_\phi:= G_{2,\phi} - a_{+, \phi} G_{+, \phi} - a_{-, \phi} G_{-, \phi}
\end{align}
be orthogonal to both $G_{+, \phi}, G_{-, \phi}$. Therefore,
\begin{align}\label{apm}
    a_{\pm, \phi} = \frac{\lan G_{2, \phi}, G_{\pm, \phi} \ran_p }{\lan  G_{\pm, \phi}, G_{\pm, \phi} \ran_p }.
\end{align}
Next note that $H_\phi |W_p$ is also orthogonal to both $G_{+, \phi}, G_{-, \phi}$. Since the orthogonal complement of $<G_{+, \phi}, G_{-, \phi}  >$ is one dimensional, we must have $H_\phi|W_p = \pm H_\phi$, giving $H_\phi$ to be the last of our desired basis elements. We do not work out the exact sign, as this is not required for us.

From Proposition \ref{lem:LpIP}, we see that 
\begin{align}\label{normpm}
    \lan G_{\pm,\phi}, G_{\pm,\phi} \ran_p \asymp \lan G_{1,\phi}, G_{1,\phi} \ran_p\pm \lan G_{1,\phi}, G_{3,\phi} \ran_p\asymp \lan F_{1,\phi}, F_{1,\phi}\ran_p;\\
    |\lan G_{2,\phi}, G_{\pm, \phi} \ran_p| \ll  p^{k-3/2} \lan F_{1,\phi}, F_{1,\phi}\ran_p;\q \lan H_\phi, H_\phi \ran_p \asymp p^{2k-4}\lan F_{1,\phi}, F_{1,\phi}\ran_p.
\end{align}
We also have the relation \cite[pp. 551]{KS}: $ \lan F_1, F_1\ran_{\sptwo} = c_k \lan \phi, \phi \ran_{\sltwo}$. Thus, the level $p$ norms can be related as $\lan F_1, F_1\ran_{p}= (p^2+1) c_k \lan \phi, \phi \ran_{p}$.

This takes care of the oldspace. For the newspace, all SK newforms are EZI lifts of unique Jacobi newforms of level $p$. That each of these are invariant under $W_p$ follows from multiplicity-one in a standard manner.
We summarize the above discussion in the following proposition.

\begin{prop}\label{oldbasis}
    Let $\mc B_{k,1}^J$ denote an orthonormal basis of Hecke eigenforms for $J_{k,1}^{cusp}$. Then the set
\begin{equation}
    \mc B ^{SK, old}(p)=\left\{ G_{+,\phi}, G_{-,\phi}, H_\phi:  \phi\in \mc B_{k,1}^J\right\}
\end{equation}
is an orthogonal basis of eigenvectors of $W_p$ for the oldspace $\skk^{old}(p)$ with $G_{+,\phi}, H_\phi$ as in \eqref{g-phi}, \eqref{h-phi} respectively.  Moreover, we have that 
\begin{equation}
    \lan G_{\pm,\phi}, G_{\pm,\phi} \ran_p \asymp \lan F_{1,\phi}, F_{1,\phi}\ran_p\asymp p^3 \lan \phi, \phi \ran_{1};\q  \lan H_\phi, H_\phi \ran_p \asymp p^{2k-4}\lan F_{1,\phi}, F_{1,\phi}\ran_p\asymp p^{2k-1}\lan \phi, \phi \ran_{1}.
\end{equation}

For the newspace, $\mc B ^{SK, new}(p)=\{ \mc L_p(\phi) \colon \phi \in \jkp^{new} \}$ is the $W_p$ invariant basis.
 \end{prop}

\section{Sup-norm of Jacobi forms} \label{jacobi-sup-norms}
The purpose of this section is twofold:

(i) To study the size of $\jkn$ -- and this is done by naturally reducing the problem to one about half-integral weight forms; and

(ii) Prove a non-trivial bound for a Hecke eigenform $\phi \in \jkn$ by the same strategy as (i) above, but there are quite subtle issues here -- things do not follow in a straightforward way, and one has to go at the root of the half-integral case treated by Kiral \cite{Kir} and account for the extra subtlety mentioned above.

\subsection{Theta decomposition of Jacobi forms} \label{theta-decom}
Let $N$ be an odd, square-free integer and $\phi\in\jkn$ be an eigenform of all Hecke operators away from the level $N$. 
The periodicity properties of $\phi$ allows one to write (cf. \cite{EZ})
\begin{align}
\phi = h_0(\tau) \theta_0(\tau,z) +h_1(\tau) \theta_1(\tau,z).
\end{align}
where the classical Jacobi theta functions $\theta_i$ of weight $1/2$ and index $1$ are given by
\begin{equation}
\theta_i(\tau,z)=\sumn_{n\equiv i\bmod 2} e( n^2\tau/4)e(nz).
\end{equation}

Then the Eichler-Zagier isomorphism is the map $EZ: \phi \mapsto h(\tau) =h_0(4 \tau) +h_0(4 \tau)$. Crucial to our approach is a convenient description of theta components $h_j$ of $\phi$ in terms of the half-integral weight cusp form $h$. This is clear when $N=1$ from \cite{EZ}. But it is not as trivial when $N>1$; we refer the reader to \cite[Satz 8]{Kr}. For convenience, we recall the result here. 
\begin{prop} \label{krprop}
    $h_0 = \frac{1}{2} \big ( h(\frac{\tau-1}{4}) + h(\frac{\tau+1}{4} ) \big); \q h_1 = \frac{1}{2i} \big( h(\frac{\tau-1}{4}) - h(\frac{\tau+1}{4} ) \big)$.
\end{prop}
Further, from \cite[Satz 9]{Kr} we see that $h\in S_{k-1/2}^+(4N)$ -- the Kohnen's plus space of level $4N$ (see \cite{kohnen1982newforms}). 
From Proposition \ref{krprop}, we can therefore write
\begin{align} \label{kr-reln}
\phi &= \frac{1}{2} \left( h(\frac{\tau-1}{4}) (\theta_{0}(\tau,z) - i \theta_{1}(\tau,z) ) +h(\frac{\tau+1}{4}) (\theta_{0}(\tau,z) + i \theta_{1}(\tau,z)  ) \right).
\end{align}

\subsubsection{A convenient fundamental domain}

To start with, note that the invariant function attached to $\phi\in J_{k,m}^{cusp}(N)$, written as $\widetilde{\phi}$ is defined as
\begin{align}
\widetilde{\phi} := v^{k/2} e^{- 2 \pi m y^2/v} |\phi|. \label{inva}
\end{align}
Let us now fix a fundamental domain $\mc F_N \subset \h$ for the action of $\Gamma_0(N) $ on $\h$ defined by 
\begin{align} \label{fndef}
\mc F_N = \cup_{q \mid N} \cup_{ 0 \le j \le d-1} A_{q,j} \mc F_1;
\end{align}
where $\mc F_1$ is the classical fundamental domain for $\sltwo$ on $\h$ (contained in the Siegel set $\{ z \in \h \mid \Im(z) \ge \sqrt{3}/2$), and $A_{q,j}$ are a set of left-coset representatives of $\Gamma_0(N)$ in $\sltwo$ given by (cf. \cite{AU}):
\begin{align} 
\sltwo = \cup_{q\mid N} \cup_{ 0 \le j \le q-1}  \Gamma_0(N)A_{q,j} ; \q A_{q,j} = \begin{pmatrix} \beta & 1 \\ q \gamma & \frac{N}{q}  \end{pmatrix} \begin{pmatrix}  0 & -1 \\  1 & j  \end{pmatrix} \q (\beta \frac{N}{q} - q \gamma =1). \n
\end{align}
Next, note that for any function $\Psi$ on $\h \times \C$, $M \in \sltwor $ and $k \ge 1$, one has the relation:
\begin{align} \label{inv}
\left( v^{k/2} e^{- 2 \pi y^2/v} |\Psi| \right) \circ (M(\tau,z))  = v^{k/2} e^{- 2 \pi y^2/v} | (\Psi |_{k,1} M ) (\tau,z)|.
\end{align}
This can be checked from \cite{sko-zag} for $\Psi=1$ and $M \in \sltwor$, but our case follows immediately from this. 

\begin{lem} \label{cosetj}
 (i) A complete set of coset representatives of $\Gamma_0(N)^J$ in $\Gamma_1^J$ can be taken to be $(\gamma, [0,0]) $ where $\gamma$ runs over $\Gamma_0(N) \backslash \Gamma_1$. 
 
 (ii) As a fundamental domain $\mc F_{N}^J$ can be taken as $\mc F^J_N= \bigcup_\gamma (\gamma, [0,0]) \cdot \mc F_1^J$ with $\gamma$ as in (i) above.
\end{lem}
\begin{proof}
    Clearly it is enough to prove (i). If $(M',X') \in \Gamma_1^J$, then we can always left multiply it with $(1_2,X)$, where $X$ is chosen such that $XM'+X=0$. Thus we can assume $X'=0$. Then left multiplication by a suitable element $(g,0)$ with $g \in \Gamma_0(N)$ gives (i).
\end{proof}

We now come back to our goal of bounding $\tilde{\phi}$.
First, by invariance under $\Gamma_0(N)$ along with \eqref{inv} and \lemref{cosetj} we have
\begin{align} \label{supM}
    \sup\nolimits_{(\tau,z) \in  \h \times \C}  \widetilde{\phi}(\tau,z) = \sup\nolimits_{(\tau,z) \in  \mc F^J_N}  \widetilde{\phi}(\tau,z) =  \max\nolimits_{q,j} \sup\nolimits_{(\tau,z) \in  \mc F^J_1}  T_{q,j},
\end{align}
where the quantities $T_{q,j}$ are defined by
\begin{equation}
 T_{q,j}:=   v^{k/2} e^{- 2 \pi y^2/v} | (\phi |_{k,1} A_{q,j} ) (\tau,z)|,
\end{equation}
and $A_{q,j}$ runs over a set of coset representatives of $\Gamma_0(N)$ in $\sltwo$. Then $T_{q,j}$ is bounded by
\begin{align} \label{h-th-bd1}
  \sup\nolimits_{(\tau,z) \in \mc F^J_1 } \, v^{k/2}   e^{- 2 \pi y^2/v} \sumn_\pm \big| h(\frac{\tau \pm 1}{4}) (\theta_{0}(\tau,z) \pm  i \theta_{1}(\tau,z) )|_{k,1} A_{q,j}  \big|.
\end{align}
Put $\kappa := k-1/2$. Then the quantities in $\sum_\pm | \cdots |$ above are bounded by the sums of the type
\begin{align} \label{h-th-bd2}
    \big| h(\frac{\tau \pm 1}{4}) |_{\kappa} A_{q,j} \big|\cdot \big| (\theta_{0}(\tau,z) \pm i \theta_{1}(\tau,z) )|_{\frac{1}{2},1} A_{q,j} \big|.
\end{align}
Therefore we have reduced the problem to $h$, but first we have to tackle the theta series.

\subsubsection{Handling the theta series:} Let us put 
\begin{align}
    \vartheta(\tau,z) = |\theta_{0}(\tau,z) |^2 + |\theta_{1}(\tau,z) |^2.
\end{align}
Although the invariant function $v^{1/2} e^{- 4 \pi y^2/v} |\theta_j|^2$ is not bounded in $\h \times \C$, we shall now show that the function $e^{- 4 \pi y^2/v} \vartheta(\tau,z)$ is bounded in $\mc F_1$, and that will be enough for us.

\begin{lem} \label{varth}
  $e^{- 2 \pi y^2/v} \vartheta(\tau,z)$ is bounded by an absolute constant on $\mc F_1^J$.
\end{lem}

\begin{proof}
From the calculations in \cite[Section 4.2]{PASD}, we see that $\displaystyle e^{- 2 \pi y^2/v} \vartheta(\tau,z)\ll (1+v^{-1/2})$.
The lemma now follows, since $\tau\in \mc F_1$.
\end{proof}

In view of \eqref{h-th-bd1} and \eqref{h-th-bd2}, we have to bound
$\displaystyle \max \nolimits_{q,j} \sup \nolimits_{\tau \in \mc F_1} \left|(\theta_{0}(\tau,z) \pm i \theta_{1}(\tau,z) ) |_{k} A_{q,j}\right| $.
We now use the transformation properties of $\theta_{\mu}$ under $\sltwo$ (more precisely under its double cover $\widetilde{\sltwo}$; see \cite{Kr}, \cite{Sko}) as follows:
\begin{align}
\theta_{\mu}(\tau,z) |_{1/2} (A, \omega(\tau)) = c_1 \theta_{0}(\tau,z)  + c_2 \theta_{1}(\tau,z), 
\end{align}
where $A=\psmb a & b \\ c & d \psme \in \sltwo, \omega(\tau)^2=c\tau+d$ and $c_1,c_2$ are complex numbers with absolute value less or equal to $1$. From this, and \lemref{varth} it is now clear that on $\mc F_1^J$,
\begin{align} \label{th-bd}
    \max \nolimits_{q,j} \sup \nolimits_{\tau \in \mc F_1} \left|(\theta_{0}(\tau,z) \pm i \theta_{1}(\tau,z) ) |_{k} A_{q,j}\right| \ll \left( e^{- 4 \pi y^2/v} \vartheta(\tau,z) \right)^{1/2} \ll 1.
\end{align}
In the following subsection, we reduce the problem of bounding the sup-norm of a Jacobi form to that of a half-integral weight form using \propref{krprop} and the Atkin-Lehner operators for the group $\Gamma_0(4N)$.

\subsubsection{Handling the half-integral weight form:}
We only handle the `$+$' sum in \eqref{h-th-bd2}, the other case is entirely the same. We have
\begin{align}
  \q  \big| h(\frac{\tau \pm 1}{4}) |_{\kappa} M \big| \ll \big| h|_\kappa\smat{1}{1}{0}{4}\smat{\beta\frac{N}{q}}{1}{N\gamma}{\frac{N}{q}} \smat{\frac{q}{N}}{0}{0}{1}\smat{0}{-1}{1}{j}   \big|=  \big| h|_\kappa\smat{1}{0}{0}{4} \smat{\beta\frac{N}{q}}{1}{N\gamma}{\frac{N}{q}}\cdot g \cdot  \smat{\frac{q}{N}}{0}{0}{1}\smat{0}{-1}{1}{j} \big|, \label{miracle?}
\end{align}
for some $g \in \Gamma_0(N)$.
At this point, we note that it is possible to choose $\beta$ such that $\beta \equiv 0 \bmod 4$ and $\beta \frac{N}{q} \equiv 1 \bmod q$, by the Chinese remainder theorem. Thus, putting $\beta=4\beta'$, we see that $\smat{1}{0}{0}{4}\smat{\beta\frac{N}{q}}{1}{N\gamma}{\frac{N}{q}} = \smat{\beta'\frac{4N}{q}}{1}{N\gamma}{\frac{4N}{q}}  = W(\frac{4N}{q}, 4N)$, the Atkin-Lehner involution at $4N/q$ of level $4N$. This is well-defined since $N$ is odd and square-free. Moreover, $h$ being a Hecke eigenform, so is  
\[ h_{d}:=h|W(d, 4N) , \text{ for any } \,  d|4N.\] Therefore, 
\eqref{miracle?} can be written as 
\begin{align} \label{h1-g}
  \big| h_{4N/q} |_\kappa g \smat{\frac{q}{N}}{0}{0}{1}\smat{0}{-1}{1}{j}   \big|.
\end{align}
The next step is to write $g$ in terms of coset representatives $\Gamma_0(4N)\backslash \Gamma_0(N)$. Let us put
\begin{align}
    A:= \smat{1}{0}{N}{1} ; \, B:= \smat{\beta}{1}{N\gamma}{4},
\end{align}
where $4 \beta - N \gamma=1$ and $\beta$ even.

\begin{lem}
  The set $\{1, A, A^2, A^3,  B, A^2 B \}$ can be taken as a set of coset representatives $\Gamma_0(4N)\backslash \Gamma_0(N)$.
\end{lem}
\begin{proof}
A set of coset representatives for $\Gamma_0(4N)\backslash \Gamma_0(2N)$ can be taken to be $\{1,A^2 \}$ and those for $\Gamma_0(2N)\backslash \Gamma_0(N)$ as $\{1,B,A \}$, respectively. The lemma follows.
\end{proof}
Then \eqref{h1-g} can be written as (with $s=0,1,2,3$)
\begin{align}
    \left|h_{4N/q} |_\kappa A^s C_j\right| \text{ or } \left|h_{4N/q}|_\kappa BC_j\right| \text{ or } \left|h_{4N/q}|_\kappa A^2B C_j\right|,
 \end{align}
where $C_j= \smat{\frac{q}{N}}{0}{0}{1}\smat{0}{-1}{1}{j} = \smat{0}{-\frac{q}{N}}{1}{j} $. Now we observe that
\begin{align}
       A^s C_j &= W(4N,4N) \begin{pmatrix}   \frac{1}{4N} & \frac{j-sq}{4N} \\ 0 & \frac{q}{N} \end{pmatrix} \, (0 \le s \le 3),\\
       BC_j&= W(4,4N) W(4N,4N) \smat{\frac{1}{4N}}{\frac{j}{4N}}{0}{\frac{q}{4N}},\\
       A^2BC_j&= g'W(4,4N) W(4N,4N) \smat{\frac{1}{4N}}{\frac{-2q+j}{4N}}{0}{\frac{q}{4N}} \text{ for some } g'\in\Gamma_0(4N).
\end{align}
As a consequence, it is enough to consider the first two cases above. Moreover, $W(4,4N)W(4N,4N)=W(N,4N)$ -- Atkin-Lehner operator at $N$ on $\Gamma_0(4N)$. 

For $d|4N$, if we put 
\[ h'_d:=h_{4N/q} |W(d,4N),\] 
then after multiplication by $v^{\kappa/2}$, \eqref{h1-g} can be written as
\begin{align} \label{h4N}
  v^{\frac{\kappa}{2}} \left| h'_{4N} |_\kappa \begin{pmatrix}   \frac{1}{4N} & \frac{j-q}{4N} \\ 0 & \frac{q}{N} \end{pmatrix} \right|= ( \frac{vq}{4N^2})^{\frac{\kappa}{2}} (q/N)^{-\kappa} \left|h'_{4N} \big (\frac{\tau + (j-q)}{4q} \big )\right| = (\frac{v}{4q})^{\frac{\kappa}{2}} \left| h'_{4N} \big (\frac{\tau + (j-q)}{4q} \big )\right|  .
\end{align}
Similarly
\begin{equation}\label{hN}
    v^{\kappa/2} \left| h'_{N} |_\kappa \begin{pmatrix}   \frac{1}{4N} & \frac{j}{4N} \\ 0 & \frac{q}{4N} \end{pmatrix} \right|= (v \frac{q}{(4N)^2})^{\kappa/2} (q/4N)^{-\kappa} |h'_N \big (\frac{\tau + j}{q} \big )| = (\frac{v}{q})^{\kappa/2} | h'_N \big (\frac{\tau + j}{q} \big )|  .
\end{equation}
From \eqref{h4N} and \eqref{hN} we calculate that
\begin{align}
   & (\frac{v}{4q})^{\kappa/2 }  \cdot v^{1/4} \cdot \left|  h'_{4N} \big (\frac{\tau + (j-q)}{4q} \big ) \right|  = (\frac{v}{4q})^{\kappa/2 }  \cdot (\frac{v}{q})^{1/4} \cdot \left| h'_{4N} \big (\frac{\tau + (j-q)}{4q} \big ) \right| \cdot q^{1/4} \label{genie} \\
    &  \ll \sup \nolimits_{ v \ge 1/N} \, v^{\kappa/2+1/4} \left| h'_{4N} \right|  \cdot q^{1/4}. \nonumber
\end{align}
and similarly
\begin{equation}
    (\frac{v}{q})^{\kappa/2} v^{1/4}| h'_N \big (\frac{\tau + j}{q} \big )|\ll \sup\nolimits_{v \ge 1/N} \, v^{\kappa/2+1/4} |h'_N|\cdot q^{1/4}. \label{nf1}
\end{equation}
Further from the properties of the Atkin-Lehner operators $W(d,4N)$, we have that $\norm{h'_{4N}}=\norm{h_q}$, $\norm{h'_{N}}=\norm{h_{4q}}$ and $\norm{h}= \norm{h_d}$. Thus, from what has been done in the above two sections (see especially \eqref{supM}--\eqref{h-th-bd1}, \eqref{h4N}, \eqref{hN} and \eqref{th-bd}), it follows that (recall that $\kappa = k - 1/2$)
\begin{align} \label{Tqj}
  T_{q,j}  \ll \sup\nolimits_{ \tau \in \h} \, v^{\kappa/2+1/4} \left| h_{q} \right|  \cdot q^{1/4}. 
\end{align}
We note here that $h_q\in S_{k-1/2}^+\left(4N, \left(\frac{4q}{\cdot}\right)\right)$. 

At this point, let us mention that, if the factor $v^{1/4}$ was absent in the above, we could have directly invoked \cite{Kir}. Moreover, the bounds obtained after amplification in \cite{Kir} are only on the set $\mc S \subset \h$ defined below in \eqref{S-def}, so one can't quote them directly in \eqref{Tqj}.

\subsubsection{The final step:} \label{key-point}
The final step for us is to observe that the maximum of the function $v^{\kappa/2+1/4} \left| h_q \right|$ is also achieved on a more convenient subset $\mc S \subset \h$, defined as
\begin{equation} \label{S-def}
    \mc S=\{\tau\in\h : \im ( \gamma\tau) \le \im (\tau) \text{ for all } \gamma\in A_0(2N) \},
\end{equation}
where $A_0(2N)$ is the group generated by all the Atkin-Lehner operators and $\Gamma_0(2N)$. The set $\mc S$ was first defined and studied in \cite{HT}, and has been used in \cite{Kir} as well.

Next, for a Dirichlet character $\chi\bmod 4N$, let $\mc B(N, \chi)$ denote a basis of Hecke eigenforms for $S_\kappa(4N,\chi)$ and consider $\mc B(N)=\cup_\chi \mc B(N, \chi)$. Let $\mc B'(N)=\cup_{A\in A_0(2N)} \mc B(N)|_\kappa A$. We note here that the union is over finitely many $A$ since $A_0(2N)$ is generated by finitely many Atkin-Lehner operators and $\Gamma_0(2N)$.

Let $w \in \h$ and let $\delta \in A_0(2N)$ be such that $w=\delta (\tau)$ with $\tau \in \mc S$ defined above. 
\begin{align}
\im(w)^{\kappa/2+1/4} \left| h_q(w) \right| = \im(\delta \tau)^{1/4} \im(\tau)^{\kappa/2} \left| h_q|\delta(\tau) \right|.
\end{align}
But we note that $h_1 = h_q|\delta \in \mc B'(N)$. However, for any $\delta \in A_0(2N)$ and any $\tau \in \mc S$ one has $\im(\delta \tau) \le \im(\tau)$, whence
\begin{align} \label{S}
    \im(w)^{\kappa/2+1/4} \left| h_q(w) \right| \le \im(\tau)^{\kappa/2+1/4} \left| h_1(\tau) \right|
\end{align}
and this proves our assertion. Therefore, if we look at \eqref{Tqj} and \eqref{S}, our problem reduces to finding
\begin{align}
    \sup \nolimits_{ \tau \in \mc S} v^{\kappa/2+1/4} \left| H \right|  \cdot N^{1/4},
\end{align}
where $H$ is any element in $\mc B'(N)$. Our next proposition does this task.
\begin{prop}
  With the above notation, we have \, $ \sup_{ \tau \in \mc S} v^{\kappa/2+1/4} \left| H \right|  \cdot N^{1/4} \ll_\epsilon N^{- \frac{1}{36} +\epsilon}$.
\end{prop}
\begin{proof}
    We quote two inequalities from \cite{Kir} which we would use: the first is from the Fourier expansion and states in our situation that for any $\tau \in \h$,
    \begin{align} \label{fe}
        v^{\kappa/2+1/4} \left| H \right| \ll \frac{N^\epsilon }{N^{1/2} v^{1/4}},
    \end{align}
    where we assume $\norm{H}=1$. The second inequality comes from the amplified analysis of the geometric side of the Bergman kernel and states, in our situation, that for all $\tau \in \mc S$
    \begin{align} \label{bk}
        v^{\kappa/2+1/4} \left| H \right| \ll \frac{v^{1/4} N^\epsilon}{\Lambda^{1-\epsilon}} \left( \Lambda^{1/2}(1+ N^{1/2} v) + 2(1+ \Lambda^4 N^{1/2} v) \right),
    \end{align}
where $\Lambda>0$ is arbitrary. We choose $\Lambda = N^\beta$, $\beta>0$ to be determined later. Furthermore, for $\alpha>0$, we split $v$ into two regions: (i)  $\ge N^{-\alpha}$ and use it in \eqref{fe}
, and (ii)  $\le N^{-\alpha}$ and use it in \eqref{bk}. 

Equating the two terms $\Lambda^{1/2}$ and $\Lambda^4 N^{1/2}  v$ in \eqref{bk} gives us the equation
\begin{align}
    \frac{\beta}{2} = 4 \beta + \frac{1}{2} - \alpha \q \text{  or,  } \q \alpha - \frac{7 \beta}{2}= \frac{1}{2}.
\end{align}
Next, equating the contribution from \eqref{fe} and that from \eqref{bk} we get the equation
\begin{align}
    - \frac{1}{2} +\frac{\alpha}{4} = -\frac{\alpha}{4} - \beta + \frac{\beta}{2}\q \text{  or,  } \q \alpha + \beta =1.
\end{align}
    This gives $\alpha = \frac{8}{9}, \beta = \frac{1}{9}$. Finally
    \begin{equation}
        v^{\kappa/2+1/4} \left| H \right|  \cdot N^{1/4} \ll N^{\frac{\alpha}{4} - \frac{1}{2} + \frac{1}{4}} = N^{- \frac{1-\alpha}{4}} = N^{- \frac{1}{36} +\epsilon}. \qedhere
    \end{equation} 
\end{proof}

\subsection{Size of the space of Jacobi forms of index \texorpdfstring{$1$ }{1}and  level \texorpdfstring{$N$}{N}}
\begin{prop}
 $\sup(\jkn) \gg   1 $.
\end{prop}

\begin{proof}
Given that the proof will follow the argument via Poincar\'e series, we will be brief. For the lower bound we note that for any $v_0>0,y_0$,
\begin{align}
    \sup(\jkn) \ge  v_0^{k} e^{-4\pi   v_0^{-1}y^2_0 } e^{-4\pi v_0}\sum \nolimits_{\phi\in B^J(N)} |c_\phi(1,0)|^2.
\end{align}
We also have that $\sum \nolimits_{\phi\in B^J(N)}|c_\phi(1,0)|^2=  4^{k-3/2} \pi^{k-3/2} \Gamma(k-3/2)^{-1} \cdot p_{1,0}(1,0)$, where $p_{n,r}(n,r)$ is the $(n,r)$-th Fourier coefficient of the Jacobi Poincar\'e series $P_{n,r}(\tau,z)$. We take $v_0=k/4 \pi$ and $y_0=0$. This implies that 
$\sup(\jkn)  \gg  p_{1,0}(1,0) \gg 1$ if we invoke the bound for the Fourier coefficient of the Jacobi Poincar\'e series (see e.g., \lemref{poincareN}) and this finishes the proof.
\end{proof}

\subsubsection{Upper bound:}
We first note that in the previous section, the arguments up to the bound \eqref{Tqj} hold verbatim on average, i.e., one has
\begin{align} \label{extra-q}
    \sup \jkn \ll \max_{\chi \bmod N } \max_{q |N} \sup_{ v \ge 1/N} v^{\kappa+1/2} \sumn_{h_q \in \mc B(N, \chi)} \left| h_{q} \right (\tau)|^2  \cdot q^{1/2}.
\end{align}
Here we are tacitly using the facts that $h \mapsto h_q \mapsto h'_q$ are $L^2$ isometries (cf. two lines above \eqref{Tqj}) between the spaces $S_{k-1/2}^+\left( 4N \right)$ and $S_{k-1/2}^+\left(4N, \chi \right)$ and $S_{k-1/2}^+\left(4N, \chi' \right)$ respectively for some characters $\chi, \chi' \bmod N$. See \cite{kohnen1982newforms} for proofs of these facts.

When $v \ge 1$, we rely on the Poincar\'e series for the relevant space and can show that (see Section \ref{appendix})
\begin{align}
    \sumn_{h_q \in \mc B(N, \chi)}v^{\kappa+1/2}  \left| h_{q} \right (\tau)|^2 \ll v^{1/2-1} \ll 1.
\end{align}
When $1/N \le v \le 1$, we simply bound the extra $v^{1/2}$ by $1$ and use the fact that $\sup( S^+_\kappa(N,\chi)) \ll 1$ from Theorem \ref{skn-int}. To preserve the continuity of the arguments, we relegate the proof, which is very much analogous to integral weights, to Section \ref{appendix} at the end of the paper. From \eqref{extra-q} therefore we get \thmref{jk1-sup}. 

It is worth mentioning here that if one agrees to sacrifice $N^\epsilon$, then $\sup( S^+_\kappa(N,\chi)) \ll N^\epsilon$ follows easily via a standard application of the relevant large sieve inequality for Fourier coefficients from \cite{lam2014local}.
\section{Conjecture about the size of the Saito-Kurokawa lifts} \label{skN-conj}
For the space $\stwon$, one can propose a conjecture about its size rather easily by looking at a global lower bound by transferring the question to one about average sizes of squares of Fourier coefficients $\sum_F |a_F(T)|^2$ and using Petersson's trace formula. One gets this lower bound to be $1$ and expects this to be the upper bound as well. 

Things are more subtle for the thin subspace of SK lifts. In a related context, one can already see this in the case of index $p$ Jacobi newspace (cf. \cite[Proof of Prop.~5.1]{PASD}) where asymptotics of averages of central values of certain $L$-functions were needed as a substitute to the subtle issue of a Petersson formula for newforms. In this paper, we also (have to) study the new and old spaces separately -- and thus the above remarks apply.

\subsection{The lower bound}
The lower bound will follow from the well-known inequality (follows from Parseval's inequality, for instance):
for any $Y_0>0$ and $T_0 \in \Lambda_2$ we have
\begin{align} \label{sklbd}
    \sup(SK_N) \ge    \det(Y_0)^k \exp(-4 \pi \tr(T_0Y_0)) \sumn_{F \in B_k(N)} |a_F(T_0)|^2.
\end{align}
Since we are summing over an arbitrary orthonormal basis in the BK, a natural way to get the lower bound would have been via the Poincar\'e series. Thus we might relate the quantities $\sum_{F \in B_k(N) } |a_F(T_0)|^2$ with the Fourier coefficients of Jacobi-Poincar\'e series of index $1$ and level $N$; in a way similar to \cite{sd-hk}. Note that we must pass to Jacobi forms, as there is no useful theory of Poincar\'e series on $\skkn$.
The passage onto the space of Jacobi forms of index $1$ would require a relation between the Petersson norm of $\lan F,F \ran $ with $\lan \phi_F,\phi_F \ran $ where $F$ and $\phi_F$ correspond to each other. For $F$ varying over an arbitrary orthogonal basis (or even in a basis consisting of eigenforms), there is no apparent relation between the Petersson norms as mentioned above. This is because all known methods to derive such a relation rely upon a relation in the Hecke algebra: $V_m^* \cdot V_m = \sum_{d|m, (d,N)=1} i_d d^{k-2} T^J(m/d)$ for all $m \ge 1$. So one would require $\phi$ to be an eigenfunction of all $T^J(m)$ for our method to be successful. In addition, we believe that \cite[Prop.4.6]{Ag-Br} is incorrect as stated; it holds only on the newspace.\footnote{Indeed taking $M=p$, Prop.~4.6 in \cite{Ag-Br} would imply that $U_J(p)$  is self-adjoint in level $p$, which may not be true, see Section~\ref{appendix2} for details.}

Therefore to obtain the lower bound, we just consider the contribution of the newforms. For this, first note that if $F,G$ are distinct newforms of level $N$ in $\skkn$, then they are orthogonal. This can be seen in various ways. First, $\skk^{new}(N)$ and $J_{k,1}^{cusp,new}(N)$ are Hecke-isomorphic and (strong) multiplicity-one result holds for both of them—this implies that $F,G$ are orthogonal. More directly, one can also see it from the relation between $\lan F,G \ran $ with $\lan \phi_F,\phi_G \ran $ described below, which will also be crucial in what follows.

Consider the Dirichlet series $D_{F,G,N}(s)$ introduced by Kohnen-Skoruppa \cite{KS}. Put
\begin{align}
    D_{F,G,N}(s) := \sumn_{m \ge 1} \lan \phi_{F,m},\phi_{G,m} \ran m^{-s}.
\end{align}
Recall that the Petersson norms are not normalized by volume. Quoting from \cite[Theorem~3.72]{Br} we see that $D_{F,G,N}(s)$ has a residue at $s=k$, given by\footnote{\eqref{residuebringman} differs by a factor of $\pi^{-k}$ from \cite[Theorem~3.72]{Br} because of different normalizations.}{}\footnote{We prefer to use \cite{Br} as we feel that there are some inaccuracies in \cite{Ag-Br} about this residue; e.g., the factor $M^{-\kappa}$ in formula~(11) loc. cit., which does not fit into the calculations -- if true then, say, by taking $\kappa\ge 4$, the sup-norm problem becomes trivial. Furthermore, \cite{Ag-Br} refers to a paper by T. Horie, which we could not access after a lot of effort.} 
\begin{align}\label{residuebringman}
    \frac{(4 \pi)^k \, \pi^{2} }{\zeta^N(4)\Gamma(k)} \cdot \lan F,G \ran \cdot N^{-2} \cdot \sumn_{l|N} \mu(l)l^{-2}.
\end{align}
In particular, $D_{F,G,N}(s)$ is holomorphic at $s=0$ if $\lan F,G \ran =0$. On the other hand, computations by Aggarwal-Brown \cite{Ag-Br} show that the same residue also equals (clearly the calculations in loc. cit. hold for two distinct newforms $\phi,\psi$)
\begin{align} \label{jac-res}
        \frac{\zeta(2)}{\zeta^N(4)} \cdot \left(\prod\nolimits_{p|N}(1-p^{-1})^2(1+p^{-1}) \right)\cdot \lan \phi_F,\phi_G \ran \cdot L(k,f),
\end{align}
where $f,\phi_F,F$ correspond to each other uniquely. Our assertion about the orthogonality of $F,G$ and the relation between the inner products is now clear. 
\begin{rmk}
    Note that \eqref{jac-res} applied to two distinct newforms $\phi, \psi$ does not lead to the bizarre equality $L(k,f)=L(k,g)$, as the RHS of \eqref{jac-res} is actually zero in that case.
\end{rmk}

More precisely, if we compare \eqref{residuebringman} and \eqref{jac-res}, we can write,
\begin{align} \label{petreln}
   \lan \phi_F, \phi_F \ran = \frac{(4\pi)^k \pi^2  \zeta(2)^{-1}}{\Gamma(k) L(k,f)} \cdot N^{-2} \left(\prod\nolimits_{p|N}(1-p^{-1})^{-2}(1+p^{-1})^{-1} \right) \lan F, F \ran .
\end{align}
So we will now approach the lower bound via a suitable application of Waldspurger's formula. We start with the lower bound from \eqref{sklbd} further restricting it to newforms:
\begin{align} \label{sklbd-new}
    \sup(\skkn) \ge    \det(Y_0)^k \exp(-4 \pi \tr(T_0Y_0)) \sumn_{F \in H^*_k(N)} |a_F(T_0)|^2 \big /\lan F, F \ran.
\end{align}
We can choose any $T_0$ such that $\mrm{disc}(2 T_0)=D_0$ is fundamental. Then $a_F(T_0)=c_\phi(|D_0|)$.
A good choice is (if we refer to Waldspurger's formula and the exponent of $D_0$ there) $T_0=1_2$ so that $D_0=-4$. 
We write \eqref{sklbd-new} in terms of a basis $ \phi \in J_{k,1}^{cusp,new}(N)$. 

\textit{For the rest of this section, let us assume that $N=p$, an odd prime.}

We now appeal to Waldspurger's formula (cf. \cite{KZ}, also \cite[Theorem 3.13]{DPSS}) to notice that 
\begin{align} \label{wald}
\frac{4^{k-1}}{6} \frac{ |c_\phi(-4)|^2}{\lan \phi_F,\phi_F \ran} = \frac{ |c_{\phi_F}(-4)|^2}{\lan h_F,h_F \ran} = 2 \frac{\Gamma(k-1) 4^{k-3/2}}{\pi^{k-1}}\frac{ L (1/2, f \otimes \chi_{-4}) }{\lan f,f \ran} , 
\end{align}
where $h_F \in S^+_{k-1/2}(4p)$ is the unique newform in Kohnen's plus space corresponding to $\phi_F$, having the same Fourier coefficients. 

Next, we appeal to the following average result from \cite[Theorem~1.5~(1)]{PASD} viz.
\begin{align} \label{avg-Lvalues}
    \sumn_{f \in H^*_k(p)} \frac{L(1/2, f \otimes \chi_{-4}) }{\lan f,f \ran_p} \gg 1, 
\end{align}
for all $p$ large enough, and the implied constant being absolute. 

\begin{rmk} \label{sq-free-central}
   The lower bound in \eqref{avg-Lvalues} should also hold for all square-free levels. We, however, are not aware of a quotable reference. The method in \cite{PASD} can be generalized along the same lines to obtain it. This would then immediately give the lower bound $N^{-2}$ for $\sup(\skkn)$. 
\end{rmk}

We can then use the expression for $\lan F, F \ran $ from \eqref{petreln} in \eqref{sklbd-new}, along with \eqref{wald} and \eqref{avg-Lvalues}. Hence:
\begin{align}
     \sup(\skk^{new}(p)) \gg \det(Y_0)^k \exp(-4 \pi \tr(Y_0))  \frac{4^k p^{-2}}{\Gamma(k)}  \frac{(4 \pi)^k \pi^k}{\Gamma(k-3/2)}  \cdot \frac{1}{p^2};
\end{align}
and then choosing $Y_0= k 1_2 \cdot i/4 \pi$ gives us 
\begin{align} \label{lbd-level-p}
    \sup(\skk^{new}(p)) \gg k^{5/2} p^{-2}.
 \end{align}
The genesis of \conjref{sk-conj} is now clear from \eqref{lbd-level-p}. Later, we will see in Sections~\ref{old-cont-02} and \ref{old-cont-1} that the upper bound $O(p^{-2})$ holds for the oldspace.

\section{Description of the fundamental domain.}\label{fund-dom}
In order to understand the size of SK lifts in the fundamental domain $\mc F_2(p)$, we first recall the set of explicit coset representatives $\Gp \backslash \sptwo$. We follow the notation used in \cite{BNmodp} and for $0\le j\le 2$ put
\begin{equation}
    w_j=\begin{pmatrix}
        0_j & 0 & -1_{j} & 0\\
        0 & 1_{n-j} & 0 & 0_{n-j}\\
        1_{j} & 0 & 0_{j} & 0\\
        0 & 0_{n-j} & 0 & 1_{n-j}
    \end{pmatrix}.
\end{equation}
Further, for $A\in \GL_2(\mf Z)$ and $B\in M_2 (\mf Z)$, we define
\begin{equation}
m(A):=\smat{A}{0}{0}{(A^{-1})^t} \text{ and } n(B):=\smat{1}{B}{0}{1}.
\end{equation}
Let $\mc R(j)$ denote the set of matrices of the form $w_j n(B_j) m(A)$, where $B_j$ and $A$ are as in \cite{BNmodp} (lifts of matrices in $\mrm{Sym}_j(\mf F_p)$ and $P_{2,j}\backslash \GL_2(\mf F_p)$ to $\sptwo$, respectively). Then a set of right coset representatives for $\gp \backslash \sptwo$ is given by
\begin{equation}
    \gp \backslash \sptwo =\bigcup \nolimits_{j=0}^{2} \mc R(j).
\end{equation}
We call the coset representatives as Type~0, Type~1 or Type~2 depending on whether they belong to $\mc R(0)$, $\mc R(1)$ or $\mc R(2)$ respectively.

Let $\mc F_2$ denote a fundamental domain for the action of $\sptwo$ on $\htwo$. Then the fundamental domain $\mc F_2(p)$ for the action of $\Gp$ can now be written as
\begin{equation}
    \mc F_2(p)=  \bigcup \nolimits_{j=0}^{2} \bigcup \nolimits_{\gamma\in \mc R(j)} \gamma \mc F_2.
\end{equation}
We call $\cup_{\gamma\in \mc R(j)}\gamma\mc F_2$ as the Type $j$ region. For each $j$, the type $j$ region requires different treatments while bounding the Bergman kernel. In the table below, we summarize the tools that we will use in these different regions in the next part of the paper.
\begin{table}[h]
    \centering
    \begin{tabular}{|c|c|}
    \hline
    {\it Region} & {\it Tools Used} \\ \hline
        Type 0 &  Fourier Expansion \\ \hline
         Type 1 & Fourier-Jacobi Expansion \\ \hline
         Type 2 & Counting Argument \\ \hline
    \end{tabular}
    \caption{Regions and Tools}
    \label{tab:placeholder}
\end{table}

We would bound the BK for SK as follows:
\begin{align}\label{newoldsplit}
    \sup(\skkp)
    \ll \sup_{Z\in \mf H_2}\sum_{F\in \mc B^{SK, new}(p)}\frac{\det(Y)^k|F(Z)|^2}{\norm{F}^2_p}+ \sup_{Z\in \mf H_2}\sum_{F\in \mc B^{SK, old}(p)}\frac{\det(Y)^k|F(Z)|^2}{\norm{F}^2_p}.
\end{align}
Denote the corresponding sizes of new and old spaces in the above equation as $\sup(\skk^{new}(p))$ and $\sup(\skk^{old}(p))$, respectively. Thus it is enough to bound the new and old spaces separately.

\section{Newform contribution in Type~0 region.}\label{Type0}

In this region, we have $Z\in \mc F_2$ (the standard Siegel's fundamental domain of level $1$ inside $\h_2$) and thus we have that $Y\gg 1$. We now recall the formula for the Fourier coefficients of $F \in \skk^{new}(p)$ in terms of those of the lifted $\phi \in J_{k,1}^{cusp,new}(p)$. Throughout the remainder of this section, we shall put 
\begin{equation}
    F(Z) = \sumn_{m \ge 1} \phi_{m,F}(\tau,z) e(m \tau'). 
\end{equation}
and for any $\mc D \equiv 0,-1 \bmod{4}$, $\mc D>0$,
    we put $\displaystyle c_{\phi,F}(\mc D) := c_{\phi_{1,F}}(\mc D)$.

Then we have, for $T = \smat{n}{r/2}{r/2}{m}$ and $\mf c(T) = (n,r,m)$ -- the content of $T$, the following relation holds (where $D= \det(2T)$):
\begin{align} \label{fcreln}
    a_F(T) =  \sumn_{a|\mf c(T), \, (a,N)=1} a^{k-1}  c_{\phi,F}\left( D/a^2 \right).
\end{align}
We begin by estimating the quantity $\mathcal{P}_*(T):= \sum_{F \in \mc B^{SK, new}(p) } |a_{F}(T)|^2/\norm{F}^2_p$ . If we also put 
    $p_*(D) := \sum_{\phi \in \mc B^{J,new}_{k,1}(p)}  |c_{\phi,F}(D)|^2$, where $\mc B^{J,new}(p)$ denotes any orthonormal basis of $J_{k,1}^{cusp,new}(p)$, then we see that 
    \begin{align}
        \mathcal{P}_*(T)  \ll p^{-2} \, ( \sumn_a a^{k-1} p_*(D/a^2)^{1/2} )^2.
    \end{align}
    This follows by using the relations \eqref{fcreln} and \eqref{petreln}:
\begin{equation}
    \begin{split}
      \mathcal{P}_*(T)&= \sumn_F \Big| \sumn_{a|\mf c(T)} a^{k-1}  c_{\phi,F}\left( \frac{D }{a^2} \right) \Big|^2 \ll p^{-2} \Big( \sumn_a a^{k-1} \big( \sumn_\phi |c_{\phi,F}\left( \frac{D }{a^2} \right) |^2 \big)^{1/2} \Big)^2.
    \end{split}
\end{equation}

Let $P_{D,r} \in \jkn$ be the $D$-th Poincar\'e series, it is well-known that it depends only on $D$. The lemma below gives bounds on its Fourier coefficients, which follows from \cite[Thm.~3.23]{Br} and standard estimates.
\begin{lem}\label{poincareN}
$c_{D,r} = 1 + O((D,N)^{1/2}D^{1/2} N^{-1+\epsilon})$. 
\end{lem}

Now using the bound for the Fourier coefficients of Poincar{\'e} series from Proposition \lemref{poincareN} we see that (see \cite{PASD} for similar arguments in level $1$)
\begin{align} 
\mathcal P_*(T) \ll_\epsilon p^{-2} \frac{ 4^k \pi^{2k} }{\Gamma(k) \Gamma(k-3/2) } 
 \cdot \left(  D^{k-3/2} c(T)^{1+\epsilon} + p^{-1/2 +\epsilon} k^{-1/3} D^{k-1}   c(T)^{\epsilon} \right). \label{p*t}
\end{align}
Next, note that we can write by \cite[Lem.~4.4]{sd-hk} (follows from the Cauchy-Schwartz inequality) that
\begin{align} \label{lem41}
     \sumn_{F \in \mc B^{SK, new}(p) }  \frac{\det(Y)^k |F(Z)|^2}{\norm{F}^2_p} \le q_*(Y)^2,
     \end{align}
where for $Y>0$, we have put
\begin{align} \label{qp}
   q_{*}(Y) := \sumn_T \mathcal P_*(T)^{1/2} \det(Y)^{k/2} \exp(-2 \pi \tr TY).
\end{align}
Plugging in the bound for $\mathcal P_*(T)$ in \eqref{qp} we get that $ q_*(Y) $ is
\begin{align}
\ll p^{-1}Q(1/2,3/4-\epsilon;Y)+ p^{-5/4+\epsilon} Q(\epsilon, 1/2; Y), \label{twoQ}
\end{align}
where the implied constant depends polynomially on $k$ and
\begin{align}\label{QabY}
    Q(\alpha, \beta; Y): = \frac{ (4\pi)^k}{ \Gamma(k) } \sumn_T \mf c(T)^\alpha  \det(T)^{k/2-\beta +\epsilon} \det(Y)^{k/2} \exp(-2 \pi \tr TY) .
\end{align}

\subsection{Analysis of the Fourier expansion}
This section collects various bounds coming from the Fourier expansions of modular forms in level $N$. We start with a general set-up and consider any $n  \ge 1$.
For any $S\subseteq \Lambda_n^+$ we put
\begin{equation}
    Q_S(\beta;Y) := \sumn_{T\in S} \det(T)^{k/2-\beta} \det(Y)^{k/2} \exp(-2 \pi \tr TY),
\end{equation}
Then we can write $Q_{\Lambda_n^+}(\beta; Y)$ as
\begin{equation}\label{qabydef}
    Q_{\Lambda_n^+}(\beta; Y)= Q_{\mc C(Y, N^\epsilon)}(\beta;Y) + Q_{\Lambda_n^+ \setminus \mc C(Y, N^\epsilon)}(\beta;Y),
\end{equation}
where the set $\mc C(Y, N^\epsilon)$ is defined as below. 
\begin{align}\label{cydef}
    \mc C(Y, N^\epsilon) := \{ T \in \Lambda^+_n | \text{ all the eigenvalues of } TY \text{ are of the magnitude } \frac{k}{4\pi}+O(k^{1/2+\epsilon} N^\epsilon)  \}.
\end{align}
Now we show that the tail in \eqref{qabydef} has sub-exponential decay. To be precise, we show that (assuming that $\beta \ge 0$):
\begin{lem}\label{tailest}
For $Y>0$ Minkowski reduced, $N$ large enough, one has
    \begin{align}
    Q_{\Lambda_n^+ \setminus \mc C(Y, N^\epsilon) }(\beta;Y) \ll \exp(- c N^\epsilon) \det(Y)^{\beta- (n+1)/2}.
\end{align}
Moreover, if one of the diagonal entries $v,v' \gg N^\epsilon$, one has a sub-exponential decay: $\displaystyle Q_{\Lambda_n^+}(\beta; Y) \ll \exp(-c_1 N^\epsilon) \det(Y)^{\beta- (n+1)/2}$.
\end{lem}
\begin{proof}
Let $\lambda_1(TY)\le \lambda_2(TY)\le \cdots \le\lambda_n(TY)$ denote the eigenvalues of $TY$. Then $\Lambda_n^+ \setminus \mc C(Y, N^\epsilon)$ is the set of $T$ such that $\lambda_n(TY) > N^\epsilon$. For any $A>0$ and an integer $t\ge 0$, let us put
\begin{align}
    D_t (A):= \{ T \in \Lambda_n    | 2^t A \le \lambda_n(TY) < 2^{t+1}A  \}.
\end{align}
Clearly $Q_{\Lambda_n \setminus \mc C(Y, N^\epsilon) }(\alpha, \beta,Y)$ is bounded by
\begin{align} \label{tail1}
\det(Y)^\beta\sumn_{t\ge 0} \sumn_{T \in D_{t}(N^\epsilon)}\prod\nolimits_{j=1}^{n}m(\lambda_j(TY)),
\end{align}
where $ m(\lambda) := \exp \big( (\frac{k}{2}-\beta) \log(\lambda) - 2 \pi \lambda \big)$.

The function $m(x)$ is bounded on positive reals and decays sub-exponentially when $x> N^\epsilon$. With this in mind, the expression in \eqref{tail1} is bounded up to a constant depending only on $k$, by
\begin{align}
  \exp(-c_1 N^{\epsilon}) \det(Y)^\beta \sumn_{t=0}^\infty \# D_{t}(N^\epsilon) \exp(-c_1 2^t N^{\epsilon}) .
\end{align}
The proof is now complete after we observe the following bound for the cardinality of the set $D_{t}(N^\epsilon)$:
\begin{equation}
    D_{t}(N^\epsilon) \ll N^\epsilon \cdot 2^{(n+1)t} (y_1y_2\cdots y_n)^{-(n+1)/2} \asymp N^\epsilon \cdot 2^{3t} \det(Y)^{-(n+1)/2}. \qedhere
\end{equation}

For the last assertion of the lemma, notice that $\mc C(Y, N^\epsilon)$ is empty if either $v$ or $v'$ is $\gg N^\epsilon$, since then $\tr (TY) \gg N^\epsilon$ for large enough $N$.
\end{proof}

\begin{lem} \label{count-bd2}
  Let $ \mc C(Y, N^\epsilon)$ be as defined in \eqref{cydef} and $Y$ Minkowski reduced.
  Then $\# \mc C(Y, N^\epsilon) \ll N^\epsilon \det(Y)^{-(n+1)/2}$. 
  \end{lem}
\begin{proof}
First, we note that since $k$ is fixed, for large $N$,  $T\in \mc C(Y, N^\epsilon)$ implies that $\tr(TY)\ll N^\epsilon$. It is easy to see that $\displaystyle t_1y_1+t_2y_2+\cdots +t_ny_n \ll N^\epsilon$. This gives $N^\epsilon  y_j^{-1}$ choices for the $t_j$. Moreover we must have $y_j \ll N^\epsilon$, otherwise the count is zero.

The off-diagonal elements $t_{ij}$ (where $i \neq j$) satisfy $2 |t_{ij} |\le (t_i t_j)^{1/2}$. Given $t_j \ge 1$ ($1\le j \le n$), the total count is given by
\begin{equation}
N^\epsilon\prod\nolimits_{i=1}^{n}y_i^{-1}\prod\nolimits_{j>i} (y_iy_j)^{-1/2}= N^{\epsilon}\det(Y)^{-(n+1)/2}. \qedhere
\end{equation}
\end{proof}

\begin{rmk}
    More generally in degree $n$, following the lines of \cite[Lem.~4.9]{sd-hk} one can also arrive at the estimate $\# \mc C(Y, N^\epsilon)
    \ll_{n,k} N^{\frac{n+1}{2}+\epsilon}\det(Y)^{-\frac{n+1}{4}}$ valid for $Y \gg 1/N$. We note however that this result leads to the same bound that we get in \eqref{new-type2}.
\end{rmk}

Clearly, in $Q_{\Lambda_n^+}(\beta; Y)$ we can assume that $Y$ is Minkowski reduced, hence it is enough to bound it with this assumption. For the $T \in \mc C(Y, N^\epsilon)$ we simply note that since $m(x) \ll_k 1$ for all $x>0$ (cf. \eqref{tail1}),
\begin{equation}
     Q_{\mc C(Y, N^\epsilon)}(\beta;Y) \ll \det(Y)^\beta
     \sumn_{T \in \mc C(Y, N^\epsilon)} 1 \ll \det(Y)^\beta \# \mc C(Y, N^\epsilon).
\end{equation}
Therefore, combining \lemref{tailest} and \lemref{count-bd2}, we get the following bound for $Q_{\Lambda_n^+}(\beta; Y)$.
\begin{prop}\label{prop:QbY}
For any $Y>0$, one has
\begin{equation}
     Q_{\Lambda_n^+}(\beta; Y)\ll_{k,n} N^{\epsilon}\det(Y)^{\beta-(n+1)/2} .\qedhere
\end{equation}   
\end{prop}

Now we return to the case at hand and assume $n=2$. Using the analysis of the Fourier expansion and the counting argument in the previous section, we arrive at the following estimate for the function $Q(\alpha,\beta;Y)$. 
\begin{lem}\label{lem:QabY}
    Let $\alpha< 2$, $\beta>0$ and $Q(\alpha,\beta;Y)$ be as in \eqref{QabY}, then we have
    \begin{equation}
        Q(\alpha,\beta;Y)\ll p^{\epsilon} \det(Y)^{\beta-3/2}.
    \end{equation}
\end{lem}
\begin{proof}
We start from \eqref{QabY}, and put $d=\mf c(T)$ and write $T=dT_0$ where $T_0$ is primitive. Further, $Y_d:= dY$. Then we can write,
\begin{align}
    Q(\alpha, \beta; Y) & \ll \sumn_{d\ll p^\epsilon/\det (Y)^{1/2}} \frac{1}{d^{ 2 \beta - \alpha - \epsilon}} \sumn_{T_0 \text{ primitive}} \frac{\det(T_0 Y_d)^{k/2}  \exp(- 2 \pi \tr T_0Y_d)}{ \det(T_0)^\beta \Gamma(k)}\\
    & \ll p^{\epsilon} \det(Y)^{\beta-3/2}\sumn_{d\ll p^\epsilon/\det (Y)^{1/2}} d^{-3+\alpha+\epsilon}\ll p^{\epsilon} \det(Y)^{\beta-3/2},
\end{align}
where we use \propref{prop:QbY} for the sum over $T_0$ and the fact that the sum over $d$, extended to $\mf N$, is convergent for $\alpha< 2$.\qedhere
\end{proof}
Now we bound the space of newforms in the region $Y\gg 1$. 
Putting $\alpha=1/2$, $\beta= 3/4-\epsilon$, we get from \eqref{QabY} and \eqref{twoQ} that
$\displaystyle q_*(Y)\ll p^{-1+\epsilon}\det(Y)^{\epsilon}$.
Thus
\begin{equation} \label{new-type0}
    \sup_{Z \in \mf H_2, Y \gg 1}\sum_{F \in \mc B^{SK, new}(p)}  \frac{\det(Y)^k |F(Z)|^2}{\norm{F}^2_p} \ll p^{-2+\epsilon}.
\end{equation}

\section{Newform contribution in Type~2 region.}
The Type~2 region is determined by the Type~2 representatives, i.e., $\gamma\in \mc R(2)$.
For Type~2 representatives, we have to consider $G (\mc A_2 \langle Z \rangle )$, where we put $G(Z) = \finv$ and $\mc A_2 = w_2 n(B)m(A)$. Let us note that $w_2 = J_1 \times J_1$. Therefore, we can write
\begin{align} \label{type3-1}
  G (\mc A_2 \langle Z \rangle )  \le \max_{B} \det(Y)^{k/2} | \left(F|  W_p \cdot B_p^{-1} n(B) m(A)\right)|.
\end{align}
From \cite{schmidt-AL}, $W_p$ given by the Fricke involution at $p$, is ($\gp$ equivalent to) the Atkin-Lehner involution for the group $\gp$, of which $F$ is an eigenfunction of eigenvalue $+1$ (cf. \cite{schmidt-SKlift}, even though we don't need the exact sign). Therefore, the right-hand side of \eqref{type3-1}
is just (with $B_p = \psmb p1_2 & 0_2 \\ 0_2 & 1_2 \psme$)
\begin{align}
    \max_{B} \det(Y)^{k/2} |F | B_p^{-1} n(B) | =  \max_{B} p^{-k} \det(Y)^{k/2} |F \left(\frac{ Z  +B}{p} \right) | .
\end{align}
Here our convention for the action of $\gamma$ is $Z \mapsto \det(\gamma)^{k/2} \det(CZ+D)^{-k} F(\gamma Z)$. So in these cases, we have for each $B$,
\begin{align}
   \sup_{Z \in  \mc F_2} p^{-k} \det(Y)^{k/2} |F \left(\frac{ Z  +B}{p} \right) | & \ll  \sup_{Z \in \mf H_2, Y \gg p^{-1}} \det(Y)^{k/2} |F(Z)| .
\end{align}
Therefore, in this section we will assume that
\begin{align} \label{y-region}
    Z \in \htwo, \, Y=\im(Z) \gg p^{-1} \text{ and } v,v'\ll p^\epsilon,
\end{align}
otherwise by \lemref{tailest} we will have sub-exponential decay. 

\subsection{Trivial bound in Type~2 region} Using the arguments as in Section \ref{Type0}, we get that 
\begin{equation}
    \sup_{Z \in \mf H_2, Y \gg 1/p}\sumn_{F \in \mc B^{SK, new}(p)}  \frac{\det(Y)^k |F(Z)|^2}{\norm{F}^2_p} \ll q_*(Y)^2.
\end{equation}
Next, from Lemma \ref{lem:QabY}, in the region $Y\gg 1/p$, we get $q_*(Y)\ll p^{1/2+\epsilon}\det(Y)^{\epsilon}$. This gives: 
\begin{equation} \label{new-type2}
    \sup_{Z \in \mf H_2, Y \gg 1/p}\sumn_{F \in \mc B^{SK, new}(p)}  \frac{\det(Y)^k |F(Z)|^2}{\norm{F}^2_p} \ll p^{1+\epsilon}.
\end{equation}
To give some perspective, we call this bound as the `trivial bound'. To obtain a non-trivial bound in the region $Y\gg 1/p$, we work with the Fourier-Jacobi expansion of $F$ in the next section.
\subsection{Treatment of Type~2 region via the Fourier-Jacobi expansion}

We will attempt to treat the Fourier Jacobi coefficients in the same footing as the Fourier coefficients of cusp forms. 

\textit{The arguments for the rest of the section are valid for any square-free $N$, thus we use $N$, instead of $p$}. 

We start with (note that $tv=\det(Y)$)
\begin{align} \label{m-sum-phi}
    \det Y^k \sumn_F \frac{|F(Z)|^2}{\lan F,F \ran} \asymp N^{-2} \det (v t)^k \sumn_\phi \frac{ |\sumn_m \phi_m e(m \tau')|^2}{\lan \phi,\phi \ran}.
\end{align}
First, we show that the $m$ sum above can be truncated, at a negligible cost, at 
\begin{equation} \label{l-def}
    m \le L:=N^\epsilon/v'. 
\end{equation}
The result given below is certainly not the best possible, but it is enough for our purposes. Also we consider more generally for \textit{any} index $m$ Jacobi form, not necessarily of the form $V_m \phi$, with $\phi \in \jkn$. Perhaps this will have uses elsewhere, and is new, anyway. Recall $\widetilde{\phi}$ from \eqref{inva}. We have the following estimate, which is the `trivial bound' in the  $N$ aspect.

\begin{prop} \label{phi-crude}
    Let $\phi \in \jkmc(N)$. Then we have 
    \begin{equation}
        \norm{\phi}_\infty = \sup\nolimits_{\h \times\C} \widetilde \phi/\norm{\phi}_N \ll N^{\epsilon} m^{1+\epsilon}.
    \end{equation}
\end{prop}

\begin{proof}
We take the most expedient route just by appealing to the theta expansion and bounding the Theta components of $\phi$ on the one hand, and on the other, bound the $h_\mu$ individually by considering them as cusp forms on $\Gamma(4mN)$. Put $r=k-1/2$. By Cauchy-Schwarz (denoted by CS henceforth), we have
\begin{align} \label{phi-adhoc}
    v^{k} e^{- 4 \pi y^2/v}|\phi(\tau,z) |^2/\norm{\phi}^2_N  \le  \, \sum_{\mu \bmod {2m}} v^r |h_\mu (\tau)|^2 /\norm{\phi}^2_N \cdot \sum_{\mu \bmod {2m}} v^{1/2} e^{- 4 \pi y^2/v} |\theta_\mu(\tau,z)|^2.
\end{align}
We will be brief, since the proof follows in spirit the arguments from the cases $m=1$ along with $N=1$ (cf. \cite{PASD}).

As in Section \ref{jacobi-sup-norms}, we can work in the region $\cup_j \gamma_j \mc F^J_1$ -- which shows that it is sufficient to the bound on $\mc F^J_1$. We first consider the theta series.
\begin{align}
    &\sumn_{\mu \bmod {2m}} v^{1/2} e^{- 4 \pi m y^2/v}  |\theta_\mu(\tau,z)|\gamma_j |^2  = v^{1/2} e^{- 4 \pi m y^2/v} \sumn_\mu \Big|\sumn_\nu \varepsilon(\nu, \mu; \gamma_j)\theta_\nu(\tau,z) \Big|^2\\
    & \ll v^{1/2} e^{- 4 \pi m y^2/v}\left(\sumn_\mu |\theta_\mu|^2 \right)\ll  m^{1/2} (1+ \frac{1}{vm})^{1/2}(1+ \frac{v}{m})^{1/2}  \ll m^{1/2}v^{1/2}.
\end{align}
Next, we recall that $h_\mu\in S_r(\Gamma(4mN))$ (see \cite[Section~4, Part~B]{boche-das} for a proof). Thus the Petersson norm relation between $\phi$ and the $h_\mu$ can be rewritten as follows:
\begin{align} \label{phim-hmu}
    (mN)^{2}(1 + o(mN) ) \norm{\phi}^2_N = \frac{1}{\sqrt{m}}\sumn_\mu \norm{h_\mu}_{\Gamma(4mN)}^2\gg m^{-1/2} \norm{h_\mu}_{\Gamma(4mN)}^2.
\end{align}
The above follows by noting that (see \cite[Thm.5.3]{EZ} for example)
\begin{equation}
    \lan \phi, \phi\ran_N=\frac{1}{\sqrt{2m}}\int_{\Gamma_0(N)\backslash\h}\sumn_{\mu} |h_\mu(\tau)|^2 v^{k-5/2}du dv.
\end{equation}
We are left with bounding the quantity $v^r |h_\mu |_r \gamma_j (\tau)|^2$ in the region $\mc F_1$. Note that $h_\mu |_r \gamma_j \in S_r(\Gamma(4mN))$ and has a Fourier expansion of the form $\sum a(n) q^{n/mN}$. Thus using the arguments as in \cite[Proposition 3.1]{Kir}, we see that
\begin{equation}
    v^r |h_\mu |_r \gamma_j (\tau)|^2/\norm{h_\mu}_{\Gamma(4mN)}^2 \ll \frac{1}{[\Gamma_0(4):\Gamma(4mn)]}\cdot \frac{(mN)^{1+\epsilon}}{v}\ll (mN)^{-2+\epsilon}v^{-1}.
\end{equation}
Summing over $\mu \bmod{2m}$, and using \eqref{phim-hmu} for each $\mu$ we get
\begin{align}
    v^{k} e^{- 4 \pi y^2/v}|\phi(\tau,z) |^2/\norm{\phi}^2_N\ll N^{\epsilon} m^{2+\epsilon} v^{-1/2}\ll N^{\epsilon} m^{2+\epsilon}.
\end{align}
The last inequality follows since $\tau\in \mc F_1^J$ and this completes the proof.
\end{proof}

We now want to return to our claim about sub-exponential decay, and look at \eqref{m-sum-phi}.
Recall that $t= \det Y/v \asymp v'$ and $v' \ge 1/N$. We also have that $\norm{\phi_m}^2\ll m^{k-1+\epsilon}\norm{\phi}^2$ (see \cite[Section 7.3.1]{PASD} for details). We can then get the contribution of the `tail' $m>L$ as
\begin{align}
    & \sum_\phi \left| \sum_{m > L}  \frac{v^{k/2} e^{- 2 \pi m y^2/v} |\phi_m|}{\norm{\phi_m}}  \frac{\norm{\phi_m}}{\norm{\phi}} e^{- 2 \pi m t} \right|^2 \ll \sum_\phi \left|  \sum_{m > L} (N^\epsilon m) \, (m^{ (k-1)/2+\epsilon} ) e^{- c_0 m v'} \right|^2 \\
    & \ll N^\epsilon e^{- c_0 L} | \sumn_{m > L} m^{k/2+1/2+\epsilon} e^{-  c_0 m v'/2}|^2  \cdot \sumn_\phi 1\ll N^{1+\epsilon} v'^{-k-3 +\epsilon} e^{-c_0 N^\epsilon}\ll N^{4+\epsilon} v'^{-k} e^{-c_0 N^\epsilon} ,\label{m-sum-tail-bd}
\end{align}
for some absolute constant $c_0$. Since $t \asymp v'$, the $v'^{-k}$ cancels off in \eqref{m-sum-phi}, ultimately leaving a sub-exponential decay for $N$ large enough.

\subsection{Trivial bound for the Type 2 contribution via FJ expansion} \label{triv-bd-type2}

We first show that by using Cauchy-Schwarz inequality on the sum over $m$ one can get a `trivial' bound $\sup(\skk^{new}(N)) \ll N^\epsilon$. By virtue of \eqref{m-sum-tail-bd}, it suffices to bound
the quantity $\displaystyle \det (v t)^k \sumn_\phi \frac{ |\sum_{m \le L} \phi_m e(m \tau')|^2}{\lan \phi,\phi \ran}$, which is bounded by 
\begin{align}\label{FJCS}
\q \sumn_m \frac{ |\phi_m|^2 }{m^{k-1} } \frac{ v^k e^{- 4 \pi m y^2/v} }{\lan \phi,\phi \ran} \cdot \sumn_m t^k m^{k-1} e^{4 \pi m y^2/v - 4 \pi m v'} = \sumn_m \frac{\tilde \phi_m^2}{m^{k-1} \lan \phi,\phi \ran } \cdot \sumn_m t^k m^{k-1} e^{-4 \pi tm},
\end{align}
where $t = \det Y/v$ and $L$ is as defined in \eqref{l-def}.
Then from the definition of $V_m$ operators (see \eqref{Vmtilde} for more details), the first term in the above product \eqref{FJCS} is
\begin{align}
    \ll \sumn_m m^{k-2} \frac{\left|\sum_{ad=m}\sum_{b\bmod d}\widetilde{\phi}\left(\frac{a\tau+b}{d},az\right)\right|^2 }{m^{k-1} \lan \phi,\phi \ran} 
\end{align}
and the second sum over $m$ in loc. cit. is $O(N^\epsilon)$ (cf. \propref{prop:QbY}). Let $B_J(\,,\,)$ denote the Jacobi BK of level $N$. The length of the inner sum above is $\sigma_1(m) \asymp m$.
Then we see that the first sum in the RHS of \eqref{FJCS} is bounded by
\begin{align}
    \sumn_m \frac{1}{m} \cdot m^2 \cdot B_J(*,*) \ll \sumn_m m \ll N^{2+\epsilon},
    \end{align}
which, in view of \eqref{m-sum-phi}, gives the bound $O(N^\epsilon)$ for the global sup norm assuming \conjref{jacobi-conj}, which says that the size of $\jkn$ in $1$. This is better than the bounds obtained in \eqref{new-type2} by using the Fourier expansion, but still far from the optimal bound. In the next section, we set out to improve this bound unconditionally.

\subsection{Non-trivial bound for the BK} \label{new-count-02}
\textit{We note here that the arguments in this subsection work for all $N \ge 1$.}

To bound the BK non-trivially, first note that
\begin{align}
    \sum_\phi\sum_m\frac{ |\phi_m|^2 }{m^{k-1} } \frac{ v^k e^{- 4 \pi m y^2/v} }{\lan \phi,\phi \ran} = \sum_\phi\sum_m \frac{v^{k}e^{-4\pi my^2/v}}{m^{-k+1} \lan \phi,\phi \ran} \Big|\sum_{d|m}\sum_{b\bmod d}d^{-k}\phi\left(\frac{a\tau+b}{d},az\right)\Big|^2.
\end{align}
Expanding the square, we see that the RHS can be written as
\begin{equation}\label{FJGeom}
    \mc M(\tau, z):=\sumn_m \frac{v^{k}e^{-4\pi my^2/v}}{m^{-k+1}}\sumn_{\alpha_m, \beta_m} (dd')^{-k} B_J(\alpha_m\tau, mz/d;\beta_m\tau, mz/d'),
\end{equation}
where $B_J(\tau,z;\tau',z')$ denotes the index one, level $N$ Jacobi BK and 
\begin{equation}
    \alpha_m=\smat{a}{b}{0}{d},\, \beta_m=\smat{a'}{b'}{0}{d'}\, ad=a'd'=m,\, b \bmod d,\, b'\bmod d'.
\end{equation}

\noindent\textit{Choice of residue classes:}
    The operator $V_m$ is independent of the choice of residue class for $b \mod d$. Thus, in \eqref{FJGeom}, we are free to choose the residue classes for $b \bmod d$ and $b' \bmod d'$. Keeping in mind the future calculations, we choose the residue classes such that
    \begin{equation}
        db-d'b'\ge d^2.
    \end{equation}
As example, we can take $d'^2/d+d < b \le d'^2/d+2d$ and $0\le b' \le d'$.

The geometric side of the index one Jacobi BK is given by (see \cite[pp. 184]{sko-zag})
\begin{equation}\label{BKGeometric}
    B_{J}(\tau_1,z_1;\tau_2,z_2)= C_k\ \sumn_{\gamma\in\Gamma_0(N)}(\Theta|^{(1)}_{k}\gamma)(\tau_1,z_1;\tau_2,z_2),
\end{equation}
where $\Theta |^{(1)} \gamma$ indicates the actions of $\gamma$ with respect to the first set of variables, and
\begin{align} \label{theta-def}
    \Theta=\Theta_{k}(\tau_1,z_1;\tau_2,z_2) &:=(\tau_1-\bar{\tau}_2)^{\frac{1}{2}-k}\sumn_{\eta\in \mf Z/2\mf Z}\theta_{\eta}(\tau_1,z_1)\overline{\theta_{\eta}(\tau_2,z_2)}, \q \q \text{and}  \\
\theta_{\eta}(\tau,z)& :=\sumn_{r\in \mf Z}e( (r+\eta)^2\tau+2(\eta+r)z).
\end{align}
Let $\gamma=\smat{a}{b}{c}{d}\in \sltwo$, then
\begin{equation}\label{thetatrans}
\theta_{\mu}\left(\gamma(\tau,z)\right)(c \tau + d)^{-\frac{1}{2}} e\left(-c z^2 (c \tau+d)^{-1}\right)=\sumn_{\eta\in\mf Z/2\mf Z} \varepsilon(\eta, \mu;\gamma)\theta_{\eta}(\tau, z),
\end{equation}
where $\varepsilon(\eta, \mu;\gamma)$ are complex numbers such that $\rho(\gamma):=\left(\varepsilon(\nu, \mu;\gamma)\right)_{\nu,\mu}$ is a unitary matrix (see e.g., \cite{arakawa1992selberg}). The Bergman kernel  $B_{J}(\tau_1,z_1;\tau_2,z_2)$ can now be written as
\begin{equation}\label{bkmtheta}
\sumn_{\gamma\in\Gamma_0(N)}\  \left(j(\gamma,\tau_1)(\gamma(\tau_1)-\overline{\tau_2})\right)^{-k+\frac{1}{2}}\sumn_{\eta,\mu}\varepsilon(\eta, \mu;\gamma)\theta_{\eta}(\tau_1, z_1)\overline{\theta_{\mu}(\tau_2,z_2)}.
\end{equation}
From \eqref{bkmtheta}, we can bound $B_{J}(\tau_1,z_1;\tau_2,z_2)$ by
\begin{equation}\label{BKbydef}
    \ll  \sumn_{\gamma}\left|j(\gamma,\tau_1)(\gamma(\tau_1)-\overline{\tau_2})\right|^{-k+\frac{1}{2}}\big|\sumn_{\nu,\mu}\varepsilon(\nu, \mu;\gamma)\theta_{\nu}(\tau_1, z_1)\overline{\theta_{\mu}(\tau_2,z_2)}\big|.
\end{equation}
Let $\rho(\gamma)$ be as above and $\Theta=(\theta_{\mu})_{\mu}$. Then the sum over $\nu,\mu$ is 
\begin{align}
    =\left|\lan \rho(\gamma)\Theta(\tau_1,z_1),\Theta(\tau_2,z_2)\ran \right|{\ll}\lan \rho(\gamma)\Theta(\tau_1,z_1),\rho(\gamma)\Theta(\tau_1,z_1)\ran^{1/2}\lan \Theta(\tau_2,z_2),\Theta(\tau_2,z_2)\ran^{1/2} ,
\end{align}
where we use Cauchy-Schwarz for the first inequality, and the unitary property of $\rho_m(\gamma)$ for the second. Here $\lan \, , \, \ran$ denotes the standard inner product on Euclidean space.

Now we have that
\begin{align}
   \sumn_\mu |\theta_{\mu}(\tau,z)|^2 &\ll \ v^{-1/2} e^{4 \pi y^2/v}(1+ v^{1/2}) (1+v^{-1/2})= e^{4 \pi y^2/v} (1+ v^{-1/2})^2.
\end{align}
Thus
\begin{equation}
    \big|\sumn_{\nu,\mu}\varepsilon(\nu, \mu;\gamma)\theta_{\nu}(\tau_1, z_1)\overline{\theta_{\mu}(\tau_2,z_2)}\big|\ll e^{2 \pi (y_1^2/v_1+y_2^2/v_2)}(1+ v_1^{-1/2})(1+ v_2^{-1/2}).
\end{equation}
In our case, $\tau_1= (a\tau+b)/d$, $\tau_2= (a'\tau+b')/d'$ and we have $ad=m$, $a'd'=m$. Thus the RHS can be written as
\begin{equation}\label{thetabound}
    e^{(...)}\left(1+d(mv)^{-1/2}\right)\left(1+d'(mv)^{-1/2}\right).
\end{equation}

Now we bound the sum over $\gamma$ in \eqref{BKbydef}. Write $r=k-\frac{1}{2}$ and define
\begin{equation}
    M(\alpha_m\tau,\beta_m\tau):= \im(\beta_m\tau)^{-r}\sumn_{ \gamma\in \Gamma_0(N)} \, \big |(\gamma\alpha_m\tau-\overline{\beta_m\tau})j(\gamma,\alpha_m\tau)\im(\beta_m\tau)^{-1}\big |^{-r}.
\end{equation}

\subsubsection{The counting argument.} \label{count-N}
First, for any $g\in \mrm{GL}_2(\mf R)$ let us define 
\begin{equation}
    A_g(\tau,w):= (g\tau-\overline{w})j(g,\tau)\im(w)^{-1},
\end{equation}
For the sake of simplicity, when $\tau=w$, we write $A_g(\tau):= A_{g}(\tau, \tau)$. 

\begin{lem}\label{lem:Agamma}
    Let $\alpha_m$, $\beta_m$ and $A_g(\tau,w)$ be as above. Then we have
    \begin{equation}
        A_{\gamma}(\alpha_m\tau, \beta_m\tau)= A_{\beta_m^{-1} \gamma \alpha_m}(\tau).
    \end{equation}
\end{lem}
\begin{proof}
Let us write  $u_m+iv_m=w_m:= \beta_m\tau $ and 
\begin{align}
A_\gamma(\alpha_m\tau,w_m):=(\gamma\alpha_m\tau-\overline{w_m})j(\gamma,\alpha_m\tau)v_m^{-1}.
\end{align}
Using the cocycle conditions, we see that
\begin{align}
A_\gamma(\alpha_m\tau,w_m)&=A_\gamma(\alpha_m\beta_m^{-1} w_m, w_m)= (\gamma\alpha_m\beta_m^{-1}w_m-\overline{w_m})j(\gamma,\alpha_m\beta_m^{-1}w_m)v_m^{-1}\\
    &=\frac{d'}{d} (\gamma\alpha_m\beta_m^{-1}w_m-\overline{w_m})j(\gamma\alpha_m\beta_m^{-1},w_m)v_m^{-1}= \frac{d'}{d} A_{\gamma\alpha_m\beta_m^{-1}}( w_m, w_m).\label{d'd}
\end{align}

Next, we note that $\beta_m\tau= (\frac{1}{\sqrt m}\beta_m )\tau$ and for any $M\in\sltwor$, $A_g(M\tau, M\tau)= A_{M^{-1}gM}(\tau,\tau)$. Thus we get that
\begin{align}
    A_{\gamma\alpha_m\beta_m^{-1}}( \beta_m \tau, \beta_m \tau)= A_{\gamma\alpha_m\beta_m^{-1}}( (\frac{1}{\sqrt m}\beta_m) \tau, (\frac{1}{\sqrt m}\beta_m) \tau)= A_{\beta_m^{-1} \gamma \alpha_m}(\tau).
\end{align}
This completes the proof of the lemma.
\end{proof}
We need global lower bounds on the quantities $A_{\beta_m^{-1}\gamma\alpha_m}( \tau)$. Towards this, we prove the following.

\begin{lem}\label{lem:lowAg}
    Let $A_{\beta_m^{-1}\gamma\alpha_m}( \tau)$ be as above. Then we have 
    \begin{equation}
        |A_{\beta_m^{-1}\gamma\alpha_m}( \tau)|\ge 2.
    \end{equation}
\end{lem}
\begin{proof}
For the sake of simplicity, let us write $g= \beta_m^{-1}\gamma\alpha_m$ and $u_m'+iv_m'=w_m:= g\tau$.
Then we have that
\begin{equation}
\left| (w_m-\overline{\tau})j(g,\tau) v^{-1} \right|= |(w_m-\overline{\tau})(v_m v)^{-1/2}|\ge (v_m+v)/(v_m v)^{1/2}.
\end{equation}
 Now we use the AM-GM inequality to complete the proof. \qedhere
\end{proof}

We now define a counting function $\mc C(\tau, m, \delta)$. Suitable bounds on it will be instrumental in handling the Type~2 region.
\begin{equation}
    \mc C(\tau, m, \delta):= \#\{(\gamma, \alpha_m, \beta_m): \gamma\in \Gamma_0(N), |A_{\beta_m^{-1}\gamma\alpha_m}(\tau)|< \delta\}.
\end{equation}
Next, from the Lemma \ref{lem:lowAg}, we see that $\mc C(\tau, m, \delta)=0$ if $\delta< 2$.

Let $\gamma=\smat{a_\gamma}{b_\gamma}{c_\gamma}{d_\gamma}$. Consider the case when $c_\gamma=0$. In this case, $\gamma= \smat{1}{b_\gamma}{0}{1}$ and $\beta_m^{-1} \gamma \alpha_m =\smat{d'/d}{m^{-1}(b_\gamma dd'-b'd+bd')}{0}{d/d'}$. Thus $|A_{\beta_m^{-1}\gamma\alpha_m}(\tau)|< \delta$ implies
\begin{align}
    \left|\frac{d'}{d}\tau+ m^{-1}(b_\gamma dd'-b'd+bd')- \frac{d}{d'}\overline{\tau}\right| \le v\delta.
\end{align}
Separating real and imaginary parts, we get that
\begin{align} 
    |\frac{d'}{d}v+ \frac{d}{d'}v|\le v\delta\q \text{and}\q
    \left|\frac{d'}{d}u- \frac{d}{d'}u+ m^{-1}(b_\gamma dd'-b'd+bd')\right|\le v\delta. \label{b'-fix}
\end{align}
From \eqref{b'-fix}, the number of choices for $b,b'$ is seen to be 
\begin{align}
  1+ \frac{v \delta m}{d'}, \q   1+ \frac{v \delta m}{d} \text{ respectively.}
\end{align}
Further, $m^{-1}b_\gamma dd'$ is in an interval of length $v\delta$, thus there are $1+ mv\delta/dd'$ such $b_\gamma$. Thus the number of $(\gamma, \alpha_m, \beta_m)$ is at most
\begin{align}
    \sum_{d',d|m} (1+\frac{mv \delta }{d'} ) (1+ \frac{mv \delta }{d}) \left(1+ \frac{mv \delta }{dd'} \right) .
\end{align}
Note that $mv\le mv' \ll N^\epsilon$. Thus in this case, we get
\begin{equation}
    \mc C(\tau, m, \delta)|_{c_\gamma=0}\ll N^\epsilon \delta^3 \sum_{d,d'|m} 1 \ll N^\epsilon \delta^3.
\end{equation}
Next, we consider the case $c_\gamma\neq 0$. We have that
\begin{align} \label{g-def}
\beta_m^{-1}\gamma\alpha_m=\smat{A}{B}{C}{D} = \smat{\frac{-c_\gamma  b' + a_\gamma  d'}{d}}{\frac{-c_\gamma b b' - d_\gamma b' d + a_\gamma b d' + b_\gamma d d'}{m}}{\frac{mc_\gamma }{dd'}}{\frac{c_\gamma b + d_\gamma d}{d'}}.
\end{align}
Now, $|A_{\beta_m^{-1} \gamma \alpha_m}(\tau)| < \delta$ implies that $\left|(A\tau+B-C|\tau|^2-D\overline{\tau})v^{-1}\right|< \delta$. On separating the real and imaginary parts, we get
\begin{align}
    |A+D|< \delta ;\q |Au+B-C|\tau|^2-Du|< \delta v .
\end{align}
From $|A+D|<\delta$, we get
\begin{equation} \label{abc}
    \big|\frac{-c_\gamma  b' + a_\gamma  d'}{d}+\frac{c_\gamma b + d_\gamma d}{d'}\big|<\delta.
\end{equation}
We first note that $a_\gamma d_\gamma\equiv 1\mod N$ and thus $a_\gamma d_\gamma \neq 0$. By noting that $a_\gamma d'/d$ and $d_\gamma d/d'$ lie in a interval length $\delta$, we get that $ \displaystyle \#\{ d_\gamma\}\ll \frac{d'\delta}{d}$ and $\displaystyle \#\{ a_\gamma\}\ll \frac{d\delta}{d'}$.

We can, without loss, assume that both the above counts are at least $\ge 1$; otherwise, there exists no $\gamma$ and the count is $0$. These in unison then give the strong bounds $\displaystyle d/d' \ll \delta, \q d'/d \ll \delta$. This yields, $ \displaystyle \#\{a_\gamma, d_\gamma\}\ll \delta^2$.

Next, to count the $c_\gamma$, we note that $c_\gamma(db-d'b')/dd'$ lies in an interval of length $\delta$. Combining this with our choices of $b,b'$ and the fact that $N|c_\gamma$, we get that
\begin{equation}
    \#\{c_\gamma\}\le \frac{\delta dd'}{N(db-d'b')} \le \frac{\delta dd'}{Nd^2}\le \frac{\delta^2 }{N}.
\end{equation}
Combining the counts of $a_\gamma$, $b_\gamma$ and $c_\gamma$, we get that $\#\{\gamma : c_\gamma\neq 0\} \ll \delta^4/N$.

The number of choices for $b'$ is seen to be at most
\begin{align}
    1+ \frac{d \delta}{c_\gamma} \le 1+ \frac{d \delta}{N}.
\end{align}
Thus the total count for $(\alpha_m, \beta_m)$ this time is
\begin{align}
     \sum_{d|m} d (1+ \frac{d \delta}{N}) \sum_{d'|m} 1 \ll m + \frac{m^2 \delta}{N}.
\end{align}
Thus we have the following estimate for the total count.
\begin{equation}
    \mc C(\tau, m, \delta)|_{c_\gamma\neq 0}\ll \left(m + \frac{m^2 \delta}{N}\right) \cdot \frac{\delta^4}{N}.
\end{equation}
We conclude the above discussion in the following lemma.
\begin{lem}\label{lem:count}
    Let $Y\gg 1/N$ be reduced and $\mc C(\tau, m, \delta)$ be as above. Then for $m\ll N^\epsilon/v'$ we have
    \begin{equation}
        \mc C(\tau, m, \delta)\ll N^\epsilon \delta^3 + \left(m + \frac{m^2 \delta}{N}\right) \cdot \frac{\delta^4}{N}.
    \end{equation}
\end{lem}
We now return to estimating the quantity $\mc M(\tau, z)$. Recall that
\begin{equation}
\mc M(\tau, z) \ll \sum_m \frac{v^{k}e^{-4\pi my^2/v}}{m^{-k+1}}\big|\sum_{\alpha_m, \beta_m} (dd')^{-k} B_J(\alpha_m\tau, mz/d;\beta_m\tau, mz/d')\big|.
\end{equation}
Using \eqref{BKbydef} and \eqref{thetabound} we get
\begin{align}
    \mc M(\tau, z)&\ll \sum_m \frac{v^{k}}{m^{-k+1}}\sum_{\alpha_m,\beta_m }(dd')^{-k} \left(1+\frac{d}{(mv)^{1/2}}\right)\left(1+\frac{d'}{(mv)^{1/2}}\right) \left| M(\alpha_m\tau,\beta_m\tau)\right| \n \\
    & = \sum_m \frac{v^{k}}{m^{-k+1}}\sum_{\alpha_m, \beta_m}(dd')^{-k} \left(1+\frac{d}{(mv)^{1/2}}\right)\left(1+\frac{d'}{(mv)^{1/2}}\right)\left(\frac{dd'}{mv}\right)^{r} \sum_{ \gamma\in \Gamma_0(N)} \, \big |A_{\beta_m^{-1} \gamma \alpha_m}(\tau)|^{-r} \n \\
    &= \sum_m \sum_{\alpha_m, \beta_m} \frac{v^{1/2}}{(mdd')^{1/2}}\left(1+\frac{d}{(mv)^{1/2}}\right)\left(1+\frac{d'}{(mv)^{1/2}}\right)\sum_{ \gamma\in \Gamma_0(N)} \, \big |A_{\beta_m^{-1} \gamma \alpha_m}(\tau)|^{-r}\n \\
    &= \sum_m \sum_{\alpha_m, \beta_m} \frac{v^{1/2}}{m^{1/2}}\left(\frac{1}{d^{1/2}}+\frac{d^{1/2}}{(mv)^{1/2}}\right)\left(\frac{1}{d^{1/2}}+\frac{d'^{1/2}}{(mv)^{1/2}}\right)\sum_{ \gamma\in \Gamma_0(N)} \, \big |A_{\beta_m^{-1} \gamma \alpha_m}(\tau)|^{-r} \n \\
    &\ll \sum_m \sum_{\alpha_m, \beta_m} \frac{v^{1/2}}{m^{1/2}}\left(1+\frac{1}{v^{1/2}}\right)^2\sum_{ \gamma\in \Gamma_0(N)} \, \big |A_{\beta_m^{-1} \gamma \alpha_m}(\tau)|^{-r} \n \\
    &\ll v^{-1/2} \sum_m  \frac{1}{m^{1/2}} \sum_{\alpha_m, \beta_m}\sum_{ \gamma\in \Gamma_0(N)} \, \big |A_{\beta_m^{-1} \gamma \alpha_m}(\tau)|^{-r}. \label{mtau-bd}
\end{align}
We split the sum over $(\gamma, \alpha_m, \beta_m)$ into dyadic intervals as follows.
\begin{equation}
\begin{split}
    \sum_{\alpha_m, \beta_m}\sum_{ \gamma\in \Gamma_0(N)} \, \big |A_{\beta_m^{-1} \gamma \alpha_m}(\tau)|^{-r} &\ll \sum_{t=1}^{\infty} 2^{-rt} \mc C(\tau, m, 2^{t+1})\\
    & \ll \sum_{t=1}^{\infty} 2^{-rt}\left(N^\epsilon 2^{3t+3} + \left(m + \frac{m^2 2^{t+1}}{N}\right) \cdot \frac{2^{4t+4}}{N}\right)\\
    &\ll N^\epsilon\left(1+\frac{m}{N}+\frac{m^2}{N^2}\right).
\end{split}
\end{equation}
Recall that $m$ was supported on $[1, N^\epsilon/v']$. Thus from \eqref{mtau-bd} we get
\begin{equation}
\begin{split}
    \mc M(\tau, z)&\ll  N^\epsilon v^{-1/2} \sum_m  \frac{1}{m^{1/2}}\left(1+\frac{m}{N}+\frac{m^2}{N^2}\right)\ll N^\epsilon \left( \frac{1}{(vv')^{1/2}} + \frac{1}{Nv'^{3/2}}+\frac{1}{N^2v'^{5/2}}\right)\\
    &\ll N^{\epsilon}(vv')^{-1/2}\ll N^{1+\epsilon}.
\end{split}
\end{equation}
Now, coming back to the case $N=p$, we get
\begin{equation}
    \sup_{Z \in \mf H_2, Y \gg p^{-1}}\sumn_{F \in \mc B^{SK, new}(p)}  \frac{\det(Y)^k |F(Z)|^2}{\norm{F}^2_p} \ll p^{-2} \sup |\mc M(\tau, z)| \ll p^{-1+\epsilon} .
\end{equation}
\section{Contribution of oldforms in type 0 and 2 regions.}\label{old-cont-02}
Now we bound the size of oldspace $\skk^{old}(p)$ in the region $Y\gg 1/p$. First, we recall that $\mc B ^{SK, old}(p)=\left\{ G_{+,\phi}, G_{-,\phi}, H_\phi:  \phi\in \mc B_{k,1}^J\right\}$. Thus, the size of the oldspace in the region $Y\gg 1/p$ can be estimated as:
\begin{equation}
    \sum_{F\in \mc B^{SK, old}(p)}\frac{\det(Y)^k|F(Z)|^2}{\norm{F}^2_p}\ll \sum_{\phi\in \mc B_{k,1}^J}\frac{\det(Y)^k |G_{\pm, \phi}|^2}{\norm{G_{\pm, \phi}}^2_p}+\sum_{\phi\in \mc B_{k,1}^J}\frac{\det(Y)^k |H_{\phi}|^2}{\norm{H_{\phi}}^2_p}.
\end{equation}
From Proposition \ref{oldbasis}, the first quantity on the RHS can be written as-
\begin{equation}
    \sum_{\phi\in \mc B_{k,1}^J}\frac{\det(Y)^k |G_{\pm, \phi}|^2}{\norm{G_{\pm, \phi}}^2_p}\ll \frac{1}{p^3}\sum_{\phi\in \mc B_{k,1}^J}\frac{\det(Y)^k |F_{1, \phi}|^2}{\norm{F_{1, \phi}}^2_1}+\frac{1}{p^3}\sum_{\phi\in \mc B_{k,1}^J}\frac{\det(Y)^k |F_{3, \phi}|^2}{\norm{F_{1, \phi}}^2_1}.
\end{equation}
Now taking the sup over $\mf H_2$ (although we need it only in the region $Y\gg 1/p)$ and noting that $\sup(\mrm{SK}(1))$ is bounded, we get that
\begin{equation}
    \sup_{Z\in\mf H_2}\sumn_{\phi\in \mc B_{k,1}^J}\frac{\det(Y)^k |G_{\pm, \phi}|^2}{\norm{G_{\pm, \phi}}^2_p}\ll \frac{1}{p^3}.
\end{equation}
Next, we have that
\begin{align}
    \sumn_{\phi\in \mc B_{k,1}^J}\frac{\det(Y)^k |H_{\phi}|^2}{\norm{H_{\phi}}^2_p}\ll \frac{1}{p^{2k-4}}\sumn_{\phi\in \mc B_{k,1}^J}\frac{\det(Y)^k |G_{2,\phi}|^2}{\norm{F_{1,\phi}}^2_p}+\frac{|a_\pm|^2}{p^{2k-4}} \sumn_{\phi\in \mc B_{k,1}^J}\frac{\det(Y)^k |G_{\pm, \phi}|^2}{\norm{F_{1,\phi}}^2_p}.
\end{align}
From \eqref{apm} and \eqref{normpm}, we see that $a_\pm \ll p^{k-3/2}$. In addition, we have $G_{2,\phi}= G_{1,\phi}|U_S(p)$. Thus we get
\begin{align}
    \sumn_{\phi\in \mc B_{k,1}^J}\frac{\det(Y)^k |H_{\phi}|^2}{\norm{H_{\phi}}^2_p}&\ll \frac{1}{p^{2k-1}}\sumn_{\phi\in \mc B_{k,1}^J}\frac{\det(Y)^k |G_{1,\phi}|U_S(p)|^2}{\norm{F_{1,\phi}}^2_1}+\frac{1}{p^2} \sup(\mrm{SK}(1)).
\end{align}
Let $\psi_m=\phi|U_J(p)V_m$. Arguing as in the previous section, we get
\begin{equation}\label{Upbound-type2}
    \sumn_{\phi\in \mc B_{k,1}^J}\frac{\det(Y)^k |G_{1,\phi}|U_S(p)(Z)|^2}{\norm{F_{1,\phi}}^2_1} \ll \sumn_{\phi}\sumn_{mv'\ll p^\epsilon} \frac{ |\psi_m|^2 }{m^{k-1} } \frac{ v^k e^{- 4 \pi m y^2/v} }{\lan \phi,\phi \ran}.
\end{equation}
Writing the definitions of $U_J(p)$, $V_m$ and expanding the square, the RHS is equal to
\begin{equation}
  \mc M(\tau, z;p):=  \frac{1}{p^8}\sumn_{\phi}\sumn_m \frac{v^{k}e^{-4\pi my^2/v}}{m^{-k+1} }\sumn_{\alpha_{m,p}, \beta_{m,p}} (dd')^{-k} B_J(\alpha_{m,p}(\tau, z);\beta_{m,p}(\tau, z)),
\end{equation}
where $\alpha_{m,p}=[\smat{a/p}{(b+d\ell)/p}{0}{dp}, (a\lambda/m, (b\lambda+d\mu)/m)]$ and $\beta_{m,p}=[\smat{a'/p}{(b'+d'\ell')/p}{0}{d'p}, (a'\lambda'/m, (b'\lambda'+d'\mu')/m)]$ with $b \pmod d$, $b'  \pmod d'$, $\lambda, \lambda',\mu,\mu' \pmod p$ and $\ell ,\ell' \pmod {p^2}$. Here $B_J(\tau_1, z_1; \tau_2,z_2)$ denotes the Jacobi BK of level $1$. As in the previous section, we have
\begin{align}\label{BKbydefJK1}
    B_{J}(\tau_1,z_1;\tau_2,z_2)\ll e^{2 \pi (y_1^2/v_1+y_2^2/v_2)}(1+ v_1^{-1/2})(1+ v_2^{-1/2}) \sum_{ \gamma\in \sltwo} \, \big |A_{\beta_{m,p}^{-1} \gamma \alpha_{m,p}}(\tau)|^{-r}.
\end{align}
In our case, we have $v_1= mv/d^2p^2, v_2= mv/d'^2p^2$.

Now we bound the sum over $\gamma$ in \eqref{BKbydefJK1}. Write $r=k-\frac{1}{2}$ and define
\begin{equation}
    M(\alpha_{m,p}\tau,\beta_{m,p}\tau):= \im(\beta_{m,p}\tau)^{-r}\sum_{ \gamma\in \sltwo} \, \big |(\gamma\alpha_{m,p}\tau-\overline{\beta_{m,p}\tau})j(\gamma,\alpha_{m,p}\tau)\im(\beta_{m,p}\tau)^{-1}\big |^{-r}.
\end{equation}
Here  for the sake of simplicity, we write $\alpha_{m,p}, \beta_{m,p}$ to be the $\GL(2)$ parts of $\alpha_{m,p}$ and $\beta_{m,p}$ respectively. Then from Lemma \ref{lem:Agamma}, it is enough to consider $A_{\beta_{m,p}^{-1}\gamma\alpha_{m,p}}(\tau)$. We define the (slightly modified) counting function $\mc C(\tau, m, p, \delta)$ as follows.
\begin{equation} \label{count-with-wt}
    \mc C(\tau, m,p, \delta):= \#\{(\gamma, \alpha_{m,p}, \beta_{m,p}): \gamma\in \sltwo, |A_{\beta_{m,p}^{-1}\gamma\alpha_{m,p}}(\tau)|< \delta\}.
\end{equation}
From the Lemma \ref{lem:lowAg}, we see that $\mc C(\tau, m,p ,\delta)=0$ if $\delta\le 2$.

\textbf{Choice of $\ell$ and $\ell'$:} We choose the residue class for $\ell, \ell'$ as $p^{M+4}\le \ell < p^{M+4}+p^2$ and $0\le \ell' < p^2$ so that $(d\ell-d'\ell')\gg dp^{M+4}-d'p^2\gg p^{M+4}-p^{3+\epsilon}\gg p^M$.

Let $\gamma=\smat{a_\gamma}{b_\gamma}{c_\gamma}{d_\gamma}$. Consider the case when $c_\gamma=0$. In this case, $\gamma= \smat{1}{b_\gamma}{0}{1}$ and $\beta_{m,p}^{-1} \gamma \alpha_{m,p} =\smat{d'/d}{m^{-1}(p^2b_\gamma dd'-b'd+bd'-dd'(\ell'-\ell)}{0}{d/d'}$. Thus $|A_{\beta_m^{-1}\gamma\alpha_m}(\tau)|< \delta$ implies
\begin{align}
    \left|\frac{d'}{d}\tau+ m^{-1}(p^2b_\gamma dd'-(b'd-bd')-dd'(\ell'-\ell)- \frac{d}{d'}\overline{\tau}\right| \le v\delta.
\end{align}
Separating real and imaginary parts, we get that
\begin{align} 
    |\frac{d'}{d}v+ \frac{d}{d'}v|\le v\delta;\q\text{and}\q
    \left|\frac{d'}{d}u- \frac{d}{d'}u+ m^{-1}(p^2b_\gamma dd'-(b'd-bd')-dd'(\ell'-\ell))\right|\le v\delta. \label{b'-fixold}
\end{align}
\textbf{Case 1:} $b_\gamma=0$: In this case, we have
\begin{equation}
    \left|\frac{d'}{d}u- \frac{d}{d'}u+ m^{-1}(bd'-b'd-dd'(\ell'-\ell))\right|\le v\delta.
\end{equation}
From our choice of residue classes for $\ell, \ell'$, we see that
\begin{equation}
    dd'(\ell-\ell') \ll \delta(p^\epsilon+m)+dd'.
\end{equation}
Thus we get that
\begin{equation}
    p^{M-2}\ll \ell-\ell'\ll \delta(p^\epsilon+m).
\end{equation}
Next, we have $m\ll p^\epsilon/v'\ll p^{1+\epsilon}$. Thus this is possible only when $\delta> p^{M-3}$. 

When $\delta> p^{M-3}$, the number of $\ell, \ell'$ is bounded by $(1+v m\delta/dd')^2$. Similarly, the number of choices of $b ,b'$ is seen to be bounded by 
\begin{align}
  1+ \frac{v \delta m}{d'}, \q   1+ \frac{v \delta m}{d} \text{ respectively.}
\end{align}
Since $m \ll p^{\epsilon}/v' \ll p^{1+\epsilon}$, the number of $(\gamma, \alpha_{m,p}, \beta_{m,p})$ when $\gamma =1_2$ is at most
\begin{align}
    \sumn_{d',d|m} (1+\frac{mv \delta }{dd'} )^2(1+\frac{mv \delta }{d} ) (1+\frac{mv \delta }{d'} )\ll p^\epsilon \delta^4 \cdot \sigma_0(m)^2\ll p^\epsilon \delta^4,
\end{align}
\textbf{Case 2:} $b_\gamma\neq 0$: In this case, the number of $\ell, \ell'$ is bounded by $(1+v m\delta/dd')^2$. Similarly, the number of choices of $b ,b'$ is seen to be bounded by 
\begin{align}
  1+ \frac{v \delta m}{d'}, \q   1+ \frac{v \delta m}{d} \text{ respectively.}
\end{align}
Further, $m^{-1}p^2b_\gamma dd'$ is in an interval of length $v\delta$, thus there are $mv\delta/p^2dd'$ such $b_\gamma$. Thus the number of $(\gamma, \alpha_{m,p}, \beta_{m,p})$ is at most
\begin{align}
    \sumn_{d',d|m} (1+\frac{mv \delta }{dd'} )^2(1+\frac{mv \delta }{d} ) (1+\frac{mv \delta }{d'} )\frac{mv \delta }{p^2dd'} .
\end{align}
Note that $mv\le mv' \ll p^\epsilon$. Thus when $c_\gamma=0$, by choosing $M= 103$ (say) we get
\begin{equation}
    \# \mc C(\tau, m,p, \delta)|_{c_\gamma=0}\ll \begin{cases}
        p^{-2+\epsilon}\delta^5 & \text{ if } \delta< p^{100}\\
        p^\epsilon \delta^4+ p^{-2+\epsilon}\delta^5 & \text{ otherwise }.
    \end{cases}
\end{equation}
Next we consider the case $c_\gamma\neq 0$. We have that
\begin{align} \label{g-defold}
\beta_{m,p}^{-1}\gamma\alpha_{m,p}=\smat{A}{B}{C}{D} = \pmat{-c_\gamma  \frac{b'+d'\ell'}{p^2dd'} + a_\gamma  \frac{d'}{d}}{*}{\frac{mc_\gamma }{p^2dd'}}{c_\gamma\frac{ b+d\ell }{p^2dd'}+ d_\gamma \frac{d}{d'}}.
\end{align}

Separating the real and imaginary parts, we get
\begin{align}
    |A+D|< \delta ;\q |Au+B-C|\tau|^2-Du|< \delta v .
\end{align}
From $|A+D|<\delta$, we get
\begin{equation} \label{abcold}
    \left|-c_\gamma  \frac{b'+d'\ell'}{p^2dd'} + a_\gamma  \frac{d'}{d}+c_\gamma\frac{ b+d\ell }{p^2dd'}+ d_\gamma \frac{d}{d'}\right|<\delta.
\end{equation}
The number of $a_\gamma, d_\gamma$ are bounded by $1+d\delta/d'$ and $1+d'\delta/d$ respectively.

Next, to count the $c_\gamma$, we note that $c_\gamma(\frac{ b+d\ell }{p^2dd'}-\frac{b'+d'\ell'}{p^2dd'})$ lies in an interval of length $\delta$. Since $c_\gamma\neq 0$, we see that
\begin{equation}
    \# \{c_\gamma\}\ll \frac{p^2dd'\delta}{|(d\ell-d'\ell')+(b-b')|}.
\end{equation}
By making use of choice of $\ell$, $\ell'$ we get the following estimate for the count of $c_\gamma$.
\begin{equation}
    \# \{c_\gamma\}\ll p^{-M}dd'\delta.
\end{equation}
The number of $\ell, \ell'$ is bounded by $p^4$, and $b,b'$ are bounded by $d,d'$ respectively.
Thus the total count for $(\gamma, \alpha_{m,p}, \beta_{m,p})$ this time is
\begin{align}
     p^{-M+4}\sumn_{d',d|m} (dd')^2\delta(1+\frac{d\delta}{d'})(1+\frac{d'\delta}{d}) \ll p^{-M+4}(m^{4+\epsilon}\delta+ m^{3+\epsilon} \delta^3)\ll p^{-M+4}m^{4+\epsilon}\delta^3.
\end{align}

We conclude the above discussion in the following lemma.
\begin{lem}\label{lem:countold}
    Let $Y\gg 1/p$ be reduced and $\mc C(\tau, m,p, \delta)$ be as above. Then for $m\ll p^\epsilon/v'$ we have
   \begin{equation}
    \# \mc C(\tau, m,p, \delta)\ll \begin{cases}
        p^{-2+\epsilon}\delta^5+p^{-100}m^{4+\epsilon}\delta^3  & \text{ if } \delta< p^{100};\\
        p^\epsilon \delta^4+ p^{-2+\epsilon}\delta^5+p^{-100}m^{4+\epsilon}\delta^3 & \text{ otherwise }.
    \end{cases}
\end{equation}
\end{lem}
Now returning to the estimation of Jacobi BK of level $1$, we get
\begin{align}
    \mc M(\tau,z;p)&\ll \frac{p^{2k}v^{1/2}}{p^5}\sum_m m^{-1/2}\sum_{\alpha_{m,p}, \beta_{m,p}} \frac{1}{(d'd)^{1/2}} \left(1+\frac{dp}{(mv)^{1/2}}\right)\left(1+\frac{d'p}{(mv)^{1/2}}\right)M(\alpha_{m,p}\tau,\beta_{m,p}\tau)\\
    &\ll p^{2k-3}v^{-1/2}\sumn_m m^{-1/2} \sumn_{\alpha_{m,p},\beta_{m,p}} M(\alpha_{m,p}\tau,\beta_{m,p}\tau).
\end{align}
We split the sum over $(\gamma, \alpha_{m,p}, \beta_{m,p})$ into dyadic intervals as follows.
\begin{equation}
\begin{split}
    \sum_{\alpha_{m,p}, \beta_{m,p}}\sum_{ \gamma\in \sltwo} \, \big |A_{\beta_{m,p}^{-1} \gamma \alpha_{m,p}}(\tau)|^{-r}\ll \sum_{t=1}^{\infty} 2^{-rt} \mc C(\tau, m,p, 2^{t+1}).
\end{split}
\end{equation}
Further, splitting the sum according as $(2^t=) \delta< p^{100}$ or not and using the bounds for $\mc C(\tau, m,p, 2^{t+1})$ from Lemma \ref{lem:countold} we get that RHS is bounded by
\begin{equation}
    \ll\sum_{t=1}^{50\log p} 2^{-rt} \mc C(\tau, m,p, 2^{t+1}) + \sum_{t> 50\log p} 2^{-rt} \mc C(\tau, m,p, 2^{t+1})\ll p^{-2+\epsilon}.
\end{equation}Thus 
\begin{align}
    \mc M(\tau,z;p) \ll p^{2k-5+\epsilon}v^{-1/2}\sumn_m m^{-1/2}\ll p^{2k}(vv')^{-1/2} p^{-5+\epsilon} \ll p^{2k-4+\epsilon}.
\end{align}
In conclusion, we get that
\begin{align}
    \sum_{\phi\in \mc B_{k,1}^J}\frac{\det(Y)^k |H_{\phi}|^2}{\norm{H_{\phi}}^2_p} &\ll \frac{1}{p^{2k-1}}\cdot p^{2k-4+\epsilon}+\frac{1}{p^2}\sup(\mrm{SK}(1))\ll \frac{1}{p^2}.
\end{align}
Thus in regions 0 and $2$, the oldspace has the following bound.
\begin{equation}
   \sup_{Z\in \mf H_2, Y\gg 1/p} \sum_{F\in \mc B^{SK, old}(p)}\frac{\det(Y)^k|F(Z)|^2}{\norm{F}^2_p}\ll \frac{1}{p^{2}}.
\end{equation}
\section{Type~1 region.}\label{avgtype1}
Let us write $Z=X+iY \in \mc F_2(p) :=\Gamma_0^{(2)}(p) \backslash \mf H_2$, so in particular $Y$ is Minkowski-reduced. Write
$  Z = \smat{\tau}{z}{z}{\tau'}$
with $u+iv=\tau, \tau'=u'+iv' \in \mf H$ and $z=x+iy \in \mf C$. Further, let us set $|Y| = \det(Y)$. It is also convenient to introduce the parameter $\mrm t$ by defining
\begin{equation}
{\mrm t} = \tz := |Y|/v = v'-y^2/v.
\end{equation}
If $Z \in \mc F_2$, then it follows from the reduction theory that $v,v' \geq \sqrt{3}/2$ (cf. \cite{Fr}). Moreover, for the same reason, ${\mrm t} = |Y|/v \gg v' \gg 1$. This, however, need not be true in $\fp$. Recall that in the Type~1 case that we are in, the coset representatives are of the form $\alpha_{A, B}:=w_1 n(B)m(A)$. Then we have 
\begin{align} \label{type2-1}
    \sup_{ Z \in {\alpha_{A,B} \mc F_2} } \sum_{F}\frac{\det(Y)^{k} |F(Z)|^2}{\norm{F}^2_p}= \sup_{Z \in \mc F_2}\sum_F  \frac{\det(Y)^{k} |F | \alpha_{A,B}(Z)|^2}{\norm{F}^2_p} .
 \end{align}
 
\begin{align}
     w_1 \begin{pmatrix}
\tau & z \\ z & \tau'
\end{pmatrix} = \pmat{- 1/\tau }{z/\tau}{z/\tau}{\tau'- z^2/\tau}.\label{JI}
\end{align}
For future reference, let us put, for $Z \in \mc F_2$
\begin{align}
    W = W_{A,B} := n(B) m(A) Z \q \text{and write  } W = \smat{w_1}{w_2}{w_2}{w_4}, \im(W)= \mf Y = \smat{\my_1}{\my_2}{\my_2}{\my_4}  .
\end{align}
We then have from \eqref{JI} that (note that $\det Y = \det \mf Y$ since $k$ is even)
\begin{align} \label{FtoFJ}
\sum_F \frac{\det(Y)^k |F|w_1(W)|^2}{\norm{F}^2_p} = \sum_F \frac{\det(\mf Y)^k}{\norm{F}^2_p} |\sum_{m \ge 1} (\phi_{F,m} |(J, [0,0]))(w_1,w_2) \cdot e(m w_4)|^2 ,
\end{align}
where $J=\smat{0}{-1}{1}{0}\in\sltwo$.
We would now need the following identity which is valid for any $ Z \in \mf H_2 $ and any function $\phi \colon \mf H \times \mf C \rightarrow \mf C$. We provide a proof for the convenience of the reader.
\begin{lem}
For $\phi$ as above and $\tilde{\phi}$ is as defined in \eqref{inva}, we have
\begin{align} \label{jeqiuvariance}
    v^{k/2} \abs{(\phi|_{k,m}(J, [0,0]))(\tau, z)} = e^{2\pi my^2/v} \widetilde{\phi}(-1/\tau, z/\tau).
\end{align}
\end{lem}
\begin{proof}
    We start from the RHS, and look at
\begin{align}
    \widetilde{\phi}(-1/\tau, z/\tau) = (\frac{v}{|\tau|^2})^{k/2} e^{- 2 \pi m \frac{|\tau|^2}{v} (\frac{yu-xv}{|\tau|^2})^2 } |\phi(-1/\tau,z/\tau)|,
\end{align}
whereas
\begin{align}
    v^{k/2} \left|\phi | J\right| &=  (\frac{v}{|\tau|^2})^{k/2} e^{2 \pi m \im (z^2/\tau)} |\phi(-1/\tau,z/\tau)| = (\frac{v}{|\tau|^2})^{k/2} e^{2 \pi m  (\frac{-v(x^2-y^2) + 2xyu}{|\tau|^2} ) } |\phi(-1/\tau,z/\tau)|.
\end{align}
Then one notes that
\begin{align*}
    \frac{-v(x^2-y^2) + 2xyu}{|\tau|^2} + (\frac{yu-xv}{v|\tau|^2})^2 = \frac{1}{|\tau|^2} (-v(x^2-y^2) + 2xyu + \frac{y^2 u^2}{v} - 2 xyu + x^2 v   ) = \frac{2 \pi m y^2}{v}.
\end{align*}
This completes the proof of the lemma.
\end{proof}
\subsection{Contribution of newforms.}\label{newformtype1}
To proceed further, using the Fourier-Jacobi expansion $F$ and the relation between the Petersson norms of $F$ and the corresponding Jacobi form as in \eqref{petreln}, we see that
\begin{align}
         \sum_{F \in \mc B^{SK,new}(p)} \frac{\det(Y)^{k} |F(Z)|^2}{\norm{F}^2_p} &\le
    \left( \sumn_{m}   t^{k/2} p_m(\tau,z)^{1/2} \exp(- 2 \pi m t)    \right)^2 \\
    & \ll p^{-2}  \left( \sumn_{m}   t^{k/2} p_m^J(\tau,z)^{1/2} \exp(- 2 \pi m t)  \right)^2, \label{n2-}
\end{align}
where we have put (with $\tilde{\phi}$ as in \eqref{inva})
\begin{align}
    p_m(\tau,z) &:=  \sum_{F \in \mc B^{SK,new}(p)} \frac{\widetilde{\phi_{m,F}(\tau,z)}^2}{\norm{F}^2_p};\q p_m^J(\tau,z) := \sum_{\phi_{1,F} \in \mc B^{J,new}_{k,1}(p)}  \frac{\widetilde{\phi_{m,F}(\tau,z)}^2}{\lan \phi_{1,F}, \phi_{1,F}\ran_p}.
\end{align}
We notice that $p_J(m)$ is related to the Bergman kernel for the space $\jkn$. Keeping in mind that $\phi_{m,F} = V_m(\phi_{1,F})$, we calculate
\begin{align}\label{Vmtilde}
    \widetilde{\phi|V_m}^2 & =v^{k}e^{-4\pi my^2/v}|\phi|V_m|^2
     = m^{-2} v^{k}e^{-4\pi my^2/v} \left|\sum_{ad=m}\sum_{b\bmod d}a^{k}\phi\left(\frac{a\tau+b}{d},az\right)\right|^2\\
    &= m^{-2} \left|\sum_{ad=m}\sum_{b\bmod d}(a d)^{k/2} \widetilde{\phi}\left(\frac{a\tau+b}{d},az\right)\right|^2.
\end{align}
Here we have used the following relation:
\begin{align}
\widetilde{\phi} \left(  \frac{a\tau+b}{d}, az\right) &= \lbr \frac{av}{d} \rbr^\frac{k}{2} e^{-2 \pi da^2y^2/av} |\phi \left(  \frac{a\tau+b}{d}, az\right) |=\lbr \frac{av}{d} \rbr^\frac{k}{2} e^{-2 \pi my^2/v} |\phi \left(  \frac{a\tau+b}{d}, az\right)  |. \label{tildephi}
\end{align}
Thus (for the sake of simplicity we assume $\lan \phi, \phi \ran =1$ here)
\begin{align}
    p_J(m) &\le m^{ k -2 } \sum_{ \phi } \left( \sum_{ad=m} d \right) \left( \sum_{ad=m} \sum_{b\bmod d} \widetilde{\phi}\left(\frac{a\tau+b}{d},az\right)^2 \right)\\
    &\ll m^{ k-1 + \epsilon } \left( \sum_{ad=m} \sum_{b\bmod d} \sum_{\phi} \widetilde{ \phi }\left(\frac{a\tau+b}{d},az\right)^2 \right)\\
    &\ll m^{k+\epsilon} \sup_{(\tau,z)\in \mf H \times \mf C} \sum_{\phi \in \mc B^J_{k,1}(p)} \widetilde{ \phi } (\tau,z)^2 = m^{k+\epsilon} \sup(\jkp). \label{pjm-supjk}
\end{align}
Combining \eqref{FtoFJ} and \eqref{jeqiuvariance} we can write (recall $\im(W)= \mf Y = \smat{\my_1}{\my_2}{\my_2}{\my_4}$)  
\begin{align}
    \sum_F \frac{\det(Y)^{k} |(F | \alpha_{A,B})(Z)|^2}{\norm{F}^2_p}
    & \le \sum_F\frac{(\my_4-\my_2^2/\my_1)^{k}}{\norm{F}^2_p} |\sumn_{m \ge 1} e^{-2\pi m(\my_4- \my_2^2/\my_1)} \widetilde{\phi_m}(-1/w_1, w_2/w_1)|^2 \\
    &\le p^{-2}\tw^{k} (\sumn_{m \ge 1} e^{-2\pi m\tw} p_m^J(-1/w_1, w_2/w_1)^{1/2})^2 \\
    & \le p^{-2} \tw^{k} \sup(\jkp) (\sumn_{m \ge 1} e^{-2\pi m\tw} m^{k/2+\epsilon} )^2 .\label{klnparab1}
\end{align}
Let us note that the function $f(x):=x^{k/2+\varepsilon}e^{-2 \pi  x\mrm t}$ increases up to $x=\Upsilon:= \frac{k/2+\varepsilon}{ 2 \pi \mrm t} $ and decreases thereafter. Using the fact that $\sum_{m=1}^\infty f(m) \leq \int_0^\infty f(t) dt + 2 f(\Upsilon)$ when $   \Upsilon \geq 1$ (see \cite{Xia}), that (with $ \mrm t := \tw$)
\begin{align}
& {\mrm t}^{k/2} \lbr \sum_{m=1}^\infty m^{k/2+\epsilon}  e^{-2 \pi m {\mrm t}}  \rbr 
\ll_\varepsilon {\mrm t}^{k/2} \lbr  \int_0^\infty x^{k/2+\varepsilon} e^{-2 \pi {\mrm t} x} dx + 2 (\frac{k/2+\varepsilon}{2 \pi {\mrm t}} )^{k/2+\varepsilon} e^{-   k/2 -\varepsilon} \rbr 
\\
& \ll_\varepsilon {\mrm t}^{k/2} \lbr (2 \pi {\mrm t})^{-k/2-1-\varepsilon} \Gamma(k/2+1+\varepsilon)  + 2 (\frac{k/2+\varepsilon}{2 \pi {\mrm t} } )^{k/2+\varepsilon} e^{-   k/2 -\varepsilon}  \rbr  \n \\
& \ll_\varepsilon \lbr \frac{\Gamma(k/2+1+\varepsilon)  }{ {\mrm t} \, }   + (k/2+\varepsilon)^{k/2+\varepsilon} e^{-   k/2 -\varepsilon}  \rbr (2 \pi)^{-k/2} 
 \ll_{k, \varepsilon} \max(1/ \mrm{t}, 1).
\label{Fphi}
\end{align}
The case $  \Upsilon < 1$ is similar, and we get the same bound as \eqref{Fphi} using the inequality $\sum_{m=1}^\infty f(m) \leq \int_0^\infty f(t) dt + f(1)$ and $f(1) \leq  f(\Upsilon)$. 
We need an upper bound for the quantity $\tw^{-1}$. Recall from our definition ($\im(W)= \mf Y = \smat{\my_1}{\my_2}{\my_2}{\my_4}$) that 
\begin{align}
\tw^{-1} =  \my_1 / \det( \mf Y)= (AYA^t)_{11}/ \det(Y) \ll (a_1^2+a_2^2) y_2/y_1y_2 \ll (a_1^2+a_2^2),
\end{align}
since $W=AZA^t+B$. We have to thus choose the `lifts' $\{A\}$ from their corresponding images modulo $p$ carefully so that $\max_{A}(a_1^2+a_2^2)$ is as small as possible in terms of $p$. The projective line $\mathbb P^1(\pfp)$ over $\pfp$ can be parametrized by $\{ (0,1), (1,x)\}$, $x \in \pfp$. So we can lift these to $\GL(2, \mf Z)$ by
\begin{align}
    \pmat{1}{0}{0}{1}, \q \pmat{0}{-1}{1}{x} \, (x \in \pfp).
\end{align}
Therefore we can and will choose $A$ as above so that $\tw^{-1} \ll 1$. From \eqref{type2-1}, \eqref{pjm-supjk}, \eqref{klnparab1} and the above discussion therefore we get
\begin{align}
    \sup_{ Z \in {\alpha_{A,B} \mc F_2} } \sum_{F \in \mc B^{SK,new}(p)} \frac{\det(Y)^{k} |F(Z)|^2}{\norm{F}^2_p} \ll p^{-2} \cdot  \sup(\jkp).
\end{align}
\subsubsection{Contribution of oldforms:}\label{old-cont-1} We have that
\begin{equation}\label{oldtype1}
    \sum_{F\in \mc B^{SK, old}(p)}\frac{\det(Y)^k|(F|\alpha_{A,B})(Z)|^2}{\norm{F}^2_p}= \sum_{\phi\in \mc B_{k,1}^J}\frac{\det(Y)^k |(G_{\pm, \phi}|\alpha_{A,B})(Z)|^2}{\norm{G_{\pm, \phi}}^2_p}+\sum_{\phi\in \mc B_{k,1}^J}\frac{\det(Y)^k |(H_{\phi}|\alpha_{A,B})(Z)|^2}{\norm{H_{\phi}}^2_p}.
\end{equation}
Recall that $G_{\pm, \phi}= (1\mp \frac{1}{p})(F_{1,\phi} \pm F_{1,\phi}|W_p)$. Thus
\begin{equation}
    \sum_{\phi\in \mc B_{k,1}^J}\frac{\det(Y)^k |(G_{\pm, \phi}|\alpha_{A,B})(Z)|^2}{\norm{G_{\pm, \phi}}^2_p}\ll \frac{1}{p^3}\sum_{\phi\in \mc B_{k,1}^J} \frac{\det(Y)^k|F_{1,\phi}(Z)|^2}{\norm{F_{1,\phi}}^2_1}+ \frac{1}{p^3}\sum_{\phi\in \mc B_{k,1}^J} \frac{\det(Y)^k|(F_{1,\phi}|W_p\alpha_{A,B})(Z)|^2}{\norm{F_{1,\phi}}^2_1}.
\end{equation}
The first term is clearly bounded by $\sup(\mrm{SK}(1))$. For the second term, we note that
\begin{equation}
    \sup_{Z\in\mc F_2} \sum_{\phi\in \mc B_{k,1}^J} \frac{\det(Y)^k|(F_{1,\phi}|W_p\alpha_{A,B})(Z)|^2}{\norm{F_{1,\phi}}^2_1}\ll \sup_{Z\in \mf H_2} \sum_{\phi\in \mc B_{k,1}^J} \frac{\det(Y)^k|F_{1,\phi}(Z)|^2}{\norm{F_{1,\phi}}^2_1}= \sup(\mrm{SK}(1)).
\end{equation}
Thus 
\begin{equation} \label{g+-bd}
    \sup_{Z\in\mc F_2}\sum_{\phi\in \mc B_{k,1}^J}\frac{\det(Y)^k |(G_{\pm, \phi}|\alpha_{A,B})(Z)|^2}{\norm{G_{\pm, \phi}}^2_p}\ll \frac{1}{p^3}.
\end{equation}
Now we are left with bounding the $\{H_\phi\}$ contribution in \eqref{oldtype1}. Towards this, first note that $|H_\phi|\le |G_{2,\phi}|+p^{k-3/2}|F_{1,\phi}|+p^{k-3/2}|F_{1,\phi}|W_p|$. The terms containing $F_{1,\phi}$ and $F_{1,\phi}|W_p$ can be dealt with as above, and we get the bound $p^{-5}$ for their contribution. Thus we are left with bounding
\begin{equation}
    \mc U(Z):=\frac{1}{p^{2k-4}}\sum_{\phi\in \mc B_{k,1}^J}\frac{\det(Y)^k |(G_{1,\phi}|U_S(p)|\alpha_{A,B})(Z)|^2}{\norm{F_{1,\phi}}^2_p}.
\end{equation}
Since the lifting map $\mc L_p$ and the Hecke operator $U_S(p)$ commute, $G_{2,\phi}= \mc L_p(\phi|U_J(p))$. Thus from \eqref{FtoFJ}, we get that 
\begin{equation}
    \mc U(Z)= \frac{1}{p^{2k-4}}\sum_{\phi\in \mc B_{k,1}^J}\frac{\det(\mf Y)^k}{\norm{F_{1,\phi}}^2_p}|\sum_{m \ge 1} ((\phi|U_J(p))|V_m |(J, [0,0]))(w_1,w_2) \cdot e(m w_4)|^2.
\end{equation}
We have that $\norm{F_{1,\phi}}^2_p\asymp p^3 \norm{\phi}^2_1$. Thus using the same arguments as in subsection \ref{newformtype1}, we get that
\begin{equation} \label{Uptype1}
    \mc U(Z)\ll \frac{1}{p^{2k-1}}\sup_{\mf H\times \mf C}\sum_{\phi} \frac{v^k e^{-4\pi y^2/v}|\phi|U_J(p)|^2}{\norm {\phi}_1^2}.
\end{equation}
Note that the sum over $\phi$ on the RHS is the term corresponding to $m=1$ in \eqref{Upbound-type2}. Thus using the same counting arguments for the inequalities as in \eqref{b'-fixold} and \eqref{abcold} with $m=d=d'=1$, we see that (with the same notations)
\begin{equation}
     \# \mc C(\tau, 1,p, \delta) \ll v\delta(1+v\delta)(1+v\delta/p^2)+ p^{-M}\delta^3\ll v\delta(1+v\delta)^2.
\end{equation}
Thus the sum over $\phi$ in \eqref{Uptype1} is
\begin{align}\label{BJbound}
   \ll p^{2k-5}v^{1/2} (1+pv^{-1/2})^2 v(1+v^2)\ll p^{2k-5} (v^{3/4}+pv^{1/4})^2 (1+v^2).
\end{align}
Since the Fourier expansion of a Jacobi cusp form decays sub-exponentially for $v\gg p^\epsilon$, we can restrict ourselves to the region $v\ll p^\epsilon$. Thus we get
\begin{align} \label{ujp-bd1}
   \sup_{\mf H\times \mf C}\sum_{\phi} \frac{v^k e^{-4\pi y^2/v}|\phi|U_J(p)|^2}{\norm {\phi}_1^2} \ll p^{2k-3+\epsilon}.
\end{align}
Thus using \eqref{ujp-bd1} in \eqref{Uptype1}, we get
\begin{equation} \label{uz-bd}
   \mc U(Z)\ll p^{-2+\epsilon}.
\end{equation}
Putting together the bounds from \eqref{g+-bd} and \eqref{uz-bd}, the size of the oldspace in the region defined by Type 1 representatives can be estimated as 
\begin{equation}
    \sup_{Z\in\mc F_2}\sum_{F\in \mc B^{SK, old}(p)}\frac{\det(Y)^k|(F|\alpha_{A,B})(Z)|^2}{\norm{F}^2_p}\ll \frac{1}{p^{2-\epsilon}}.
\end{equation}

\section{Appendix~1: The case \texorpdfstring{$n=1$}{n=1}}\label{appendix}
For a newform $f$ of square-free level $N$, one can use Rankin-Selberg theory, as shown in \cite{blo-hol}, to obtain $\norm{f}_\infty \ll_\epsilon N^\epsilon$. The best known bound for $f$ is $\norm{f}_\infty \ll_\epsilon N^{-1/4+\epsilon}$ from  \cite{steiner}. Recall that we are interested in the $L^\infty$-size of $S_k(N)$, which is measured by the quantity $ \sup(S_k(N)):=\sup_\h \sumn_f v^k |f|^2 $. Squaring and adding these results leads to the bound $ \sup(S_k(N)) \ll N^{1/2+\epsilon} $ at best. It was
shown in \cite{jorgenson2004bounding} that the size in question is $O_k(1)$ when $k=2$ and $N$ is square-free, using geometric methods. Here we want to use classical analytical methods to obtain results of the same strength for all $\kappa > 2$ and all levels $N$, including half-integral weights. 

Let $\kappa>2$ and $S_\kappa(N)$ denote the space of holomorphic cusp forms of weight $k$ and level $N$. We would show that $\sup (S_\kappa(N)) \asymp_k 1$ for all $N$ if $\kappa \in \z$ and for all $N$ square-free otherwise.
We deal with the square-free and non-square-free levels separately. Note that via spectral large sieve, one can easily obtain the bound $\sup (S_\kappa(N)) \ll_\epsilon N^\epsilon$, cf. \cite{michel1998points}, \cite{lam2014local}. The lower bound $\sup (S_\kappa(N)) \gg_k 1$ follows trivially from the fact that the first coefficient $p_{k,N}(1)$ of the Poincar\'e series  $P_{k,N}(1)$ for $\Gamma_0(N)$, satisfies $p_{k,N}(1) \gg_k 1$, see \eqref{pkN} and \cite[Sec.~7.1]{sd-hk}. In what follows, we focus only on the upper bound.

\subsection{\texorpdfstring{$N$}{N} square-free.}
Consider the following region $\mc S$ from \cite{HT}. Let $A_0(N)$ be the group inside $\sltwor$ generated by the Atkin-Lehner involutions of level $N$ and $\Gamma_0(N)$.
\begin{equation}
    \mc S=\{\tau\in\h : \im ( \gamma\tau) \le \im (\tau) \text{ for all } \gamma\in A_0(N) \}.
\end{equation}
If $\tau \in \mc S$, then $\im (\tau) \ge \sqrt{3}/2N$. Since each term in the Bergman kernel is invariant under the group $A_0(N)$ and the Atkin-Lehner operators are isometries, we can restrict our attention to $\mc S$. 

In the half-integral weight case, notice that similarly, it is enough to find an upper bound -- however simultaneously for $B_{\kappa}(N, \chi)(\tau)$ for all Dirichlet characters $\chi \bmod N$, even if we start with the trivial character (here $\chi$ quadratic is enough). We state without proof the following assertion, catered to our needs, whose proof is the same as in \cite{Kir}; applying it to $\displaystyle B_k(N,\chi)(\tau)$ in place of $\tilde F(\tau):= v^{\kappa/2} |f(\tau)|$. Only the automorphy properties of the functions play a role here.

 Let $N$ be odd and square-free. Let $A_0(2N)$ denote the group generated by $\Gamma_0(2N)$ and the Atkin-Lehner involutions $W(d;N)$. Then we have the following result. Define the set
 \begin{equation}
   \mc S_{1/2}=\{ \tau \in\h :  \im ( \tau)\ge \frac{\sqrt{3}}{4N} \text{ and } |c \tau+d |^2 \ge \frac{1}{2N} , (c,d)\in \mf Z^2-(0,0)\}.
\end{equation}

\begin{prop}[\cite{Kir}] 
   Let $\mc A=\{W(d;N), \smat{1}{0}{2N}{1}\}$. Then \[\max_\chi~ \max_{A\in \mc A} ~\sup_{\tau\in \mf H}B_{\kappa}(N, \chi)(\tau)|A\] is attained in $\mc S_{1/2}$.
\end{prop}
With this setting, in the next subsection, we consider the size of $S_\kappa(N)$ when $\kappa$ is integral and not separately deal with the case $\kappa$ half-integral, as the proofs are entirely analogous. We only mention in the passing some notable points, if any.

\subsubsection{The case $k>2$} Let $\mc B_{k}(N)$ denote an orthonormal basis for $S_k(N)$. Then the Bergman kernel for the space $S_k(N)$ is given by
\begin{equation}\label{BKn=1}
    \mbb B_{k,N}(\tau):= \sumn_{f\in \mc B_{k}(N)}v^k |f(\tau)|^2.
\end{equation}
\subsubsection*{The region $v\ge 1$ via the Fourier expansion:}
Using the Fourier expansion, we can write
\begin{equation}
    \mbb B_{k,N}(\tau)\ll v^k \left(\sumn_n p_{k,N}(n)^{1/2} \exp(-2\pi n v)\right)^2.
\end{equation}
We have (cf. \cite{iwaniecaut})
\begin{equation} \label{pkN}
p_{k,N}=\frac{(4\pi n)^{k-1}}{\Gamma(k-1)}\left(1+O(n^\epsilon (n,N)^{1/2} N^{-1+\epsilon})\right). 
\end{equation}
Write
\begin{equation}
    \mc Q=\sumn_{n>0} \left(\frac{4\pi vn}{k}\right)^{k/2}\exp(k/2-2\pi v n).
\end{equation}
Then from \cite[Lemma 7.8]{sd-hk} we have $\mc Q\ll _k (1+1/v)$. In our case, i.e., when $v\ge 1$,  we have
\begin{equation}\label{BKNbound}
    \mbb B_{k,N}(\tau)\ll_k \mc Q^2 + N^{-1+\epsilon} \mc Q^2\ll _k \left(1+\frac{1}{N^{1-\epsilon}}\right)\ll _k 1.
\end{equation}
Thus, the BK is bounded in the region $v\ge 1$. For the rest of the arguments, we assume $v\le 1$.

For any such $\tau\in \mc S$ with $v\le 1$, consider the BK for the space of weight $k$ and level $N$ cusp forms given by
\begin{equation}
   \mbb B_{k,N}(\tau)= \frac{k}{4 \pi} \sumn_{\gamma\in\Gamma_0(N)}v^k\  \left(j(\gamma,\tau)(\gamma(\tau)-\overline{\tau})\right)^{-k}.
\end{equation}
We write $\tilde{\tau}=\gamma(\tau)$ and note that $\tilde\tau\in\mc S$ and thus $\tilde v:=\im(\tilde\tau)<v$. Then we see that
\begin{equation}
   \mbb B_{k,N}(\tau)= \sumn_{\gamma\in\Gamma_0(N)}(v \tilde v)^{k/2}\  (\tilde \tau-\overline{\tau})^{-k}.
\end{equation}
Now, any $\gamma\in\Gamma_0(N)$ can be written as $\gamma=\gamma_\infty\cdot\gamma_1$ with $\gamma_\infty\in \sltwo_\infty$ and $\gamma_1=\smat{*}{*}{c}{d}\in\Gamma_0(N)$. Thus $\tilde \tau =\tau_1+n$ for some $n\in \mf Z$ with $\tau_1=\gamma_1(\tau)$. From which we get
\begin{equation}
   \mbb B_{k,N}(\tau)= \sumn_{\gamma_1\in\Gamma_0(N)}(v  v_1)^{k/2}\sumn_{n\in\mf Z}  ( \tau_1-\overline{\tau}+n)^{-k}.
\end{equation}
First we write $w= \tau_1-\overline{\tau}$. Replacing $n$ by $n+M$ for some $M\in \mf Z$, we can assume that $|w_1|\le 1/2$, where we have put $w_1=Re(w)$.

Now consider the sum over $n$.
\begin{equation}
   \big| \sum_{n\in \mf Z} \frac{1}{(w+n)^k} \big|\le \frac{1}{|w|^k}+\sum_{n\neq 0}\frac{1}{|(w_1+n)+iw_2|^k}\le \frac{1}{w_2^k}+\frac{1}{w_2^k}\sum_{n\neq 0}|1+(\frac{w_1+n}{w_2} )^2|^{-k/2}\ll_k w_2^{-k}.
\end{equation}
The last inequality holds since $w_2=v+ v_1\le 2v \le 2$; the first inequality holds since $\tau \in \mc S$ and the second because we have $v \le 1$. Thus we get
\begin{equation} \label{cc-1}
    \mbb B_{k,N}(\tau)\ll \sum_{(c,d)=1, N|c}\frac{(v v_1)^{k/2}}{(v+ v_1)^k}= \sum_{(c,d)=1, N|c}\frac{1}{ (|c\tau+d| + |c\tau+d|^{-1})^k } \le \sum_{(c,d)=1, N|c}\frac{1}{|c\tau+d|^k}.
\end{equation}
Now we have
\begin{equation}
    \sum_{(c,d)=1, N|c}\frac{1}{|c\tau+d|^k}=\sum_{(c,d)=1}\frac{1}{|cN\tau+d|^k}=\sum_{(c,d)=1}\frac{1}{|c\tau'+d|^k}, \q (\tau' := N \tau).
\end{equation}
We note here that by replacing $(c,d)$ by $(c, d+cM)$ (which is a bijection) for some $M\in\mf Z$, we can assume that $|Re(\tau')| \le 1/2$. Also note that $\im (\tau')=N \im (\tau)\ge \sqrt{3}/2$.

When $c=0$, the RHS is bounded. Therefore, we consider the case $c\neq 0$. Further, dividing the range of $|c\tau'+d|$ into dyadic intervals, we see that
\begin{equation}
    \mbb B_{k,N}(\tau)\ll \sum_{|c\tau'+d|< 1}\frac{1}{|c\tau'+d|^k}+\sum_{j=0}^{\infty}\sum_{2^j\le |c\tau'+d| <2^{j+1}}\frac{1}{|c\tau'+d|^k}.
\end{equation}
When $|c\tau'+d|< 1$, we see that $|cv'|<1$ and $|c|<1/v'\le 2/\sqrt{3}$ so that $|c|=0,1$. The case $c=0$ was already considered above, so we assume that $|c|=1$, which then implies $|d| \le 3/2$ since $|Re(\tau')| \le 1/2$ and $|cu'+d|<1$. We now use the fact that $v'\ge \sqrt{3}/2$ we see that
$ \displaystyle \sum_{|c\tau'+d|< 1, c \neq 0}\frac{1}{|c\tau'+d|^k}\le 6 \cdot (2/\sqrt{3})^k$.

When $2^j\le |c\tau'+d| <2^{j+1}$, we see that $|c|\le 2^{j+1}v'^{-1}$ and $|d|\le |c|/2+2^{j+1}$. Thus using the counting argument, we get that the second sum is 
\begin{equation}
    \ll \sum_{j=0}^{\infty} (1+\frac{2^{j+1}}{v'}) (\frac{2^{j}}{v'} + 2^{j+1}) \frac{1}{2^{jk}}\ll \sum_{j=0}^{\infty} \frac{2^{2j+2}}{2^{jk}}\ll_k 1.
\end{equation}

\subsection{\texorpdfstring{$N$}{N} not necessarily square--free.}
    Let $\mbb B_{k,N}(\tau)$ be as defined in \eqref{BKn=1}. First, we show that it is enough to work with a subgroup conjugate to $\Gamma_0(N)$ in the region $v\ge \sqrt{3}/2N$. For this we make use of the following result from \cite{saha2017sup}.

\begin{lem}\cite{saha2017sup}\label{lemSaha}
    Given any $\tau\in \h$, we can find a divisor $M$ of $N$ with $M^2|N$ and an Atkin-Lehner operator $W(M,N)$ such that $\im (\sigma^{-1} W(M,N) \tau)\ge \frac{\sqrt{3}M^2}{2N}$ and $|c(\sigma^{-1} W(M,N) \tau)+d |^2 \ge \frac{3M^2(c, N/M^2)}{4N}$ for some $\sigma\in\sltwo$ with $C(\sigma)=N/M$.
\end{lem}
Let $M$ be a divisor of $N$ such that $M^2|N$ and $\sigma\in \Gamma_0(N)\backslash\sltwo/\sltwo_\infty$. Then define the set
\begin{equation}
   \mc S_M:=\{ \tau' \in\h :  \im ( \tau')\ge \frac{\sqrt{3}M^2}{2N} \text{ and } |c \tau'+d |^2 \ge \frac{3M^2(c, N/M^2)}{4N} \}.
\end{equation}
From \lemref{lemSaha} it is clear that
\begin{align}
   \mbb B_{k,N}(\tau) &:= \sumn_{f }\left( v^k |f|^2 \right) (W_M^{-1} \sigma \tau') = \sumn_{f }v'^k \left| (f|_k W_M^{-1} \sigma) (\tau')\right|^2 \\
   & = \sumn_{f}  v'^k \left| (f|_k  \sigma) (\tau')\right|^2 = \mbb B_{k,\sigma^{-1} \Gamma_0(N) \sigma}(\tau'),
\end{align}
where $\mbb B_{k,\sigma^{-1} \Gamma_0(N) \sigma}(\tau')$ denotes the BK for the conjugate group $\sigma^{-1} \Gamma_0(N) \sigma$.
Thus we conclude that 
\begin{equation}
    \sup_{\tau\in \h} \mbb B_{k,N}(\tau)\ll \max_{M, M^2|N} \max_{\sigma }\sup_{\tau' \in \mc S_{ M}} \mbb B_{k,\sigma^{-1} \Gamma_0(N) \sigma}(\tau').
\end{equation}
Now we bound $\mbb B_{k,\sigma^{-1} \Gamma_0(N) \sigma}(\tau)$ for $\tau\in \mc S_M$. Note the change from $\tau'$ to $\tau$. 

We show how to reduce to the considerations in the previous subsection.
We first restrict to those $\tau \in \mc S_M$ for which $v \ll 1$.
Let $h$ denote the width of the cusp at $\infty$ for the subgroup
\[ \Gamma_\sigma:=\sigma^{-1} \Gamma_0(N) \sigma.\]
In our case, we have $h= N/(N^2/M^2, N)$ (cf. \cite{saha2017sup}). Now, by the choice of $\sigma$ and $M$, we have $M^2|N$. Thus $h=1$ and $\smat{1}{1}{0}{1}\in\sigma^{-1} \Gamma_0(N) \sigma$. Therefore, we can write uniquely $\gamma = \gamma_\infty \gamma_1$ as in the previous section and proceed in the same manner. 
Put $A_\sigma = \sltwo_\infty \backslash \Gamma_\sigma$.

Now since $\Gamma_\sigma\subset \Gamma_0( N/M)$, we infer that the sum over $\gamma_1 \leftrightarrow A_\sigma$ is majorised by $\{ (c,d) \mid N/M |C, (c,d)=1\}$, by positivity. Then we note that \eqref{cc-1} holds for the present case and $|c \tau+d| = |c' \frac{N}{M} \tau +d| $ for some $c'$.
Further $\im(\frac{N}{M} \tau) = \frac{N}{M} v \ge \frac{\sqrt{3}M}{2}$, since $\tau \in \mc S_M$. This settles the present case.

We now handle the case $\tau \in \mc S_M$, $v \gg 1$.
From the theory of Poincar\'e series, we can write
\begin{equation}
    \mbb B_{k,\sigma^{-1} \Gamma_0(N) \sigma}(\tau)\ll v^k\left(\sumn_n p_{k,\sigma}(n)^{1/2} \exp(-2\pi n v)\right)^2,
\end{equation}
where $p_{k,\sigma}(n)=\frac{(4\pi n)^{k-1}}{h^k\Gamma(k-1)}P_{k,\sigma}(n)$. Here $h$ denotes the width of the cusp at $\infty$ and $P_{k,\sigma}(n)$ is the $n$-th Fourier coefficient of the $n$-th Poincaré series at $\infty$. In our case, we have $h= N/(N^2/M^2, N)$. Now, by the choice of $\sigma$ and $M$, we have $M^2|N$. Thus $h=1$ and this implies that (see \cite[3.19]{iwaniecaut})
\begin{equation}
    P_{k,\sigma}(n)=1 + 2\pi i^{-k}\sumn_{c>0} c^{-1} S(n,n,c) J_{k-1}(4\pi n/c),
\end{equation}
where $S(n,n,c)$ is the Kloosterman sum at the cusp defined by $\sigma$: 
  \[   S(n,n,c)=\sum_{\gamma=\smat{a}{*}{c}{d}\in \sltwo_\infty\backslash \Gamma_\sigma/\sltwo_\infty } e_c(an+dn).\]
Now from the previous discussions we know that $\smat{1}{1}{0}{1}\in\Gamma_\sigma$ and $\sigma^{-1} \Gamma_0(N) \sigma\subset \Gamma_0( N/M)$. Thus the coset representatives for $\sltwo_\infty\backslash \Gamma_\sigma/\sltwo_\infty$ is given by the set (see \cite[Proposition 2.7]{iwaniecaut})
\begin{equation}
\{\gamma=\smat{a}{*}{c}{d} \in \Gamma_\sigma : (c,d)=1, \frac{N}{M}|c,  a,d \bmod c; ad \equiv 1 \bmod c  \}.
\end{equation}
Now consider the sum over $c$. Using the bound $J_\ell (x) \le \min\{1, \frac{(x/2)^\ell}{\Gamma(\ell +1)}\}$ for the Bessel function and splitting the sum over $c$ into dyadic intervals, we note that the sum over $c$ is bounded by
\begin{align}
    \sum_{N/M\le c\le 2\pi n} c^{-1}| S(n,n,c)| +\sum_{j=0}^{\infty}\sum_{2^j\cdot 2\pi n< c\le 2^{j+1}\cdot 2\pi n} c^{-1}|S(n,n,c)| \frac{(2\pi n/c)^{k-1}}{\Gamma(k)}.
\end{align}
Now using the bound $\sum_{c\le X}c^{-1} { S(n,n,c)}\le MX/N$ (see \cite[(4.2)]{iwaniecaut}) we see that the sum over $c$ is (provided $k>2$ )
\begin{align}
    \ll \frac{Mn}{N}+\frac{Mn}{N\Gamma(k)}\sum_{j=0}^{\infty}\frac{1}{2^{j(k-1)-j-1}}\ll \frac{Mn}{N}.
\end{align}
This implies the following bound on $P_{k,\sigma}(n)$.
\begin{equation}
    P_{k,\sigma}(n)= 1+O(N^{-1}Mn).
\end{equation}
Thus, by the same analysis as in \eqref{BKNbound}, we see that when $v\ge 1$,
\begin{equation}\label{BKsigma}
    \mbb B_{k,\sigma^{-1} \Gamma_0(N) \sigma}(\tau)\ll_k \left(1+N^{-1}M\right)\ll _k 1.
\end{equation}
\subsection{The size of \texorpdfstring{$U(p)$}{U(p)} images}\label{Upcontn=1} 
As mentioned in the Introduction, we demonstrate here how the simple Hecke relations allow one to calculate the contribution of the size of the $U(p)$ images, in contrast with the degree $2$ case.

Let $N$ be any positive integer and $p$ be a prime co-prime to $N$. In this section, we obtain the size of $U(p)\left(S_k(N)\right)$. First, we note that $U(p)=T(p)-p^{k/2-1}B_p$, where $B_p=\smat{p}{0}{0}{1}$. Thus for any eigenform $f$ of level $1$, we have $U(p)f= \lambda_f(p) f- p^{k-1} f(pz)$. Next, we have the relation (see e.g., \cite{anamby2019distinguishing})
\begin{equation}
    \norm {U(p)f}_p^2 = \left(p^{k-2}+ \frac{(p-1)}{(p+1)}|\lambda_f(p)|^2 \right) \norm{f}_p^2.
\end{equation} Thus
\begin{align}
    \sum_{f} \frac{v^k|U(p)f|^2}{\norm {U(p)f}_p^2} &\ll \frac{\lambda_f(p)^2}{\left(p^{k-2}+ \frac{(p-1)}{(p+1)}|\lambda_f(p)^2 |\right)} \sum_f \frac{v^k |f|^2}{\norm{f}_p^2} + \frac{p^{2k-2}}{\left(p^{k-2}+ \frac{(p-1)}{(p+1)}|\lambda_f(p)|^2 \right)} \sum_f \frac{v^k |f(pz)|^2}{\norm{f}_p^2}\\
    &\ll \frac{1}{p+1} \sum_f \frac{v^k |f|^2}{\norm{f}_1^2} + \frac{p^{2k-2}}{p^{k-2}} \frac{1}{p+1} \cdot p^{-k} \sum_f \frac{v^k |f(z)|^2}{\norm{f}_1^2}
    \ll \frac{1}{p}.
\end{align}
This gives us that $\sup(U(p)\left(S_k(N)\right))\ll 1/p$.

\subsection{Handling congruence subgroups}
Let us denote by $i_\Gamma$ the index of a congruence subgroup $\Gamma$ in $\sltwo$ and $\mc B(\Gamma)$ denote an orthonormal basis for $S_k(\Gamma)$. Let 
\begin{equation}
    \mbb B_{k, \Gamma}(z) = \sumn_{f \in \mc B(\Gamma)} y^k |f(z)|^2.
\end{equation}
We suppose that $\Gamma= \Gamma_N$. For each $f$ as above, consider 
\begin{equation}
    g_f(z) = f |_k \alpha_N,
\end{equation}
where $f |_k \alpha_N = N^{k/2} f(Nz)$. It is easy to see that $g \in S_k(\Gamma_1(N^2))$ more precisely, $g \in S_k(\Gamma_0(N^2), \chi')$ where $\chi'$ is induced from a Dirichlet character $\chi \bmod N$. To see this, recall the decomposition of $S_k(\Gamma(N))$ by characters $\chi \bmod N$ and note that each eigenspace is mapped under $\alpha_N$ to $S_k(\Gamma_0(N^2), \chi')$ isomorphically.
Moreover, this operation preserves orthogonality.
Now
\begin{align}
    \sup_{z \in \mf H} \mbb B_{k, \Gamma}(z) = \sup_{z \in \mf H} \mbb B_{k, \Gamma}(Nz) = \sup_{z \in \mf H} \sumn_f y^k |g_f(z)|^2.
\end{align}
Notice that (see \cite[III, Theorem 4.1]{lang2012introduction}),
$\displaystyle  i_\Gamma\lan g_f , g_f \ran =  i_{\Gamma_0(N^2)}  \lan f,f \ran$.
Therefore,
\begin{align}
    \sup_{z \in H} \mbb B_{k, \Gamma}(z) \le \frac{ i_{\Gamma_0(N^2)} }{ i_\Gamma} \sup_{z \in \mf H} \sumn_\chi \mbb B_{\kappa ,\Gamma_0(4N), \chi}(z) \ll \frac{ i_{\Gamma_0(N^2)} }{ i_\Gamma} \cdot \phi(N) \ll 1.
\end{align}
In the above, we have tacitly used the fact that $\sup \mbb B_{\kappa ,\Gamma_0(4N), \chi} \ll 1$ -- whose proof can be easily deduced from our calculations with $\chi=1$ verbatim, and is left to the reader. \qed

\section{Appendix~2: The Jacobi  \texorpdfstring{$U_J(p)$}{n=1} operator }\label{appendix2}

The purpose of this section, as suggested by the referee, is to clarify why the result in \cite[Proposition 4.6]{Ag-Br} can not be true for arbitrary cusp forms. 
From \cite[Proposition 4.6]{Ag-Br}, on $J_{k,1}^{cusp}(M)$ we have
\begin{align} \label{vm*}
    V_m^* V_m = \sum_{d|m, (d,M)=1} \rho(d) d^{k-2} T_J(m/d),
\end{align}
where $\rho(d)=d\prod_{p|d}(1+1/p)$. Taking $M=m=p$, we see that
\begin{align}
 V_p^*V_p= U_J(p).   
\end{align}
If the relation in \eqref{vm*} holds on the entire space $J_{k,1}^{cusp}(M)$, then this would imply that $U_J(p)$ is self-adjoint on $J_{k,1}^{cusp}(M)$. This is of course true on $J_{k,1}^{cusp, new}(M)$. The goal of this section is to show that $U_J(p)$, however, is not self-adjoint on the full space $J_{k,1}^{cusp}(p)$. This is well-known in the elliptic integral-weights case\footnote{see e.g., \href{https://mathoverflow.net/questions/46577/adjoint-of-atkin-lehners-u-p}{here} for a discussion}, and we reduce it to a similar statement for modular forms of half-integral weights in Kohnen's plus-space.

\begin{thm} \label{ujp-notselfadj}
There exist infinitely many odd primes $p$ such that   $U_J(p)$ is not self-adjoint on the full space $J_{k,1}^{cusp}(p)$.
\end{thm}

The proof given below holds verbatim for the Kohnen's plus space at level $4p$ and thus for the full space of half-integral weight forms of level $4p$ for suitable primes $p$. It should also hold for all odd, square-free levels. Moreover, we hope that main result, and the calculations given below may have independent interest as well.

\subsection{Proof of Theorem~\ref{ujp-notselfadj}}
Let $p$ be an odd prime and $f\in S_{2k-2}$ be an eigenfunction of $T(p)$ with eigenvalue $\lambda_p$. Further, we assume that the roots of the polynomial $X^2-\lambda_p X+ p^{2k-3}$, say $\alpha$ and $\beta$, are distinct. Thus $\lambda_p= \alpha+\beta$ and $\alpha\beta=p^{2k-3}$, and by Deligne, $\beta=\overline{\alpha}$. Note that $\alpha \neq \beta$ is equivalent to the eigen-angle $\theta_p \in (-\pi, \pi]$ defined by $\alpha/p^{k-3/2}=e^{i \theta_p}$ has to satisfy $\theta_p \neq \{0, \pi\}$. By the Sato-Tate theorem (or even by much weaker versions of it), we certainly have an infinite set of such primes $p$, from which we fix one which is odd. 

Next, let $\phi$ be the image of $f$ under the Hecke equivariant isomorphism $S_{2k-2}\cong J_{k,1}^{cusp}$. Then $\phi$ is an eigenfunction of the Jacobi $T_J(p)$ operator with the eigenvalue $\lambda_p$. Let 
\[ \phi_\alpha= \phi -p^{1-k}\alpha \phi|w_p,  \text{ and } \phi_\beta= \phi -p^{1-k} \beta \phi|w_p ,\]
where $w_p$ is the involution-operator defined as in \cite{manickam1993saito}. Then from \cite[Proposition~6]{manickam1993saito}\footnote{The correct power viz. $p^{k-2}$ appears in \cite{manickam1993saito}.}, on $J_{k,1}^{cusp}$, we have the following relations.
\begin{align}
U_J(p) +p^{k-2} w_p &= T_J(p).
\end{align}
Using the Fourier expansion of $\phi|T_J(p)$ from \cite[Theorem 4.5]{EZ}, we get
\begin{align}
    \phi|w_p= \sum_{D, r}\left(\left(\frac{D}{p}\right)c_\phi(D, r)+p^{k-1}c_\phi(\frac{D}{p^2}, \frac{r}{p}) \right)e\left(\frac{(r^2-D)\tau}{4}+rz\right).
\end{align}
Thus, we get (see also \cite[(47)]{MR})
\begin{align}
    \phi|w_p|U_J(p)=p^{k-1}\phi.
\end{align}

Using the above relations, we can compute:
\begin{align}
    \phi_\alpha|U_J(p) &= \phi|U_J(p) - p^{2-k}\alpha \phi|w_p|U_J(p)=\phi| T_J(p) -p^{k-2} \phi|w_p -\alpha \phi\\
    &= (\lambda_p - \alpha) \phi - p^{k-2}\phi|w_p= \beta (\phi - p^{1-k} \alpha \phi|w_p)= \beta\phi_\alpha.
\end{align}
Similarly, we get $\phi_\beta|U_J(p)=\alpha\phi_\beta$. 

Now suppose that $U_J(p)$ is self-adjoint on the full space. We see that
\begin{align}
    \beta\lan \phi_\alpha, \phi_\beta\ran_p=\lan \phi_\alpha|U_J(p), \phi_\beta\ran_p= \lan \phi_\alpha, \phi_\beta|U_J(p)\ran_p= \alpha \lan \phi_\alpha, \phi_\beta\ran_p.
\end{align}
Since $\alpha \neq \beta$, this implies that $\lan \phi_\alpha, \phi_\beta\ran_p =0$. However, this is not possible, as the next Lemma shows. This completes the proof of Theorem~\ref{ujp-notselfadj}.

\begin{lem}\label{lem:phialphabeta}
   With the above notation, $\lan \phi_\alpha, \phi_\beta\ran_p \neq 0$.
\end{lem}

To prove this, we need another Lemma.
If we denote the Eichler-Zagier map on $J_{k,1}(p)$ as $\mc Z_1^p$, then from \cite[Proposition 7]{manickam1993saito} we see that 
\begin{equation} \label{ez-up}
  \mc Z_1^p U(p^2)= U_J(p)\mc Z_1^p .
\end{equation}
 Since $\mc Z_1^p$ preserve Petersson scalar product up to a constant (see \cite{Ag-Br}), setting $\mc Z_1^p\phi=h$ and using \eqref{ez-up}, we get
\begin{align} \label{ujp-up2}
\lan\phi|U_J(p), \phi\ran_p = \frac{[\Gamma_0(p):\Gamma_0(4p)]}{2^{2k-1/2}}\lan h|U(p^2), h\ran_p.    
\end{align}
The inner product on the RHS can be evaluated as follows.
\begin{lem}\label{halfnorm}
    Let $h\in S_{k-1/2}^+(4)$ be an eigenfunction of the Hecke operator $T(p^2)$ with the eigenvalues $\lambda_h(p)$. Then for any $g\in S_{k-1/2}^+(4)$
    \begin{equation}\label{hUp2norm}
    \langle h|U(p^2), g \rangle_{4p} = \frac{p}{p+1}\lambda_h(p) \langle  h, g\rangle_{4p}.
\end{equation}
\end{lem}

\begin{proof}
Consider the expression for the operator $T(p^2)$ on $S_{k-1/2}^+(4)$ given by 
\begin{equation}
    p^{-k+5/2}T(p^2)h= \sum_{b=0}^{p^2-1}h|\smat{1}{b}{0}{p^2}+\sum_{d=1}^{p-1} h|\smat{p}{d}{0}{p}+ h|\smat{p^2}{0}{0}{1}.
\end{equation}
Denote by $B(p^2)$, the operator corresponding to the matrix $\smat{p^2}{0}{0}{1}$. Now using the decomposition as in \cite[pp. 451]{shimura}, we get
\begin{equation}
    \begin{split}
      p^{-k+5/2}\lan h|T(p^2) , g\ran_{4p^2} = p^2\lan h|\smat{1}{0}{0}{p^2}, g\ran_{4p^2}+ (p-1)\lan h|\smat{1}{0}{0}{p^2}, g\ran_{4p^2} + \lan h|B(p^2), g\ran_{4p^2}.
    \end{split}
\end{equation}
Now choose $b,c,d \in \mbb Z$ such that $dp^2-4bc=1$. Then we see that $\smat{p^2}{b}{4c}{d}\smat{1}{0}{0}{p^2}\smat{dp^2}{-b}{-4c}{1}=\smat{p^2}{0}{0}{1}$. Thus we get that
\begin{equation}
    \lan h|B(p^2), g\ran_{4p^2}= \frac{p^{-k+5/2}\lambda_h(p)}{p^2+p}\lan h , g\ran_{4p^2}.
\end{equation}
Now 
\begin{equation}
    \lan h|U(p^2) , g\ran_{4p^2}= p^{k-5/2}\lan\sum_{b=0}^{p^2-1}h|\smat{1}{b}{0}{p^2}, g \ran_{4p^2}.
\end{equation}
Imitating the same steps as above, we get
\begin{equation} \label{up2}
    \begin{split}
        \lan h|U(p^2) , g\ran_{4p^2}=  p^{k-5/2}\cdot p^2 \lan h|B(p^2), g\ran_{4p^2}= \frac{p\lambda_h(p)}{(p+1)}\lan h, g\ran_{4p^2}.
    \end{split}
\end{equation}
Now converting the level $4p^2$ inner products to level $4p$ inner products on both sides, we get \eqref{hUp2norm}.\qedhere
\end{proof}

\begin{proof}[Proof of Lemma \ref{lem:phialphabeta}]
We first note that
\begin{align}\label{phialphabetaIP}
    \lan \phi_\alpha, \phi_\beta\ran_p&=\lan \phi,\phi\ran_p - p^{1-k}\alpha\lan \phi|w_p, \phi\ran_p-p^{1-k}\overline{\beta}\lan \phi, \phi|w_p\ran_p+ p^{2-2k}\alpha\overline{\beta}\lan \phi, \phi\ran_p.
\end{align}
Now we find an expression for $\lan \phi|w_p, \phi\ran_p$. Since $w_p= p^{2-k}(T_J(p)-U_J(p))$, and $\phi$ is an eigenfunction of $T_J(p)$,  
\begin{align}
\lan \phi|w_p, \phi\ran_p= p^{2-k}(\lambda_p\lan \phi, \phi\ran_p -\lan \phi|U_J(p), \phi\ran_p).
\end{align}
Now, combining \eqref{ujp-up2} and \eqref{up2}, we get,
\begin{align}
    \lan \phi|U_J(p), \phi\ran_p= \frac{p}{p+1}\lambda_p \lan\phi, \phi\ran_p.
\end{align}
As a consequence, we get
\begin{align} \label{phi-wp}
    \lan \phi|w_p, \phi\ran_p= \frac{p^{2-k}}{p+1}\lambda_p \lan \phi, \phi\ran_p.
\end{align}
This expression also implies that $\lan \phi|w_p, \phi\ran_p= \lan \phi, \phi|w_p\ran_p$, as they are real valued.

So if the inner product $ \lan \phi_\alpha, \phi_\beta\ran_p=0$ in \eqref{phialphabetaIP}, since $\overline{\beta}=\alpha$, $\alpha$ must satisfy the polynomial
\begin{equation}
    p^{2k-2}- 2\frac{p\lambda_p}{p+1} X+ X^2,
\end{equation}
since $\lan \phi|w_p, \phi\ran_p= \lan \phi, \phi|w_p\ran_p$ are given by \eqref{phi-wp}.

But $\alpha$ also satisfies the polynomial $X^2-\lambda_p X+ p^{2k-3}$. Thus, we get that
\begin{align}
    &\alpha\lambda_p\left(\frac{2p}{p+1}-1\right)+p^{2k-3}(1-p)=0\implies \alpha^2+p^{2k-3}= p^{2k-2}+p^{2k-3}.\\
\end{align}
That is, $|\alpha|^2= p^{2k-2}$, which is not possible since $|\alpha|^2=\alpha\beta=p^{2k-3}$ and this gives us that $\lan \phi_\alpha, \phi_\beta\ran_p\neq 0$, as desired.
\end{proof}
\printbibliography
\end{document}